\documentclass[sn-mathphys,Numbered]{sn-jnl}

\usepackage{graphicx}%
\usepackage{multirow}%
\usepackage{amsmath,amssymb,amsfonts}%
\usepackage{amsthm}%
\usepackage{mathrsfs}%
\usepackage[title]{appendix}%
\usepackage{xcolor}%
\usepackage{textcomp}%
\usepackage{manyfoot}%
\usepackage{booktabs}%
\usepackage{algorithm}%
\usepackage{algorithmicx}%
\usepackage{algpseudocode}%
\usepackage{listings}%
\usepackage{todonotes}
\usepackage{enumitem}
\usepackage{multirow}
\usepackage{subcaption}
\usepackage{diagbox}
\usepackage{hhline}
\usepackage{float}
\usepackage{xurl}
\usepackage{lmodern}

\newcommand{\RR}{\mathbb{R}}
\newcommand{\bdry}{\mathrm{bdry\ }}

\newcommand{\epi}{\text{epi }}
\newcommand{\hypo}{\text{hypo }}
\newcommand{\argmin}{\operatornamewithlimits{argmin}}

\newtheorem{theorem}{Theorem}
\newtheorem{proposition}[theorem]{Proposition}%
\newtheorem{lemma}[theorem]{Lemma}%
\newtheorem{corollary}[theorem]{Corollary}%
\newtheorem{claim}[theorem]{Claim}%

\definecolor{blue}{gray}{0.0}

\begin{document}

\title[Gauges and Accelerated Optimization]{Gauges and Accelerated Optimization over\\ Smooth and/or Strongly Convex Sets}


\author[1]{\fnm{Ning} \sur{Liu}}\email{nliu15@jhu.edu}

\author[1]{\fnm{Benjamin} \sur{Grimmer}}\email{grimmer@jhu.edu}

\affil[1]{\orgdiv{Department of Applied Mathematics and Statistics}, \orgname{Johns Hopkins University}, \orgaddress{\city{Baltimore}, \postcode{21218}, \state{MD}, \country{USA}}}


\abstract{We consider feasibility and constrained optimization problems defined over smooth and/or strongly convex sets. These notions mirror their popular function counterparts but are much less explored in the first-order optimization literature. We propose new scalable, projection-free, accelerated first-order methods in these settings. Our methods avoid linear optimization or projection oracles, only using cheap one-dimensional linesearches and normal vector computations. Despite this, we derive optimal accelerated convergence guarantees of $O(1/T)$ for strongly convex problems, $O(1/T^2)$ for smooth problems, and accelerated linear convergence given both. Our algorithms and analysis are based on novel characterizations of the Minkowski gauge of smooth and/or strongly convex sets, which may be of independent interest: although the gauge is neither smooth nor strongly convex, we show the gauge squared inherits any structure present in the set.}

\keywords{Projection-Free, First-Order Optimization, Smoothness, Strong Convexity, Minkowski Gauges, Accelerated Convergence Rates}



\maketitle

    \section{Introduction}

We consider feasibility and optimization problems defined over sets $S_i$ possessing classic structures like smoothness and strong convexity. These structures in sets are much less explored than their function counterparts in the first-order optimization literature. We show these structures in constraint sets lead to the same speedups commonly found in first-order methods for structured functions (i.e., $O(1/T)$ convergence rates given strong convexity, accelerated $O(1/T^2)$ rates given smoothness, and fast linear convergence given both). We propose projection-free algorithms attaining these optimal rates for both feasibility problems
\begin{equation} \label{eq:Feas-Problem}
    \mathrm{Find\ } x\in\cap_{i=1}^m S_i
\end{equation}
where all the sets $S_i$ are either smooth, strongly convex, or both, and for optimization problems with similarly structured $f$ and $S_i$
\begin{equation} \label{eq:OPT-Problem}
    \begin{cases}
        \max& f(x)\\
        \mathrm{s.t.} & x\in S_i \quad \forall i=1\dots m \ .
    \end{cases}
\end{equation}

Critically, our proposed algorithms are projection-free. The typical first-order method assumption of being able to project onto $S_i$ limits algorithms to simple constraints. Instead, we only assume oracles for one-dimensional linesearches (to evaluate the gauges defined in~\eqref{eq:gauge}) and to compute normal vectors on the boundary of each $S_i$. In contrast, projected gradient methods require an orthogonal projection oracle. (Computing projections onto sets is comparable to computing proximal operators of functions, which is often costly, whereas computing normal vectors is comparable to cheaper gradient calculations.) For example, if $S_i$ is an ellipsoid, orthogonal projection lacks a closed form, but our linesearch and its normal vectors have closed forms, costing a single matrix-vector multiplication to compute. For polyhedrons, projection is a quadratic program, whereas a normal vector is computed by identifying any one active constraint.

Throughout, we consider closed convex sets $S_i\subseteq \mathcal{E}$ in some finite-dimensional Euclidean space $\mathcal{E}$. Smoothness and strong convexity can be intuitively defined as follows (formal local and global definitions are given in Section~\ref{sec:prelim}): A set $S$ is $\beta$-smooth if every unit normal vector $\zeta\in N_S(x)$ taken at some $x$ on the boundary of $S$ yields an inner approximation
\begin{equation} \label{eq:inner_approx}
B(x - \zeta/\beta, 1/\beta) \subseteq S \ , 
\end{equation}
where $B(x,r)$ is the ball of radius $r$ around $x$. This can be viewed as a ball smoothly rolling around inside the set's boundary.
Likewise, $S$ is $\alpha$-strongly convex if every unit normal vector $\zeta\in N_S(x)$ yields an outer approximation
\begin{equation} \label{eq:outer_approx}
B(x - \zeta/\alpha, 1/\alpha) \supseteq S \ . 
\end{equation}
These definitions in terms of inner and outer approximations given by each normal vector mirror the traditional functional setting, where smoothness and strong convexity correspond to upper and lower quadratic approximations being given by each subgradient.

\noindent {\bf Our Contributions.}
Our approach towards utilizing these structures focuses on each set's gauge, translated by some point $e_i$ in the interior of $S_i$ (taking $e_i=0$ gives the classic Minkowski functional):
\begin{equation} \label{eq:gauge}
    \gamma_{S_i,e_i}(x) := \inf \{\lambda>0 \mid x-e_i \in \lambda (S_i-e_i)\}\ .
\end{equation}
\begin{enumerate}
    \item {\bf Structure of Gauges of Structured Sets.} We show in Theorems~\ref{thm:gauge-strong-convexity} and~\ref{thm:gauge-smoothness} that any $\beta$-smooth or $\alpha$-strongly convex set has $O(\beta)$-smooth or $O(\alpha)$-strongly convex gauge squared $\frac{1}{2}\gamma^2_{S,e}(x)$, respectively. Theorems~\ref{thm:gauge-strong-convexity-converse} and~\ref{thm:gauge-smoothness-converse} show the converse of this result. Here the big-O notation suppresses constants depending on the size of $S$ and placement of $e\in S$.
    \item {\bf Fast Algorithms for Reformulations of~\eqref{eq:Feas-Problem} and~\eqref{eq:OPT-Problem}.} Noting $\gamma_{S,e}(x)\leq 1$ if and only if $x\in S$, the feasibility problem~\eqref{eq:Feas-Problem} can be rewritten as the unconstrained convex minimization
    \begin{equation}\label{eq:Feas-Gauge-Problem}
        \min_y \max_i\{\gamma_{S_i,e_i}(y)\}
    \end{equation}
    for any $e_i\in S_i$. Utilizing the radial duality framework of~\cite{grimmer2023radial1, grimmer2023radial2}, the constraint optimization problem~\eqref{eq:OPT-Problem} can be rewritten as the unconstrained convex minimization
    \begin{equation}\label{eq:OPT-Gauge-Problem}
        \min_y \max_i\{f^\Gamma(y),\  \gamma_{S_i,0}(y)\}
    \end{equation}
    provided $0\in \mathrm{int\ } \bigcap_i S_i$ and $f(0)>0$, where $f^\Gamma(y)=\sup \{v>0 \mid vf(y/v) \leq 1 \} $. Applying well-known first-order methods to these unconstrained gauge reformulations, our Theorems~\ref{thm:feas-rates} and~\ref{thm:OPT-rates} show a $O(1/\alpha\epsilon)$ rate for $\alpha$-strongly convex $f$ and $S_i$, an accelerated $O(\sqrt{\beta/\epsilon})$ rate for $\beta$-smooth $f$ and $S_i$ and a $O(\sqrt{\beta/\alpha}\log(1/\epsilon))$ rate given both. These rates match the optimal projected method convergence rates, avoiding orthogonal projections but requiring structure on each $S_i$.
    \item {\bf Numerical Validation.} We verify our theory numerically, seeing the expected speedups for feasibility problems over $p$-norm ellipsoids, which are known to be smooth when $2\leq p < \infty$ and strongly convex when $1<p\leq 2$. Further, for synthetic smooth ellipsoid-constrained optimization instances, we compare our projection-free accelerated method (based on the radial reformulation~\eqref{eq:OPT-Gauge-Problem}) against alternative simple first-order methods, standard first-order-based solvers (SCS and COSMO) and second-order, interior point solvers (Gurobi and Mosek). Despite only using one matrix multiplication per iteration, our method is often competitive with second-order methods. This motivates future more practical implementations and real-world validation for accelerated radial methods.
\end{enumerate}

\noindent {\bf Outline.} In the remainder of this section, we discuss related approaches in the literature. Section~\ref{sec:prelim} formally introduces smoothness and strong convexity of functions and sets and provides several examples. Section~\ref{sec:main_result} formalizes and proves our main theorems on gauges of structured sets. Finally, in Sections~\ref{sec:algs} and~\ref{sec:appli}, we propose and analyze projection-free methods using our Theorems~\ref{thm:gauge-strong-convexity} and~\ref{thm:gauge-smoothness} for feasibility problems based on a gauge reformulation and constrained optimization based on a radial reformulation, and numerically explore their effectiveness.

\subsection{Related Works}\label{subsec:relate_work}

\noindent {\bf Structural Results on Gauges.} 
There has been much interest in the influence of curvature on computational and statistical efficiency in optimization and machine learning~\cite{garber2015faster, abernethy2018faster, demianov1970approximate, levitin1966constrained, dunn1979rates}, often in relation to Frank-Wolfe-type methods. These works primarily considered the relationship between the strong convexity of a centrally symmetric set and its gauge (squared). Note that this symmetry condition ensures a set's gauge is a norm. Recently \cite{molinaro2020curvature} showed a set is strongly convex w.r.t.~its own gauge norm if and only if the set's gauge squared is a strongly convex function. Our results remove any symmetry requirements or measurements with respect to the set itself. Without assuming symmetry, there is no gauge norm to consider, so our results describe smoothness and strong convexity with respect to the Euclidean norm. To the best of our knowledge, our results on the smoothness of a gauge squared are entirely new.

\noindent {\bf Optimization over Gauges - Gauge and Radial Duality.}
A specialized duality theory between optimization problems with objective function and/or constraints defined by gauges was developed by Freund~\cite{Freund1987} and further advanced by~\cite{Friedlander2014,Aravkin2018}. Our results may offer new insights, showing that these gauge problem formulations (squared) may be more tractable to solve when the sets inducing these gauges are smooth and/or strongly convex.

In \cite{grimmer2023radial1} and \cite{grimmer2023radial2}, Grimmer established a radial duality theory between nonnegative optimization problems. This theory shows that constrained optimization can be reformulated into an unconstrained problem in terms of the gauge of its constraints. Namely, for any concave $f$ with $f(0)>0$ and convex $S$ with $0\in \mathrm{int\ }S$, the primal problem
\begin{equation}
    \begin{cases} \max_x f(x)\\ \text{s.t. } x\in S \end{cases} = \max_{x\in\mathcal{E}} \min\{f(x), \hat \iota_S(x)\}
\end{equation}
where $\hat \iota_S(x)$ is the nonstandard indicator function $\hat \iota_S(x)=\begin{cases}+\infty & \text{ if } x\in S\\ 0 & \text{ if } x \notin S \end{cases}$ can be equivalently solved by minimizing the radially dual problem
\begin{equation}
    \min_{y\in\mathcal{E}} \max\{f^\Gamma(y),\gamma_S(y)\}
\end{equation}
where $f^\Gamma(y)=\sup\{ v>0 \mid v f(y/v) \leq 1\}$~\cite[Proposition 24]{grimmer2023radial1}. Note one can have $f(0)>0$ and hence a positive maximum value without loss of generality by translating the problem to place a known feasible point $x_0\in \mathrm{int \ } \left(S \cap \{x \mid f(x) \in \mathbb{R}\}\right)$ to the origin and adding a constant. This reformulation is unconstrained and only depends on the constraints $S$ via their gauge. Consequently, our structural results on gauges will facilitate the direct use of accelerated optimization methods on this radial dual (squared). Prior radial methods have only used simple subgradient methods~\cite{Renegar2016,Grimmer2017} or smoothing techniques~\cite{Renegar2019,grimmer2023radial2}.

\noindent {\bf Optimization over Gauges - Penalization and Level-set Reformulations.}
Given $e \in S$, one can rewrite the set constraint $x\in S$ as the functional constraint $\gamma_{S,e}(x)\leq 1$. Then a constrained optimization problem~\eqref{eq:OPT-Problem} can be reformulated as
\begin{equation}
    \begin{cases} \min_x f(x)\\ \text{s.t. } \gamma_{S,e}(x)\leq 1 \ . \end{cases} 
\end{equation}
The recent works~\cite{Zakaria2022,Lu2023} approached this via methods of Lagrange multipliers, minimizing $f(x) + \lambda ( \gamma_{S,e}(x) -1)$ in online optimization settings. Their convergence guarantees carefully account for the cost of approximate gauge evaluations via membership oracle calls to $S$. In contrast, we assume exact evaluations of the gauge (as is possible for many polyhedral, spectral, and polynomial-type constraints), which facilitates access to actual boundary points of $S$ where normal vectors exist.

Alternatively, a function $f$ can be minimized over a set $S$ via the value function
\begin{equation}
    v(\lambda) = \min_x \max \{f(x)-\lambda, \gamma_{S,e}(x)-1\} \ .
\end{equation}
This has $v(\lambda)=0$ exactly when $\lambda$ equals the minimum objective value. Root finding schemes can be effectively applied here~\cite[Lemmas 2.3.4, 2.3.5, 2.3.6]{nesterov2018lectures} as well as more careful approaches, always ensuring a feasible solution path~\cite{lin2018}. An optimal, parameter-free level-set method was proposed based on a simple parallel restarting scheme by~\cite{lin2020parameter}.

Another root-finding reformulation, proposed by~\cite{aravkin2019level}, reverses the roles of constraint and objective functions. Then, a convex function $f$ can be minimized over a set $S$ via root finding on the value function
\begin{equation}
    v(\lambda) = \begin{cases} \min_x \gamma_{S,e}(x)\\ \text{s.t. } f(x)\leq \lambda \ . \end{cases} 
\end{equation}
Namely, the minimum objective value of $f$ on $S$ is the smallest $\lambda$ such that $v(\lambda)\leq 1$. For any of these penalized/level-set/root-finding optimization techniques, our results on the structure of gauges (squared) for smooth and/or strongly convex sets $S$ may yield faster algorithms. The details for such an application are beyond this work but provide an interesting future direction. A general limitation of any gauge-based algorithm or reformulation is its reliance on a strictly feasible point $e$, which may require nontrivial preprocessing. 

\noindent {\bf Functionally Constrained Optimization.}
Functionally constrained optimization is one natural family of constrained convex optimization problems with smooth and/or strongly convex constraint sets. These are problems minimizing $f(x)$ subject to $g_i(x)\leq 0$ for convex $f,g_i$. Lemma~\ref{lemma:level} shows smoothness and/or strong convexity of $g_i$ carry over to the level sets $S_i = \{x \mid g_i(x)\leq 0\}$ (assuming gradients are well behaved on the boundary of the feasible region, which constraint qualification can ensure).
Functionally constrained problems can be addressed using switching subgradient, augmented Lagrangian, and other first-order methods~\cite{Metel2021}. One drawback of such methods directly defined in terms of function and gradient evaluations of $g_i$ is that they are representation dependent (that is, replacing $g_i(x)\leq 0$ by $2g_i(x)\leq 0$ may change how the algorithm behaves). Hence, one may need to preprocess constraints, rescaling them appropriately, to achieve good performance. Approaches based on evaluating the gauges avoid this issue as $\{x \mid g(x)\leq 0\}$ is the same set as $\{x \mid \lambda g(x)\leq 0\}$ for any rescaling $\lambda>0$. Note that the radial transform $f^{\Gamma}$ is not invariant to rescaling, instead, being given by $(\lambda f)^{\Gamma} (y)=f^{\Gamma}(\lambda y)/\lambda$. However, its Lipschitz constant is invariant to rescalings: By~\cite[Proposition 1]{grimmer2023radial2} $f^\Gamma$ is $1/R$-Lipschitz where $R= \inf\{\|x\| \mid f(x)\leq 0\}$. Since this constant $R$ is invariant to rescalings, $(\lambda f)^{\Gamma}$ is also $1/R$-Lipschitz.

\noindent {\bf Projected and Stochastic Gradient-Type Methods.}
When constraints are sufficiently simple, projected gradient methods can be utilized. Given a smooth objective $f$, accelerated projected gradient methods can achieve a convergence rate of $O(1/\sqrt{\epsilon})$. To match this rate, our proposed radial methods additionally need $S$ to be smooth and bounded. Consequently, when projections are computationally inexpensive, projected gradient methods may be more attractive than radial methods. To the best of our knowledge, no performance improvements have been shown for projected gradient methods over smooth/strongly convex sets.

In settings where only stochastic access to the objective function and its gradients are available, projected methods have historically found success. Their analysis typically generalizes to handle stochasticity. In contrast, radial methods rely on evaluating a one-dimensional linesearch to compute $f^\Gamma$. Given only stochastic access to $f$, it is unclear how to produce an unbiased estimate of this root in general, let alone unbiased estimates of $f^\Gamma$'s gradients.

\noindent {\bf Conditional Gradient-Type Methods.}
Here, we primarily consider sets $S$ where first-order access to the gauge $\gamma_S$ is given. A complementary projection-free model is to consider sets $S$ where access to the support function $\sigma_S(w) = \sup\{w^Tx \mid x\in S\}$ is given. That is, assuming linear optimization of $S$ is tractable. In this case, conditional gradient methods (e.g., Frank-Wolfe~\cite{frank1956algorithm}) can be applied. 
The recent works~\cite{pena2023affine,kerdreux2021projection} have shown performance improvements for conditional gradient methods whenever the constraint set is strongly convex, or more generally, uniformly convex. 

Support functions are dual to gauges in that for any closed convex $S$ with $0\in S$, $\sigma_{S^\circ}^2 = \gamma_S^2$ where $S^\circ$ is the polar of $S$. Despite this duality, often one of these models is more amenable to computation.
For example, consider polyhedral constraints $S$ either represented by a collection of inequalities of $\{x \mid a_i^Tx\leq b_i\}$ or represented as the convex hull of its extreme points $co\{x_i\}$. In either case, projection onto this region is a quadratic program, which may require an interior point method call. Linear optimization over this region may also require such a call when represented by its inequalities but can be done in linear time when represented by its extreme points. Hence for polyhedrons with relatively few extreme points, Frank-Wolfe can be effective.
In contrast, a gauge-based approach is ideal for polyhedrons represented by relatively few inequalities. The gauge can be computed in linear time via the closed form of $\max\{a_i^Ty/b_i\}$ with normal vectors given by selecting any active constraint.

    \section{Preliminaries on Smoothness and Strong Convexity}\label{sec:prelim}
We begin by defining smoothness and strong convexity of functions, ubiquitous notions in the analysis of first-order methods. Then, we define their mirrored notions that apply to sets, which have received far less attention in the literature.

Let $S\subseteq \mathcal{E}$ be a nonempty closed convex set, $h: \mathcal{E} \rightarrow \mathbb{R} \cup \{+\infty\}$ be a closed convex function, $B(x,r)\subseteq \mathcal{E}$ be a closed ball with center $x$ and radius $r$. We denote the extended real value by $\mathbb{\overline R} = \mathbb{R} \cup \{+\infty, -\infty\}$, the boundary of $S$ by $\bdry S$, the interior of $S$ by $\mathrm{int\ }S$, and the effective domain of $h$ by $\mathrm{dom\ }h$. We denote the subdifferential of $h$ at $x\in \mathrm{dom\ }f$ by $\partial h(x) :=\{g \in \mathcal{E} \mid h(y)-h(x) \geq g^T(y-x), \forall y \in \mathcal{E} \}$ and refer to each element as a subgradient. We denote the normal cone of $S$ at $x\in S$ by $N_S(x):= \{\zeta \mid \zeta^T(y-x) \leq 0, \forall y \in S\}$ and refer to each element as a normal vector. Throughout, any norm $\|\cdot\|$ without a subscript always denotes the Euclidean norm.
For functions $f_1, f_2: \mathcal{E} \rightarrow \mathbb{\overline R}$, the {\it epi-addition} (or {\it inf-convolution}) is
$$ (f_1 \ \square\ f_2)(x)=\inf_{x_1+x_2=x} \{f_1(x_1)+f_2(x_2)\} \ .
$$

\noindent {\bf Smooth and Strongly Convex Functions.}
We say a convex differentiable function $h$ is {\it (globally) $L$-smooth} if every point $x\in\mathrm{dom\ }h$ gives a quadratic upper bound
$$ h(y) \leq h(x) + \nabla h(x)^T(y-x) + \frac{L}{2}\|y-x\|^2 \quad \forall y\in\mathcal{E} \ . $$

Smoothness can be equivalently understood by any of the following characterizations. For completeness, proof/references are given in Appendix~\ref{proof:smooth_func} as property (d) below is not standard.
\begin{proposition} \label{prop:smooth_func}
    For a closed convex differentiable function $h$ and a constant $L\in (0,\infty)$, the following are equivalent:
    \begin{enumerate}[label=(\alph*)]
    \item The function $h$ is globally $L$-smooth.
    \item For any $x_i \in \mathcal{E}$, $ (\nabla h(x_1)-\nabla h(x_2))^T(x_1-x_2)\geq \frac{1}{L}\|\nabla h(x_1)-\nabla h(x_2)\|^2 $.
    \item There exists a closed convex function $h_0$ such that $(h_0\  \square\ \frac{L}{2}\|\cdot\|^2 )= h$.
    \item For any $x\in\mathcal{E}$ and $\tilde L > L$, all $y\in B(x,\eta)$ for $\eta>0$ small enough have
    $$ h(y) \leq h(x) + \nabla h(x)^T(y-x) + \frac{\tilde L}{2}\|y-x\|^2 \ . $$
    \end{enumerate}
\end{proposition}
\noindent We refer to this last characterization as $h$ being {\it (locally) $L$-smooth w.r.t.~$(x,\nabla h(x))$}.

Similarly, we say a closed convex $h$ is {\it (globally) $\mu$-strongly convex} if every point $x\in\mathrm{dom\ }h$ and subgradient $g\in\partial h(x)$ gives a quadratic lower bound
$$ h(y) \geq h(x) + g^T(y-x) + \frac{\mu}{2}\|y-x\|^2 \quad \forall y\in\mathcal{E} \ . $$
Mirroring those of smoothness, standard equivalent characterizations are given below. Again, for completeness, proof/references are given in Appendix~\ref{proof:sc_func}.
\begin{proposition} \label{prop:sc_func}
    For a closed convex function $h$ and a constant $\mu\in (0,\infty)$, the following are equivalent:
    \begin{enumerate}[label=(\alph*)]
    \item The function $h$ is globally $\mu$-strongly convex.
    \item For any $x_i \in \mathcal{E}$ and $g_i \in \partial h(x_i)$, $(g_1 - g_2)^T(x_1-x_2) \geq \mu\|x_1-x_2\|^2$.
    \item There exists a closed convex function $h_0$ such that $(h_0\  \square\ h) = \frac{\mu}{2}\|\cdot\|^2$.
    \item For any $x\in\mathcal{E}$, $g \in \partial h(x)$ and $\tilde \mu < \mu$, all $y\in B(x,\eta)$ for $\eta>0$ small enough have
    $$ h(y) \geq h(x) + g^T(y-x) + \frac{\tilde \mu}{2}\|y-x\|^2 \ . $$
    \end{enumerate}
\end{proposition}
\noindent We refer to this last condition as $h$ being {\it (locally) $\mu$-strongly convex w.r.t.~$(x,g)$}.
For twice continuously differentiable $h$, local smoothness and strong convexity at some $x$ are equivalent to $ \mu L \preceq \nabla^2h(x) \preceq L I $.

\noindent {\bf Smooth and Strongly Convex Sets.}
We say a nonempty closed convex set $S\subset \mathcal{E}$ is {\it (globally) $\beta$-smooth}, if every point $\bar y\in\bdry S$ and unit length normal vector $\zeta\in N_S(\bar y)$ provide a ball inner approximation
\begin{equation} \label{eq:inner-ball}
    B\left(\bar y - \frac{1}{\beta}\zeta, \frac{1}{\beta}\right) \subseteq S \ .
\end{equation}
We abuse notation, saying any potentially nonsmooth set is $\infty$-smooth, viewing the limit above as reducing to $\bar y \in S$. In the other limit, the only $0$-smooth sets are half-spaces.
Mirroring Proposition~\ref{prop:smooth_func}, several equivalent characterizations of smooth sets are given below, proof deferred to Appendix~\ref{pf:smooth_set}\footnote{Characterizations of more general nonconvex notions of smooth sets are given in~\cite[Definition 3.2, Page 286]{goncharov2017strong}. However, we are unable to find any statement and proof of such equivalent characterizations for simply convex sets based on convex normal vectors. Hence our proofs of these characterizations may be of some independent interest.}.
\begin{proposition} \label{prop:smooth_set}
    For a closed convex set $S$ and a constant $\beta\in (0,\infty)$, the following are equivalent:
    \begin{enumerate}[label=(\alph*)]
    \item The set $S$ is globally $\beta$-smooth.
    \item For any $x_i \in S$ with unit length $\zeta_i \in N_S(x_i)$, $ (\zeta_1-\zeta_2)^T(x_1-x_2)\geq \frac{1}{\beta}\|\zeta_1-\zeta_2\|^2 $.
    \item There exists a closed convex set $S_0$ such that  $S_0+B(0, \frac{1}{\beta})=S$.
    \end{enumerate}
    These conditions imply the following local smoothness condition
    \begin{enumerate}
    \item[(d)] For any $\bar y\in S$, unit length $\zeta\in N_S(\bar y)$ and $\tilde \beta > \beta$, there exists $\eta>0$ small enough that $$ B\left(\bar y - \frac{1}{\Tilde{\beta}}\zeta, \frac{1}{\Tilde{\beta}}\right) \cap B(\bar y, \eta) \subseteq S \cap B(\bar y, \eta) \ . $$ 
    \end{enumerate}
\end{proposition}

\noindent Note (b) implies Lipschitzness of unit normal vectors $\|\zeta_1-\zeta_2\| \leq \beta\|x_1-x_2\|$. The above condition (d), which we refer to as being (locally) $\beta$-smooth w.r.t.~$(\bar y, \zeta)$, does not imply the other conditions: For any $\epsilon>0$, $S=\{(x,y) \mid |x| \leq \epsilon \}$ is globally only $1/\epsilon$-smooth but locally $0$-smooth. We expect this implication to hold among compact sets but are unaware of a proof of this result.

We say a set $S$ is {\it (globally) $\alpha$-strongly convex}, if every point $\bar y\in\bdry S$ and unit length normal vector $\zeta\in N_S(\bar y)$ give a ball outer approximation
\begin{equation}\label{eq:outer-ball}
    B\left(\bar y - \frac{1}{\alpha}\zeta, \frac{1}{\alpha}\right) \supseteq S \ .
\end{equation}
In particular, this reduces to saying $S$ is $0$-strongly convex if $S$ is convex. In the other limit, singletons are the only $\infty$-strongly convex sets. Mirroring Proposition~\ref{prop:sc_func}, below are equivalent characterizations, proven by the equivalences of $\mathbf{(e),(h),(f),(l)}$ in~\cite[Theorem 2.1, page 265-266]{goncharov2017strong}.

\begin{proposition}\label{prop:sc_set}
For a closed convex bounded set $S$ and a constant $\alpha\in (0,\infty)$, the following are equivalent:
\begin{enumerate}[label=(\alph*)]
    \item A set $S$ is globally $\alpha$-strongly convex.
    \item For any $x_i \in S$ with unit length $\zeta_i \in N_S(x_i)$,  $ (\zeta_1-\zeta_2)^T(x_1-x_2)\geq \alpha \|x_1-x_2\|^2 $.
    \item There exists a closed convex set $S_0$ such that  $S_0+S = B(0, \frac{1}{\alpha})$. 
    \item For any $\bar y\in S$, unit length $\zeta\in N_S(\bar y)$ and $\tilde \alpha < \alpha$, there exists $\eta>0$ small enough that $$ B\left(\bar y - \frac{1}{\Tilde{\alpha}}\zeta, \frac{1}{\Tilde{\alpha}}\right) \cap B(\bar y, \eta) \supseteq S \cap B(\bar y, \eta) \ . $$ 
\end{enumerate}
\end{proposition}

\noindent We refer to this last condition as being {\it (locally) $\mu$-strongly convex w.r.t.~$(\bar y, \zeta)$}.


Lastly, we note a helpful lemma as the containments $B(e,r) \subseteq S \subseteq B(e,R)$ imply $\gamma_{B(e,R),e} \leq \gamma_{S,e} \leq \gamma_{B(e,r),e}$. From this, it follows that $\gamma_{S,e}$ is uniformly $1/R$-Lipschitz.
\begin{lemma} \label{lem:gauge_bound}
    For any closed convex set $S$ with $e \in \mathrm{int\ } S$, the gauge $\gamma_{S,e}(z)$ has
    $$\|z-e\|/ \sup \{\|x-e\| \mid x \in S\} \leq \gamma_{S,e}(z) \leq \|z-e\|/ \inf \{\|x-e\| \mid x \notin S\} \ .$$ 
\end{lemma}

\subsection{Common Families of Smooth/Strongly Convex Sets}\label{subsec:pre_def}
Many commonly encountered families of constraint sets possess smoothness and/or strong convexity. Below, we give three explicit examples (namely, p-norms, functional constraints, and epigraphs) and further three more common operations that preserve these structures (namely, affine transformations, Minkowski sums, and intersections). See~\cite[page 259-297]{goncharov2017strong}, \cite{polovinkin1996strongly}, and \cite{vial1982strong} as classic references for more on smooth and strongly convex sets. Note for general $S$, verifying its smoothness or strong convexity and computing the associated constants can be highly nontrivial.

\begin{figure}[t]
\begin{minipage}[t]{0.32\linewidth}
    \centering
	\includegraphics[width=0.8\linewidth]{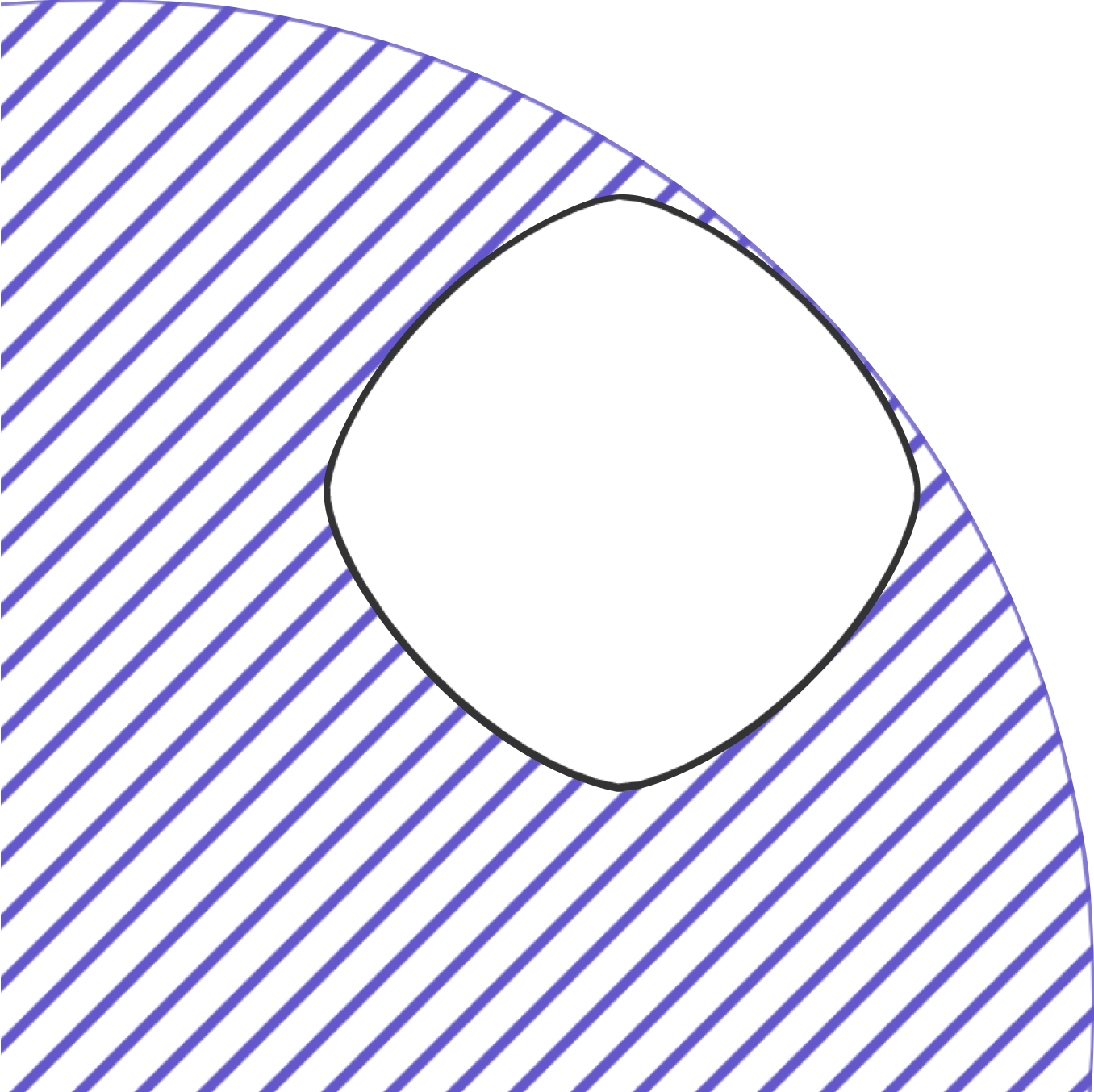}
	\caption*{(a) $p=1.5$-norm ball}
	\label{fig:sc}
\end{minipage}
\begin{minipage}[t]{0.32\linewidth}
	\centering
	\includegraphics[width=0.8\linewidth]{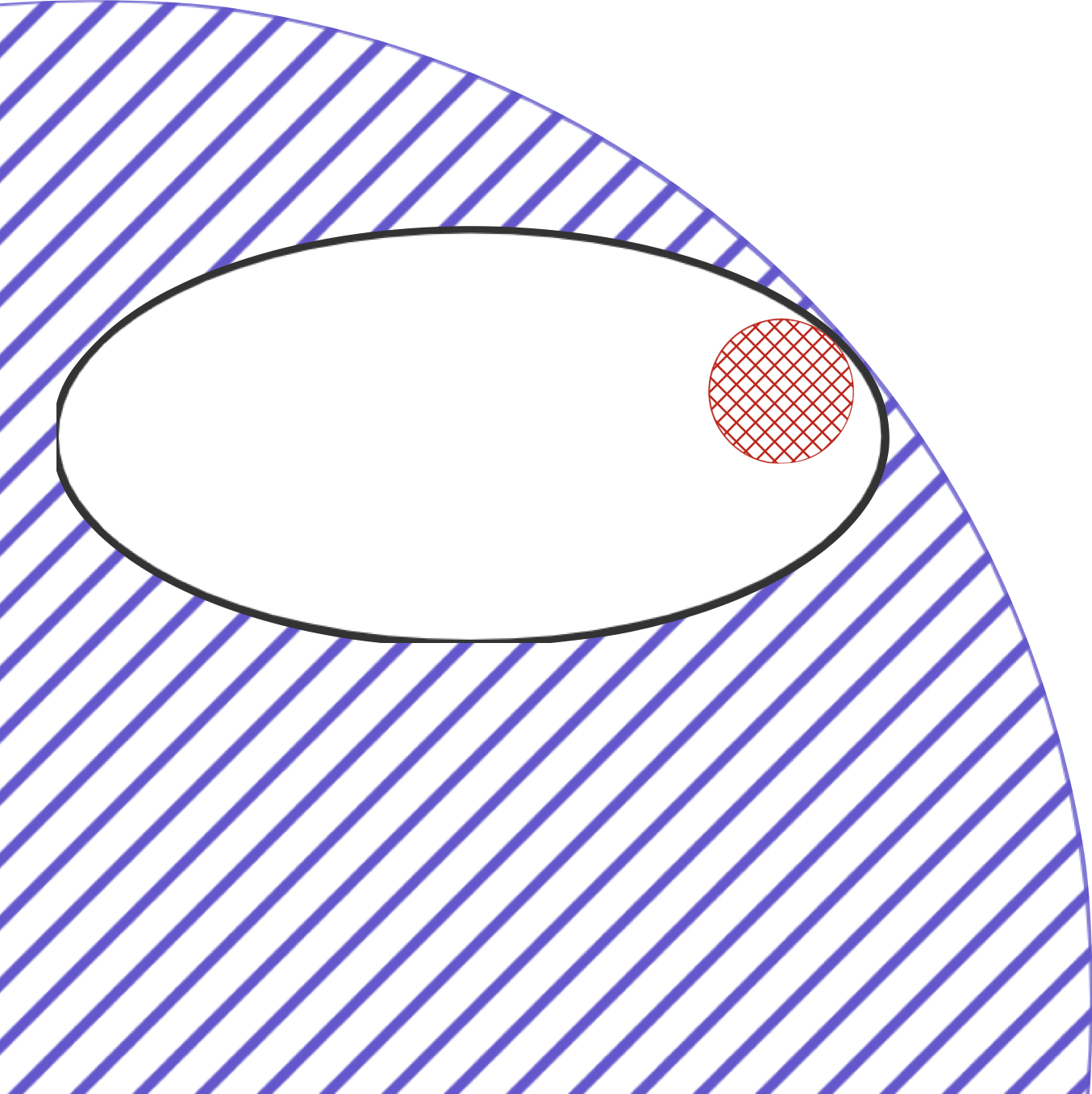}
	\caption*{(b) $p=2$-norm ellipsoid}
	\label{fig:sm}
\end{minipage}
\begin{minipage}[t]{0.32\linewidth}
	\centering
	\includegraphics[width=0.8\linewidth]{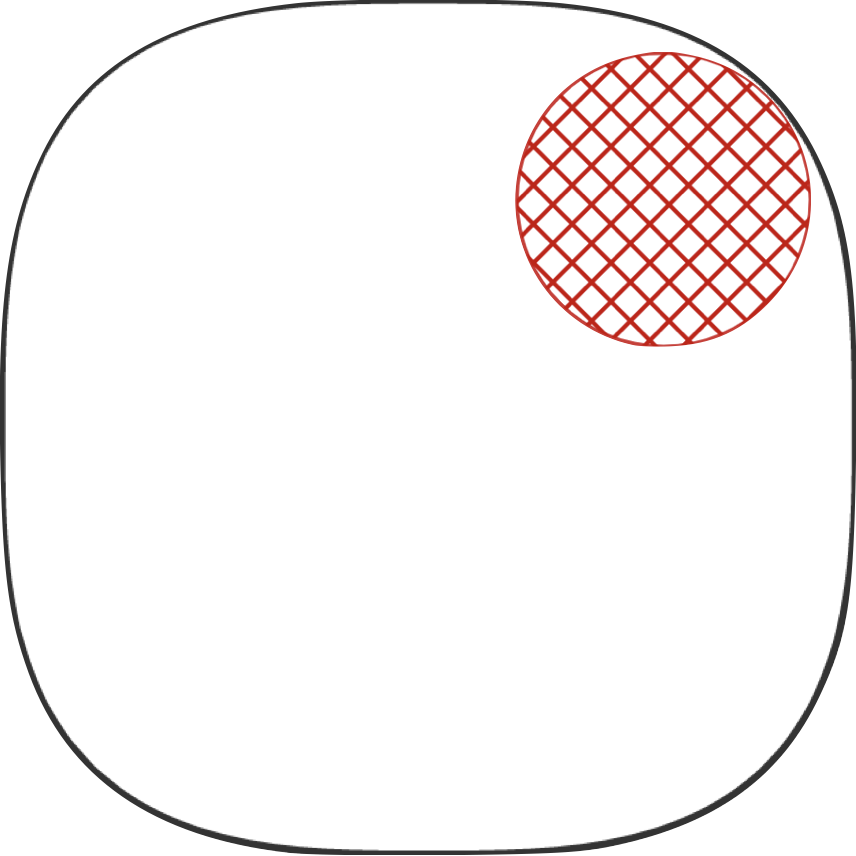}
	\caption*{(c) $p=3$-norm ball }
	\label{fig:sc_sm}
\end{minipage}
\caption{Examples of strongly convex and smooth sets with the inner approximation~\eqref{eq:inner_approx} in red and the outer approximation~\eqref{eq:outer_approx} in blue. The $p=1.5$-norm ball in (a) is strongly convex but not smooth. The $p=2$-norm ellipsoid in (b) is smooth and strongly convex. The $p=3$-norm ball in (c) is smooth but not strongly convex.}
\label{fig:def}
\end{figure}

\noindent {\bf Example 1. p-Norms and Schatten p-Norms.} For $\mathcal{E} = \RR^n,$ denote the $p$-norm of a vector by $\|x\|_p^p=\sum_{i=1}^n|x_i|^p$. The $p$-norm ball $B_p(0,1) = \{ x \mid \|x\|_p\leq 1\}$ is either smooth or strongly convex depending on the choice of $p\in(1,\infty)$. This ball is smooth with constant $\beta = (p-1)n^{\frac{1}{2}-\frac{1}{p}}$ when $p\in [2,\infty)$ and strongly convex with constant $\alpha= (p-1)n^{\frac{1}{2}-\frac{1}{p}}$ when $p\in (1,2]$~\cite[Lemma 3]{garber2015faster}. See Figure~\ref{fig:def}. Note for $p=1$ or $p=\infty$, $B_p$ is a polyhedron and hence is neither smooth nor strongly convex. Similarly, in the space of matrices with trace inner product, the Schatten $p$-norm defined by the $p$-norm of a matrix's singular values $\|X\|_p^p = \sum_{i=1}^d\sigma_i(X)^p$ yields smooth and/or strongly convex unit balls under the same conditions~\cite[Lemma 6]{garber2015faster}.

Norm constraints have found widespread usage describing trust regions, adding regularization, and inducing approximate sparsity (low-rankness) for (Schatten) norms as $p\searrow 1$. These will provide a working example throughout for applying our theory in Section~\ref{subsec:ex} and for numerics in Section~\ref{sec:appli}. 

\noindent {\bf Example 2. Functional Constraints.} Often, constraint sets are given by a level set of some function $h$, $S=\{x \mid h(x) \leq z\}$ for some fixed $z\in\mathbb{R}$. This level set inherits the smoothness and/or strong convexity of the function $h$ as stated below, allowing $L=\infty$ for nonsmooth function and $\mu=0$ for non-strongly convex functions. These properties are discussed more generally for potentially nonconvex functions in~\cite[Proposition 4.14]{vial1983strong}.
\begin{lemma} \label{lemma:level}
    Consider any closed convex $h$ that is locally $L$-smooth and $\mu$-strongly convex w.r.t. some $x$ with $h(x)=z$ and $0\neq g\in \partial h(x)$. Then $S=\{x \mid h(x) \leq z\}$ is locally $L/\|g\|$-smooth w.r.t.~$(x,\frac{g}{\|g\|})$ and $\mu/\|g\|$-strongly convex w.r.t.~$(x,\frac{g}{\|g\|})$.
\end{lemma}
\begin{proof}
      Note $\frac{g}{\|g\|}$ is a unit length normal vector of $S$ at $x$. Then $L$-smoothness ensures $h(y) \leq z$ holds whenever $h(x) + g^T(y-x) + \frac{L}{2}\|y-x\|^2 \leq z$ (i.e., $y\in B(x - \frac{\|g\|}{L} \frac{g}{\|g\|}, \frac{\|g\|}{ L })$). Similarly, $\mu$-strong convexity ensures $h(y)\leq z$ implies $h(x) + g^T(y-x) + \frac{\mu}{2}\|y-x\|^2 \leq z$ (i.e., $y\in B(x - \frac{\|g\|}{\mu} \frac{g}{\|g\|}, \frac{\|g\|}{\mu})$).
\end{proof}

\noindent {\bf Example 3. Epigraphs.} Similarly, the epigraph of a smooth and/or strongly convex function is locally smooth and/or strongly convex (proof deferred to Appendix~\ref{pf: epigraph}). 
\begin{lemma} \label{lemma:epigraph}
    Consider any closed convex $h$ that is locally $L$-smooth, $\mu$-strongly convex w.r.t. some $x$ and $g\in\partial h(x)$ and let $(\zeta,\delta) = (g,-1)/\|(g,-1)\|$. Then $\epi h = \{(x,t)\mid h(x) \leq t\}$ is locally $L|\delta|$-smooth and $\mu|\delta|^3$-strongly convex w.r.t. $((x,h(x)), (\zeta,\delta))$.
\end{lemma}
This implies the epigraph of any $L$-smooth function is globally $L$-smooth since $|\delta|\leq 1$. No similar global result holds for strongly convex functions since globally strongly convex sets must be bounded.

\noindent {\bf Example 4. Affine Transformations.} For any $L$-smooth $\mu$-strongly convex function $h$, its affine transformation $h(Ax+b)$ is well-known to be $\lambda_{ \max}(A^TA)L$-smooth and $\lambda_{\min}(A^TA)\mu$-strongly convex. Note that if $A$ has a null space, $A^TA$ has a zero eigenvalue and so strong convexity is lost. A weakened version of these two properties extends to sets, stated below. 
\begin{lemma} \label{lemma:aff_trans}
    Consider any closed, $\beta$-smooth, $\alpha$-strongly convex set $S$ and linear invertible $A\colon \mathcal{E}\rightarrow \mathcal{E}'$, $b\in\mathcal{E}'$. Then the set $\{x \mid Ax-b\in S\}$ is $\beta\frac{\lambda _{\max}(A^TA)}{\sqrt{\lambda_{\min}(A^TA)}}$-smooth and $\alpha \frac{\lambda _{\min}(A^TA)}{\sqrt{\lambda_{\max}(A^TA)}}$-strongly convex. 
\end{lemma}
\begin{proof}
    Without loss of generality, consider $b=0$. Note $\{x \mid Ax\in S\}=A^{-1}S$. Since $S$ is $\beta$-smooth, by Proposition~\ref{prop:smooth_set}(c), there exists a closed convex set $S_0$, such that $S_0+B(0,\frac{1}{\beta})=S$. In particular, $A^{-1}S_0+A^{-1}B(0,\frac{1}{\beta})=A^{-1}S$. Proposition~\ref{prop:smooth_set}(c) applied to $A^{-1}B(0,\frac{1}{\beta})$ further gives a closed convex set $B_0$ such that $B_0+B\left(0,\frac{1}{\beta}\frac{\sqrt{\lambda_{\min}(A^TA)}}{\lambda_{\max}(A^TA)}\right)=A^{-1}B(0,\frac{1}{\beta})$. Since $A^{-1}B(0,\frac{1}{\beta})$ is a compact ellipsoid, $B_0$ has to be a compact set. It follows $A^{-1}S_0+B_0$ is closed and convex. Hence $\left(A^{-1}S_0+B_0\right)+B\left(0,\frac{1}{\beta}\frac{\sqrt{\lambda_{\min}(A^TA)}}{\lambda_{\max}(A^TA)}\right)=A^{-1}S$ and so applying Proposition~\ref{prop:smooth_set}(c) again with the closed convex set $A^{-1}S_0+B_0$, we conclude $A^{-1}S$ is $\beta\frac{\lambda_{\max}(A^TA)}{\sqrt{\lambda_{\min}(A^TA)}}$-smooth. Symmetric reasoning shows $S$ is $\alpha \frac{\lambda _{\min}(A^TA)}{\sqrt{\lambda_{\max}(A^TA)}}$-strongly convex. 
\end{proof}

\noindent {\bf Example 5. Minkowski Sums.} The Minkowski sum of two sets, defined as $S+T=\{s+t \mid s\in S, t\in T\}$, inherits smoothness from either set and remains strongly convex only if both sets are. These set properties are the reverse of the typical properties of summing functions. There $f+g$ is smooth if both $f$ and $g$ are but inherits strong convexity if either $f$ or $g$ is. The following lemma formalizes this.
\begin{lemma} \label{lemma:mink_sum}
   Consider any closed, $\beta_i$-smooth, $\alpha_i$-strongly convex sets $S_i$. Then $S_1+S_2$ is $(\beta_1^{-1}+\beta_2^{-1})^{-1}$-smooth and $(\alpha_1^{-1}+\alpha_2^{-1})^{-1}$-strongly convex.
\end{lemma}
\begin{proof}
    These results follow directly from equivalent characterization (c) of smoothness and strong convexity. For smoothness, Proposition~\ref{prop:smooth_set} ensures there exist closed convex $S_{0,i}$ such that $S_{0,i}+B(0,\frac{1}{\beta_i})=S_i$. Hence $S_1+S_2=S_{0,1}+B(0,\frac{1}{\beta_1})+S_{0,2}+B(0,\frac{1}{\beta_2})=S_{0,1}+S_{0,2}+B(0,\frac{1}{\beta_1}+\frac{1}{\beta_2})$, and since $S_{0,1}+S_{0,2}$ is a closed convex set, $S_1+S_2$ is $( \beta_1^{-1}+\beta_2^{-1})^{-1}$-smooth. Similarly for strong convexity, Proposition~\ref{prop:sc_set} ensures there exist closed convex $S_{0,i}$ such that $S_{0,i}+S_i = B(0,\frac{1}{\alpha_i})$. It follows that $S_1+S_2+S_{0,1}+S_{0,2} = B(0,\frac{1}{\alpha_1}+\frac{1}{\alpha_2})$, and since $S_{0,1}+S_{0,2}$ is a closed convex set, $S_1+S_2$ is $( \alpha_1^{-1}+\alpha_2^{-1})^{-1}$-strongly convex.
\end{proof}
\noindent In particular, noting the ball $B_\epsilon=\{x\mid \|x\|_2\leq \epsilon\}$ is $1/\epsilon$-smooth for any $\epsilon>0$, this lemma gives a natural $1/\epsilon$-smoothing of any nonsmooth set $S$ by $S+B_\epsilon$. However, checking membership in $S+B_\epsilon$ amounts to orthogonally projecting onto $S$, which is only computable for sufficiently simple $S$. 

\noindent {\bf Example 6. Intersections of strongly convex sets.} Lastly, we note that just as the maximum of several strongly convex functions is strongly convex, the intersections of strongly convex sets are strongly convex (see \cite[Proposition 2]{vial1982strong}).
\begin{lemma}
Consider any closed $\alpha$-strongly convex sets $S_i$ with nonempty intersection. Then $\cap_i S_i$ is $\alpha$-strongly convex.
\end{lemma}

    \section{The Structure of Gauges of Structured Sets}\label{sec:main_result}

In this section, we prove characterizations relating a set's strong convexity and smoothness to those of its gauge squared. For notational ease, we suppose $e=0$ and denote $\gamma_S = \gamma_{S,e}$. Further, let
\begin{align}
    D(S) &:= \sup \{\|x\| \mid x\in S\}  \label{eq:D} \ ,\\
    R(S) &:= \inf \{\|x\| \mid x\notin S\} \label{eq:R} \ .
\end{align} 
Note $\bar y = y/\gamma_S(y)$ is on the boundary of $S$ unless $\gamma_S(y)=0$.
The following theorem gives the local result: the gauge squared is strongly convex at $y$ if the set is at $\bar y$. We defer the proof of this theorem to Section~\ref{pf:gauge-strong-convexity}.

\begin{theorem} \label{thm:gauge-strong-convexity}
    Consider some constant $\alpha\in (0,\infty)$, any $y\in \mathcal{E}$ and closed convex set $S$ with $0\in \mathrm{int\ } S$. If $\gamma_S(y) \neq 0$ and at the point $\bar y = y/\gamma_S(y)$, $S$ is $\alpha$-strongly convex w.r.t.~$\left(\bar y, \zeta\right)$ for some unit normal vector $\zeta \in N_S(\bar y)$, then $\frac{1}{2}\gamma_S^2$ is strongly convex with parameter 
    \begin{equation}\label{parameter:strongly_convex}
    \frac{1}{2( \zeta^T \bar y)^3}\left( \zeta^T \bar y+\alpha\|\bar y\|^2-\sqrt{( \zeta ^T\bar y+\alpha\|\bar y\|^2)^2-4\alpha( \zeta^T \bar y)^3}\right) 
    \end{equation}
    w.r.t.~$(y, g)$, where $g =\frac{\gamma_S(y)\zeta}{\zeta^T\bar y} \in \partial (\frac{1}{2}\gamma_S^2)(y)$.
    If $\gamma_S(y)=0$ and $y = 0$, then $\{0\} = \partial (\frac{1}{2}\gamma_S^2)(y)$ and $\frac{1}{2}\gamma_S^2$ is $\frac{1}{D(S)^2}$-strong convex w.r.t.~$(0,0)$. If $\gamma_S(y)=0$ and $y \neq 0$, $\frac{1}{2}\gamma_S^2(y)$ is not strongly convex at $y$.
\end{theorem}
Note that~\eqref{parameter:strongly_convex} can be lower bounded by the following $O(\alpha)$ quantity:
\begin{equation} \label{eq:gauge-strong-convexity-bound}
    \frac{\alpha}{\zeta^T\bar y + \alpha \|\bar y\|^2} \ .
\end{equation}
Similarly, we find the following guarantee that the smoothness of $\frac{1}{2}\gamma_S^2$ at $y$ follows from the smoothness of $S$ at $\bar y$. The proof of this theorem is deferred to Section~\ref{pf:gauge-smoothness}.
\begin{theorem}\label{thm:gauge-smoothness}
    Consider some constant $\beta\in (0,\infty)$, any $y\in \mathcal{E}$ and closed convex set $S$ with $0\in \mathrm{int\ } S$. If $\gamma_S(y) \neq 0$ and at the point $\bar y = y/\gamma_S(y)$, $S$ is $\beta$-smooth w.r.t.~$\left(\bar y, \zeta\right)$ for some unit normal vector $\zeta \in N_S(\bar y)$, then $\frac{1}{2}\gamma_S^2$ is smooth with parameter 
    \begin{equation} \label{parameter:smooth}
    \frac{1}{2( \zeta^T \bar y)^3}\left( \zeta^T \bar y+\beta\|\bar y\|^2+\sqrt{( \zeta ^T\bar y+\beta\|\bar y\|^2)^2-4\beta( \zeta^T \bar y)^3}\right)
    \end{equation}
    w.r.t.~$(y, g)$, where $g  =\nabla (\frac{1}{2}\gamma_S^2)(y)=\frac{\gamma_S(y)\zeta}{\zeta^T\bar y}$. If $\gamma_S(y)=0$, then $\{0\} = \partial (\frac{1}{2}\gamma_S^2)(y)$ and $\frac{1}{2}\gamma_S^2$ is $\frac{1}{R(S)^2}$-smooth w.r.t.~$(y,0)$.
\end{theorem}

Note that~\eqref{parameter:smooth} can be upper bounded by the following $O(\beta)$ quantity:
\begin{equation} \label{eq:gauge-smooth-bound}
    \frac{\zeta^T \bar y+\beta \|\bar y\|^2}{(\zeta^T\bar y)^3} \ .
\end{equation}

The above two theorems provide the gauge squared's local strongly convex/smooth parameter. Then the gauge square's global strongly convex/smooth parameter can be obtained by bounding these over all possible points $\bar y$ (i.e. all boundary points) 
due to Propositions~\ref{prop:smooth_func} and~\ref{prop:sc_func}. We claim $\|\bar y\| \leq D(S)$ and $\zeta ^T \bar y \geq R(S)$. The first claim is immediate from the definition of $D(S)$. The second follows by considering the support function as $\zeta ^T \bar y = \sup_{z\in S}\{\zeta^T z\}\geq \sup_{z\in B(0,R(S))}\zeta ^T z = R(S)$. We then have the following global bounds.
\begin{corollary}\label{cor:gauge-strong-convexity-global}
    Consider a closed bounded convex set $S$ with $0\in \mathrm{int\ } S$. If $S$ is $\alpha$-strongly convex, then $\frac{1}{2}\gamma_S^2$ is strongly convex  with parameter
    $\frac{\alpha}{D(S)+\alpha D(S)^2}.
    $
\end{corollary}

\begin{corollary}\label{cor:gauge-smoothness-global}
    Consider a closed bounded convex set $S$ with $0\in \mathrm{int\ } S$. If $S$ is $\beta$-smooth, then $\frac{1}{2}\gamma_S^2$ is smooth with parameter
    $\frac{R(S)+\beta D(S)^2}{R(S)^3}.
    $
\end{corollary}

\subsection{Examples}\label{subsec:ex}
These theorems/corollaries above immediately establish the structure for the gauge squared of many common families of constraints. Here we discuss a few such examples, namely, halfspaces $\mathcal{H} = \{x \mid a^Tx \leq b\}$, $p$-norm unit balls $B_p=\{x \mid \|x\|_p \leq 1\}$, and $p$-norm ellipsoids $E_p=\{x \mid \|Ax-b\|_p \leq 1\}$. Table~\ref{tab:set-examples} summarizes these results, showing the smoothness and strong convexity of each $S$ and its half gauge squared. These examples will also be utilized in our numerical evaluations in Section~\ref{sec:appli}.

\begin{table}[t]
\centering
\begin{tabular}{|cc|c|c|c|}
\hline
\multicolumn{1}{|c}{}                     & &    &  Strong Convexity&  Smoothness\\ \hline\hline
\multicolumn{2}{|c|}{\multirow{2}{*}{$\mathcal{H}$}}                     &  \multicolumn{1}{c|}{$S$}    & 0 &0
\\
\multicolumn{2}{|c|}{}                     &  \multicolumn{1}{c|}{$\frac{1}{2}\gamma^2_S$}    & 0 &$\frac{\|a\|_2^2}{b^2}$ \\
\hline
\multicolumn{1}{|c|}{\multirow{3}{*}{\vspace{-1.7cm} $B_p$}} &   \multicolumn{1}{c|}{\multirow{2}{*}{$p\in (1,2)$}}&   $S$&  $(p-1)n^{\frac{1}{2}-\frac{1}{p}}$&$\infty$  \\ 
\multicolumn{1}{|c|}{} & \multicolumn{1}{c|}{} & $\frac{1}{2}\gamma^2_S$& $(p-1)n^{\frac{1}{2}-\frac{1}{p}}$ &$\infty$ \\ 
\multicolumn{1}{|c|}{}                  &  \multirow{2}{*}{$p=2$}&    $S$&  1& 1 \\ 
\multicolumn{1}{|c|}{} & \multicolumn{1}{c|}{} & $\frac{1}{2}\gamma^2_S$ &1 &1 \\ 
\multicolumn{1}{|c|}{}                  &  \multirow{2}{*}{$p\in (2,\infty)$}&    $S$& 0 & $(p-1)n^{\frac{1}{2}-\frac{1}{p}}$\\
\multicolumn{1}{|c|}{} & \multicolumn{1}{c|}{} & $\frac{1}{2}\gamma^2_S$ &0 &$(p-1)n^{\frac{1}{2}-\frac{1}{p}}$ \\ \hline
\multicolumn{1}{|c|}{\multirow{3}{*}{\vspace{-2cm} $E_p$}} &   \multicolumn{1}{c|}{\multirow{2}{*}{$p\in (1,2)$}}&   $S$&  $\frac{\lambda_{\min}}{\sqrt{\lambda_{\max}}}(p-1)n^{\frac{1}{2}-\frac{1}{p}}$&$\infty$  \\ 
\multicolumn{1}{|c|}{} & \multicolumn{1}{c|}{} & $\frac{1}{2}\gamma^2_S$&$C^{SC}_{p,b}\lambda_{\min}(p-1)n^{\frac{1}{2}-\frac{1}{p}}$ &$\infty$ \\ 
\multicolumn{1}{|c|}{}                  &  \multirow{2}{*}{$p=2$}&    $S$&  $\frac{\lambda_{\min}}{\sqrt{\lambda_{\max}}}$& $\frac{\lambda_{\max}}{\sqrt{\lambda_{\min}}}$ \\ 
\multicolumn{1}{|c|}{} & \multicolumn{1}{c|}{} & $\frac{1}{2}\gamma^2_S$ &$\frac{\lambda_{\min}}{(1+\|b\|_2)+(1+\|b\|_2)^2}$ &$\frac{\lambda_{\max}(2-\|b\|_2)}{(1-\|b\|_2)^2}$ \\ 
\multicolumn{1}{|c|}{}                  &  \multirow{2}{*}{$p\in (2,\infty)$}&    $S$& 0 & $\frac{\lambda_{\max}}{\sqrt{\lambda_{\min}}}(p-1)n^{\frac{1}{2}-\frac{1}{p}}$\\
\multicolumn{1}{|c|}{} & \multicolumn{1}{c|}{} & $\frac{1}{2}\gamma^2_S$ &0 &$C^{SM}_{p,b}\lambda_{\max}(p-1)n^{1-\frac{2}{p}}$ \\ \hline
\end{tabular}
\captionof{table}{Strong convexity and smoothness constants. For the $p$-norm ellipsoids, $\lambda_{\max}$ and $\lambda_{\min}$ denote the maximum and minimum eigenvalues of $A^TA$, $C^{SC}_{p,b}=\frac{1}{(1+\|b\|_2)+(p-1)n^{\frac{1}{2}-\frac{1}{p}}(1+\|b\|_2)^2}$, and $C^{SM}_{p,b} = \frac{9}{2}\left(\frac{p}{1-\|b\|_p^p}\right)^3$.} \label{tab:set-examples}
\end{table}

\noindent {\bf Halfspaces.} Consider any halfspace with $\mathcal{H}=\{x \mid a^Tx \leq b\}$. Such sets are not strongly convex ($\alpha=0$) but are infinitely smooth ($\beta=0$) everywhere. To ensure the gauge is well-defined, we require $0\in \mathrm{int\ }\mathcal{H}$ (that is, $b>0$). Note $D(\mathcal{H})=\infty$ and $R(\mathcal{H})=b/\|a\|_2$. Consider any $\bar y$ on the boundary of $\mathcal{H}$ and unit normal $\zeta = a/\|a\|_2\in N_{\mathcal{H}}(\bar y)$. Note $\zeta^T\bar y= b/\|a\|_2$. Then Theorem~\ref{thm:gauge-strong-convexity} vacuously implies $\mu=0$-strong convexity and Theorem~\ref{thm:gauge-smoothness} implies $L=\|a\|_2^2/b^2$-smoothness at $\bar y$.

We can directly compute the gauge of this set, verifying our theory's tightness
$$ \gamma_S(x) = \begin{cases} \frac{a^Tx}{b} & \text{ if } a^Tx > 0\\0 & \text{ otherwise}
\end{cases} \implies \frac{1}{2}\gamma_S^2(x) = \begin{cases} \frac{(a^Tx)^2}{2b^2} & \text{ if } a^Tx > 0\\0 & \text{ otherwise.}
\end{cases}
$$

Note more general polyhedrons $\{x \mid a_i^Tx \leq b_i\ \forall i=1,\dots,m\}$, with $b_i>0$ are neither smooth nor strongly convex, and similarly, their piecewise linear gauges are neither smooth nor strongly convex when squared. However, since the gauge of an intersection is the maximum of the gauges of its components, the resulting half gauge squared will be a finite maximum of several smooth functions. Section~\ref{sec:algs} discusses algorithms for such problems.

\noindent {\bf $p$-Norm Balls.} Consider any $p\in(1,\infty)$-norm unit ball $B_p =\{x \mid \|x\|_p \leq 1\}$. Depending on $p$, this ball and its gauge squared are either smooth or strongly convex. Namely~\cite[Lemma 4]{garber2015faster} showed $B_p$ and $\frac{1}{2}\gamma_{B_p}^2$ are both $\alpha=\mu=(p-1)n^{\frac{1}{2}-\frac{1}{p}}$-strongly convex whenever $p\in(1,2]$ and are $\beta=L=(p-1)n^{\frac{1}{2}-\frac{1}{p}}$-smooth whenever $p\in[2,\infty)$. Note that neither of these bounds are tight for $p\neq 2$. Our theory could be applied to yield tighter, although substantially less elegant, bounds. The details of such an approach are given in Appendix~\ref{append:ex}. 

\noindent {\bf $p$-Norm Ellipsoids.} As our last example, we consider generalizing the example above to ellipsoidal sets $E_p = \{x \mid \|Ax-b\|_p \leq 1\}$ for any $p\in (1,\infty)$ and $A$ invertible. Note for the gauge to be well-defined (i.e., having $0\in E_p$), we require $\|b\|_p \leq 1$.
The following lemma, mirroring Lemma~\ref{lemma:aff_trans}, allows us to bound the strong convexity and smoothness of such a set's gauge squared.

\begin{lemma} \label{lem:struc_trans}
If a set $S$ has $\frac{1}{2}\gamma_S^2$ $\alpha$-strongly convex or $\beta$-smooth, then $E = \{x \mid Ax \in S\}$ with invertible $A$, has $\frac{1}{2}\gamma^2_E(y)=\frac{1}{2}\gamma^2_S(Ay)$ being $\lambda_{\min}(A^TA)\alpha$-strongly convex or $\lambda_{\max}(A^TA)\beta$-smooth, respectively.
\end{lemma}

\begin{proof}
    By the definition of the gauge, we have
\begin{align*}
\gamma_E(y) =& \inf\{\lambda>0 \mid y \in \lambda E\} 
=\inf \{\lambda>0 \mid Ay \in \lambda S\}=\gamma_S(Ay) \ . 
\end{align*}    
Since $\frac{1}{2}\gamma_S^2(Ay)$ is the affine transformation of $\frac{1}{2}\gamma_S^2(y)$, then it is $\lambda_{\min}(A^TA)\alpha$-strongly convex or $\lambda_{\max}(A^TA)\beta$-smooth, respectively.
\end{proof}

Using this lemma, it suffices to bound the gauge squared of translated balls $T_p=\{x \mid \|x-b\|_p \leq 1\}$ to deduce bounds on $E_p$.
First, we compute bounds when $p=2$. Noting $T_2=\{x \mid \|x-b\|_2 \leq 1\}$ is 1-strongly convex and 1-smooth, bounds on $E_2$'s $\alpha$-strong convexity and $\beta$-smoothness follow from Lemma~\ref{lemma:aff_trans}. Theorems~\ref{thm:gauge-strong-convexity} and~\ref{thm:gauge-smoothness} tell us $\frac{1}{2}\gamma_{T_2}^2$ is strongly convex and smooth with
\begin{align*}
    \mu &= \inf_{\|\bar y -b\|_2=1} \frac{1}{(\bar y-b)^T\bar y+\|\bar y\|^2_2} = \inf_{\|\bar y -b\|_2=1} \frac{1}{2+3(\bar y-b)^Tb+\|b\|^2_2}
    = \frac{1}{(1+\|b\|_2)(2+\|b\|_2)} \ , \\
    L &= \sup_{\|\bar y -b\|_2=1} \frac{(\bar y-b)^T\bar y+\|\bar y\|_2^2}{((\bar y-b)^T\bar y)^3} = \sup_{\|\bar y -b\|_2=1} \frac{2+3(\bar y-b)^Tb+\|b\|^2_2}{(1+(\bar y-b)^Tb)^3}
    =\frac{2-\|b\|_2}{(1-\|b\|_2)^2} \ ,
\end{align*}
where the infimum and supremum are taken by Cauchy-Schwarz $-\|b\|_2 \leq (\bar y-b)^Tb\leq \|b\|_2$.
Then following Lemma~\ref{lem:struc_trans}, $\frac{1}{2}\gamma_{E_2}^2$ is $\frac{\lambda_{\min}(A^TA)}{(1+\|b\|_2)(2+\|b\|_2)}$-strongly convex and $\frac{\lambda_{\max}(A^TA)(2-\|b\|_2)}{(1-\|b\|_2)^2}$-smooth.

A similar calculation holds for general $p$, although we lack a closed form for the infimum and supremum over all boundary points. Since $T_p$ is a translated $p$-norm ball, it has the same strong convexity/smoothness constants as $B_p$. Lemma~\ref{lemma:aff_trans} then ensures the ellipsoid $E_p$ is $\frac{\lambda_{\min}(A^TA)}{\sqrt{\lambda_{\max}(A^TA)}}(p-1)n^{\frac{1}{2}-\frac{1}{p}}$-strongly convex if $1<p \leq 2$, and $\frac{\lambda_{\max}(A^TA)}{\sqrt{\lambda_{\min}(A^TA)}}(p-1)n^{\frac{1}{2}-\frac{1}{p}}$-smooth if $p\geq 2$.
The strong convexity and smoothness constants of functions $\frac{1}{2}\gamma^2_{E_p}$ do not have such simple closed forms. In Appendix~\ref{append:ex}, we compute bounds on these constants for any translated $p$-norm ball $T_p$, from which applying Lemma~\ref{lem:struc_trans} gives $O\left(\lambda_{\min}(A^TA)n^{\frac{1}{2}-\frac{1}{p}}\right)$-strong convexity and $O\left(\lambda_{\max}(A^TA)\frac{n^{1-\frac{2}{p}}}{(1-\|b\|_p^p)^2}\right)$-smoothness when $1<p\leq 2$ and $2\leq p<\infty$, respectively.

\subsection{Tightness of our Main Theorems and their Converses}
Our main Theorems~\ref{thm:gauge-strong-convexity} and~\ref{thm:gauge-smoothness} have shown that a strongly convex/smooth set has a strongly convex/smooth gauge function squared. Here we observe that these results are essentially tight in two respects. Theorem~\ref{thm:hessian-blow} below shows that no improvement in these constants is possible. Then Theorems~\ref{thm:gauge-strong-convexity-converse} and~\ref{thm:gauge-smoothness-converse} show the converse results, that if a set's gauge squared is $\alpha$-strongly convex or $\beta$-smooth then the set itself must be $O(\alpha)$-strongly convex or $O(\beta)$-smooth, respectively. Proofs of each of these results are deferred to Sections~\ref{proof:hessian-blow}, \ref{pf:gauge-strong-convexity-converse}, and~\ref{pf:gauge-smoothness-converse}.

\begin{theorem}\label{thm:hessian-blow}
    For any values of $\gamma,R, D>0$, there exists a convex set $S$, $\bar y\in \bdry S$ and unit $\zeta\in N_{S}(\bar y)$ such that, $S$ is $\gamma$-strongly convex and $\gamma$-smooth with respect to $(\bar y,\zeta)$ and
    $$\lambda_{min}(\nabla^2 \frac{1}{2}\gamma^2_S(\bar y)) =\frac{1}{2R^3}\left(R+\gamma D^2-\sqrt{(R+\gamma D^2)^2-4\gamma R^3}\right) , 
    $$
    $$\lambda_{max}(\nabla^2 \frac{1}{2}\gamma^2_S(\bar y)) =\frac{1}{2R^3}\left(R+\gamma D^2+\sqrt{(R+\gamma D^2)^2-4\gamma R^3}\right)
    $$
    where $D=\|\bar y\|$ and $R= \zeta ^T \bar y$.
\end{theorem}
\begin{theorem} \label{thm:gauge-strong-convexity-converse}
    Consider any set $S$ with $0\in \mathrm{int\ } S$ and $y\in\mathcal{E}$ with $\gamma_S(y)\neq 0$. If $\frac{1}{2}\gamma_S^2$ is $\mu$-strongly convex w.r.t.~$(y, g)$, then the set $S$ is strongly convex with parameter $\mu \zeta^T\bar y$ w.r.t.~$(\bar y,\zeta)$, where $\bar y = y/\gamma_S(y)$, $\zeta = g/\|g\|$.
\end{theorem}
\begin{theorem}\label{thm:gauge-smoothness-converse}
    Consider any set $S$ with $0\in \mathrm{int\ } S$ and $y\in\mathcal{E}$ with $\gamma_S(y)\neq 0$. If $\frac{1}{2}\gamma_S^2$ is $L$-smooth w.r.t.~$(y, g)$, then the set $S$ is smooth with parameter $L\zeta^T\bar y$
    w.r.t.~$(\bar y,\zeta)$ where $\bar y = y/\gamma_S(y)$, $\zeta=g/\|g\|$.
\end{theorem}

\subsection{Proofs of Theorems Characterizing Structured Gauges} \label{subsec:proofs}
Both strongly convex sets and smooth sets are defined in terms of balls built by a boundary point and corresponding normal vector. This perspective is critical for the proofs of our main theorems. The following lemmas characterize the gauges of these approximating balls. Their proofs are deferred to Appendix~\ref{pf: hessian formular} and~\ref{pf: hessian eigenvalues}. Note that below, we consider sets $B(\bar y-r\zeta, r)$, which may not contain the origin. As a result, it is important that we defined the gauge as $\inf\{\lambda>0 \mid x \in \lambda S\}$, which may no longer equal $\sup\{\lambda>0 \mid x \not\in \lambda S\}$. Note when $0 \notin S$, $\gamma_S(0)=\infty$.

\begin{lemma}\label{lemma: hessian formular}
    Consider any closed convex set $S$ with $0\in \mathrm{int\ } S$ and $y\in \mathcal{E}$ with $\gamma_S(y)\neq 0$. For the boundary point $\bar y=y/\gamma_S(y)$ and any unit normal vector $\zeta \in N_S(\bar y)$, the ball $B=B(\bar y-r\zeta, r)$ with radius $r>0$ has
    \begin{enumerate}[label=(\alph*)]
        \item $\gamma_S(y)=\gamma_B(y)$,
        \item $\nabla \frac{1}{2}\gamma_B^2(y) \in \partial \frac{1}{2}\gamma_S^2(y)$,
        \item $\nabla^2 \frac{1}{2}\gamma_B^2(y) = \frac{1}{\left(\bar \zeta^T \bar{y}\right)^3}\left(\left(\bar \zeta^T \bar{y}+\|\bar{y}\|^2\right)\bar\zeta\bar\zeta^T-\bar \zeta^T \bar{y}(\bar\zeta\bar{y}^T+\bar{y}\bar\zeta^T)+\left(\bar \zeta^T \bar{y}\right)^2I \right)$
    \end{enumerate}
    where $\bar \zeta =r\zeta$.
\end{lemma}

\begin{lemma}\label{lemma: hessian eigenvalues}
    Consider any ball $B=B(\bar y-r\zeta, r)$ and $y$ with $\bar y = y/\gamma_B(y)$ and $\zeta^T\bar y > 0$, the eigenvalues of the Hessian matrix $\nabla ^2 \frac{1}{2}\gamma_B^2(y)$, $\lambda_1 \leq \lambda_2 \leq \cdots \leq \lambda_n$, are given by:
    \begin{align*}
    \lambda_1 =& \frac{1}{2( \zeta^T \bar y)^3}\left( \zeta^T \bar y+\frac{\|\bar y\|^2}{r}-\sqrt{\left( \zeta ^T\bar y+\frac{\|\bar y\|^2}{r}\right)^2-\frac{4( \zeta^T \bar y)^3}{r}}\right) \geq \frac{\frac{1}{r}}{\zeta^T \bar y + \frac{\|\bar y\|^2}{r}} \ ,  \\
    \lambda_i =& \frac{1}{r\zeta^T \bar y}, \qquad \text{for } i=2,...,n-1 \ , \\
    \lambda_n =& \frac{1}{2( \zeta^T \bar y)^3}\left( \zeta^T \bar y+\frac{\|\bar y\|^2}{r}+\sqrt{\left( \zeta ^T\bar y+\frac{\|\bar y\|^2}{r}\right)^2-\frac{4( \zeta^T \bar y)^3}{r}}\right) \leq \frac{\zeta ^T\bar y+\frac{\|\bar y\|^2}{r}}{(\zeta^T \bar y)^3}\ .
    \end{align*}
\end{lemma}


\subsubsection{Proof of Theorem~\ref{thm:gauge-strong-convexity}} \label{pf:gauge-strong-convexity}
Consider any $y$ with $\gamma_S(y) \neq 0$, and boundary point $\bar y = y/\gamma_S(y)$ with unit normal vector $\zeta \in N_S(\bar y)$. Denote the (locally) outer approximating ball $B=B(\bar y-\zeta/\alpha, 1/\alpha).$ Applying the subgradient chain rule, we compute the gradient of this ball's gauge squared $g$ as
\begin{equation} \label{eq:subgrad_gauge_square_ball}
 \partial \left(\frac{1}{2}\gamma^2_B\right)(y)=\gamma_B(y) \partial \gamma_B(y) = \left\{\frac{\gamma_B(y)\zeta}{\zeta^T\bar y}\right\} \ .
\end{equation}
By Lemma~\ref{lemma: hessian formular} (c) and Lemma~\ref{lemma: hessian eigenvalues}, this half gauge squared is locally strongly convex at $\bar y$ with constant 
$$\mu = \frac{1}{2( \zeta^T \bar y)^3}\left( \zeta^T \bar y+\alpha\|\bar y\|^2-\sqrt{\left( \zeta ^T\bar y+\alpha\|\bar y\|^2\right)^2-4\alpha( \zeta^T \bar y)^3}\right) \ .
$$
For any $\tilde \mu < \mu$, consider a neighborhood such that all $z\in B(y,\eta)$ have
$$\frac{1}{2}\gamma_B^2(z) \geq \frac{1}{2}\gamma_B^2(y) + g^T(z-y) +\frac{\Tilde{\mu}}{2}\|z-y\|^2 \ . $$
Since $S$ is strongly convex w.r.t.~$(\bar y, \zeta)$, then $ S \cap B(\bar y, \bar \eta) \subseteq B\left(\bar y - \frac{1}{\alpha}\zeta, \frac{1}{\alpha}\right) \cap B(\bar y, \bar \eta)$ for some $\bar \eta$. Hence for $z$ with $\|z-y\|\leq \gamma_S(y)\frac{R(S)}{D(S)+R(S)}\bar \eta$, the boundary point $\bar z =z/\gamma_S(z)$ has 
\begin{align*}
\|\bar z -\bar y\|&=\|\frac{z}{\gamma_S(z)}-\frac{y}{\gamma_S(y)}\| \\
&\leq \frac{\|z\||\gamma_S(y)-\gamma_S(z)|}{\gamma_S(z)\gamma_S(y)}+\frac{\|z-y\|}{\gamma_S(y)} \\
&\leq \frac{D(S)\|z-y\|}{R(S)\gamma_S(y)}+\frac{\|z-y\|}{\gamma_S(y)} \\
&\leq \bar \eta \ ,
\end{align*}
where the second inequality is given by $\gamma_S$ being $1/R(S)$-Lipschitz and $\|\bar z\|\leq D(S)$. Then $\bar z \in S \cap B(\bar y, \bar \eta)$ implies $\bar z \in B\left(\bar y - \frac{1}{\alpha}\zeta, \frac{1}{\alpha}\right)$, it follows $\gamma_S(z) \geq \gamma_B(z)$.
Applying this bound and Lemma~\ref{lemma: hessian formular} (a) and (b), we conclude for any $\tilde \mu <\mu$ in some neighborhood of $y$, all $z$ have
\begin{equation} \label{equ:lower_bound}
\frac{1}{2}\gamma_S^2(z) \geq \frac{1}{2}(\gamma^B)^2(z) \geq \frac{1}{2}\gamma_S^2(y) + g^T(z-y) +\frac{\tilde \mu}{2}\|z-y\|^2 \ .
\end{equation}
That is, $\frac{1}{2}\gamma^2_S$ is locally $\mu$-strongly convex w.r.t.~$(y,g)$.

If $\gamma_S(y)=0$ and $y=0$, then $\partial (\frac{1}{2}\gamma_S^2)(y) = \{0\}$. For any point $z$, Lemma~\ref{lem:gauge_bound} ensures $\frac{1}{2}\gamma_S^2(z)\geq \frac{1}{2}\frac{\|z\|^2}{D^2(S)}$. Thus $\frac{1}{2}\gamma_S^2$ is $\frac{1}{D^2(S)}$-strongly convex w.r.t.~$(0,0)$.

If $\gamma_S(y) = 0$ and $y\neq 0$, then $\gamma_S(ty)=t\gamma_S(y)=0$ for any $t>0$. Thus $\frac{1}{2}\gamma_S^2(y)$ is constant along this ray and consequently cannot be strongly convex at $y$.

\subsubsection{Proof of Theorem~\ref{thm:gauge-smoothness}} \label{pf:gauge-smoothness}
Following the same technique as in Theorem~\ref{thm:gauge-strong-convexity}, but utilizing the maximum eigenvalue bound of Lemma~\ref{lemma: hessian eigenvalues} instead of minimum, ensures for $\gamma_S(y) \neq 0$ the function $\frac{1}{2}\gamma_S^2(y)$ is smooth with the parameter
\begin{align*}
L &= \frac{1}{2( \zeta^T \bar y)^3}\left( \zeta^T \bar y+\beta\|\bar y\|^2+\sqrt{\left( \zeta ^T\bar y+\beta\|\bar y\|^2\right)^2-4\beta( \zeta^T \bar y)^3}\right).
\end{align*}

If $\gamma_S(y)=0$ and $y=0$, then $\partial(\frac{1}{2}\gamma_S^2)(y) = \{0\}$. For any point $z$, Lemma~\ref{lem:gauge_bound} ensures $\frac{1}{2}\gamma_S^2(z) \leq \frac{1}{2}\frac{\|z\|^2}{R^2(S)}$. Thus $\frac{1}{2}\gamma_S^2$ is $\frac{1}{R^2(S)}$-smooth w.r.t.~$(0,0)$.

If $\gamma_S(y)=0$ and $y \neq 0$, the set $S$ is unbounded in the direction of $y$ and we still have $\partial (\frac{1}{2}\gamma_S^2)(y) = \{0\}$. Since $B(0,R(S)) \subseteq S$, by convexity, we have $B(ty,R(S)) \subseteq S$ for any $t>0$, thus $T=\cup_{t\geq0} B(ty, R(S)) \subseteq S$. Since $T$ is unbounded in the direction of $y$, and bounded in the direction of $-y$, for $\bar z=z/\gamma_T(z)$, we have the following two cases. If $y^T z \geq 0$, then $\gamma_T(z) =\frac{\|z\|}{\|\bar z\|}=\frac{1}{R(S)}\|z-\frac{y^Tz}{\|y\|^2}y\|$. If $y^Tz  <0$, then $\gamma_T(z)=\gamma_{B(0,R(S))}(z)=\frac{\|z\|}{R(S)}$. It follows that
\begin{equation*}
\frac{1}{2}\gamma_S^2(z) \leq \frac{1}{2}\gamma_T^2(z) = \begin{cases} \frac{1}{2} \frac{1}{R(S)^2}\|z-\frac{y^Tz}{\|y\|^2}y\|^2 & \text{if } y^Tz \geq 0 \\ \frac{1}{2}\frac{\|z\|^2}{R(S)^2} & \text{if } y^Tz  <0.
\end{cases}
\end{equation*}
Note if $y^Tz \geq 0$, $\|z-\frac{y^Tz}{\|y\|^2}y\|^2-\|z-y\|^2 =-\left(\|y\|-\frac{y^Tz}{\|y\|}\right)^2 \leq 0$, and if $y^Tz <0$, $\|z\|^2-\|z-y\|^2 = 2y^Tz-\|y\|^2 \leq 0 $. It follows that $\frac{1}{2}\gamma^2_T(z) \leq \frac{1}{2}\frac{\|z-y\|^2}{R(S)^2}$. Thus, $\frac{1}{2}\gamma_S^2$ is $\frac{1}{R(S)^2}$-smooth w.r.t.~$(y,0)$.

\subsubsection{Proof of Theorem~\ref{thm:hessian-blow}} \label{proof:hessian-blow}
Let $S$ be the convex hull of $B(0,\frac{\zeta^T\bar y}{2}) \cup B(c, 1/\gamma)$. Consider $\bar y =(\sqrt{D^2-R^2}, -R) \in \bdry S$ and $\zeta =(0,-1) \in N_S(\bar y)$, then $c=\bar y - \zeta/\gamma $. Note that since $B(0,\frac{\zeta^T\bar y}{2})$ has radius $\frac{\zeta^T \bar y}{2}<\zeta ^T\bar y$, one has $(0,-1) \in N_S(\bar y)$. By Lemma~\ref{lemma: hessian formular}, we get 
$$ \frac{1}{2}\nabla^2 \gamma _S^2( y) = \frac{1}{R^3}\begin{bmatrix} \gamma R^2 & \gamma R\sqrt{D^2-R^2} \\ \gamma R\sqrt{D^2-R^2} & R+\gamma (D^2-R^2) \end{bmatrix}. 
$$
This positive definite matrix has 
$$\lambda_{min}(\nabla^2 \frac{1}{2}\gamma^2_S(y)) =\frac{1}{2R^3}\left(R+\gamma D^2-\sqrt{(R+\gamma D^2)^2-4\gamma R^3})\right)
$$
and
$$\lambda_{max}(\nabla^2 \frac{1}{2}\gamma^2_S(y)) =\frac{1}{2R^3}\left(R+\gamma D^2+\sqrt{(R+\gamma D^2)^2-4\gamma R^3})\right)
$$
where $\zeta \in N_S(\bar y)$ and $\|\zeta\| =1$.

\subsubsection{Proof of Theorem~\ref{thm:gauge-strong-convexity-converse}} \label{pf:gauge-strong-convexity-converse}
Note that subgradients of the half gauge squared are given by
\begin{equation} \label{eq:subgrad_gauge_square}
 \partial \left(\frac{1}{2}\gamma^2_S\right)(y)=\gamma_S(y) \partial \gamma_S(y) = \left\{\frac{\gamma_S(y)\zeta}{\zeta^T\bar y} \mid \zeta\in N_S(\bar y)\right\}
\end{equation}
by the chain rule of subgradient calculus and the formula for subgradients of a gauge. Hence $\zeta=g/\|g\|\in N_S(\bar y)$ and $g = \gamma_S(y)\zeta/\zeta^T\bar y$. For any $\tilde \mu < \mu$, consider a neighborhood such that all $z\in B(y,\eta)$ have
$$\frac{1}{2}\gamma^2_S(z) \geq \frac{1}{2}\gamma^2_S(y) + \left(\frac{\gamma_S(y)\zeta}{\zeta^T\bar y}\right)^T(z-y) + \frac{\tilde\mu}{2}\|z-y\|^2 \ . $$
Dividing through by $\gamma^2_S(y)$, all $z\in B(\bar y,\bar \eta)$ with $\bar\eta = \eta/\gamma_S(y)$ have
$$\frac{1}{2}\gamma^2_S(z) \geq \frac{1}{2}\gamma^2_S(\bar y) + \left(\frac{\zeta}{\zeta^T\bar y}\right)^T(z-\bar y) + \frac{\tilde\mu}{2}\|z-\bar y\|^2 \ . $$
Expressing the set $S = \{z \mid \gamma_S(z) \leq 1\}$, we arrive at the local strong convexity containment of
\begin{align*}
    S \cap B(\bar y, \bar \eta) &=\left\{z \mid \frac{1}{2}\gamma^2_S(z) \leq \frac{1}{2}\right\} \cap B(\bar y, \bar \eta) \\
    &\subseteq \left\{z \mid \frac{1}{2}\gamma^2_S(\bar y)+ \left(\frac{\zeta}{\zeta^T\bar y}\right)^T(z- \bar y)+\frac{\mu}{2}\|z- \bar y\|^2 \leq \frac{1}{2}\right\} \cap B(\bar y, \bar \eta) \\
    &=\left\{z \mid \left\|z-\bar y +\frac{\zeta}{\mu\zeta^T\bar y}\right\|^2 \leq \frac{1}{\mu^2(\zeta^T\bar y)^2} \right\} \cap B(\bar y, \bar \eta) \ .
\end{align*}

\subsubsection{Proof of Theorem~\ref{thm:gauge-smoothness-converse}} 
Similar to Theorem~\ref{thm:gauge-strong-convexity-converse}. \label{pf:gauge-smoothness-converse}

    \section{Accelerated Feasibility and Optimization Methods} \label{sec:algs}
Our theory on the structure of gauges enables the design and analysis of accelerated methods for both feasibility problems~\eqref{eq:Feas-Problem} and constrained optimization problems~\eqref{eq:OPT-Problem}. To provide a notation capturing both settings via~\eqref{eq:Feas-Gauge-Problem} and~\eqref{eq:OPT-Gauge-Problem}, consider a generic minimization problem of the form
\begin{equation} \label{eq:finite-minmax}
    p_* := \min_y \max_{i=1\dots n}\{h_i(y)\} >0
\end{equation}
for $n$ closed convex nonnegative functions $h_1,\dots, h_n$. In both settings, each $h_i$ is $M$-Lipschitz and $\frac{1}{2}h_i^2$ is $L$-smooth and/or $\mu$-strongly convex, depending on the structure of $f$ and $S_i$:

Consider a feasibility problem~\eqref{eq:Feas-Problem} given by any compact convex sets $S_1,\dots S_m$ with $e_i\in \mathrm{int\ } S_i$ and $\cap S_i$ nonempty. Supposing each $S_i$ is $\beta_i$-smooth and $\alpha_i$-strongly convex (allowing $\beta_i=\infty$ and $\alpha_i=0$), the problem~\eqref{eq:Feas-Gauge-Problem} takes the form~\eqref{eq:finite-minmax} with
\begin{equation} \label{eq:feas-constants}
    \begin{cases} 
        h_i(y) = \gamma_{S_{i},e_i}(y) \mathrm{\ for\ }i=1\dots m\\
        M = 1/\min\{R(S_i - e_i)\}\\
        \mu = \min\{\frac{\alpha_i}{D(S_i - e_i)+\alpha_i D(S_i - e_i)^2}\}\\
        L = \max\{\frac{R(S_i - e_i)+\beta_i D(S_i - e_i)^2}{R(S_i - e_i)^3}\}
        \end{cases}
\end{equation}
where the Lipschitz constant $M$ follows from Lemma~\ref{lem:gauge_bound} and smoothness and strong convexity constants $L$ and $\mu$ for $\frac{1}{2}h_i^2$ are exactly our main results in Corollaries~\ref{cor:gauge-strong-convexity-global} and~\ref{cor:gauge-smoothness-global}. Note the boundedness assumed of each $S_i$ here is not necessary for smoothness of $\gamma_{S_i}^2$ but rather a convenient sufficient condition enabling use of our corollaries.

Consider an optimization problem~\eqref{eq:OPT-Problem} given by compact convex sets $S_1,\dots S_m$ and continuous concave function $f$ with bounded level set $S_f = \{x \mid f(x)>0\}$. Assume $0\in \mathrm{int\ } (\cap_{i} S_i)\cap S_f$ and let $G$ be the local Lipschitz constant of $f$ on $S_f$. Observing $f^{\Gamma}(y)=f_{\Gamma}(y)$~\cite[Proposition 11]{grimmer2023radial1} and 
$f_{\Gamma}(y) =\inf\{v>0 \mid f(y/v)>1/v\}= \inf\{v>0 \mid \frac{(y,1)}{v}\in \mathrm{hypo\ }f\}=\gamma_{\mathrm{hypo\ }f}((y,1)),
$
where $\mathrm{\hypo\ }f = \{(x,u) \mid 0 < u\leq f(x)\}$, the radial dual problem~\eqref{eq:OPT-Gauge-Problem} is also minimizing a finite maximum of $n=m+1$ gauges $\max\{\gamma_{\mathrm{hypo\ }f}(y,1),\gamma_{S_i}(y)\}$. Supposing each $S_i$ is $\beta_i$-smooth and $\alpha_i$-strongly convex and $f$ is $L_f$-smooth and $\mu_f$-strongly convex on $S_f$\footnote{The fact that radial methods only depend on the smoothness and strong convexity of $f$ on a level set rather than globally is convenient benefit previously noted in~\cite{grimmer2023radial2}.},~\eqref{eq:OPT-Gauge-Problem} takes the form~\eqref{eq:finite-minmax} with
\begin{equation} \label{eq:OPT-constants}
    \begin{cases} 
        h_i(y) = \gamma_{S_{i},0}(y) \mathrm{\ for\ }i=1\dots m\\
        h_{m+1}(y) = f^\Gamma(y)\\
        M = 1/\min\{R(S_i), R(S_f)\}\\
        \mu = \min\{\frac{\alpha_i}{D(S_i)+\alpha_i D(S_i)^2},\ \frac{\mu_f}{(1+G^2)^{3/2}D(\mathrm{hypo\ }f)+\mu_f D(\mathrm{hypo\ }f)^2}\}\\
        L = \max\{\frac{R(S_i)+\beta_i D(S_i)^2}{R(S_i)^3},\ \frac{R(\mathrm{hypo\ }f)+L_f D(\mathrm{hypo\ }f)^2}{R(\mathrm{hypo\ }f)^3}\}
        \end{cases}
\end{equation}
utilizing the same constant bounds for the gauge of each $S_i$ as in~\eqref{eq:feas-constants} and where the Lipschitz constant for $f^\Gamma$ follows from~\cite[Proposition 1]{grimmer2023radial2} and its strong convexity and smoothness constants follow from our Theorems~\ref{thm:gauge-strong-convexity} and~\ref{thm:gauge-smoothness} since Lemma~\ref{lemma:epigraph} ensures local $|\delta| L_f$-smoothness and $|\delta|^3\mu_f$-strong convexity of $\mathrm{hypo\ } f$ with $1/\sqrt{1+G^2}\leq |\delta| \leq 1$.

One can design optimization methods for~\eqref{eq:finite-minmax}, taking advantage of these basic structures.
To the best of our knowledge, methods and theory for this exact setting have not been previously derived. Prior works have considered settings of minimizing a finite maximum of Lipschitz smooth functions~\cite{Nesterov2005,Beck2012,Nemirovski2004,dima2019}. A maximum of Lipschitz strongly convex functions is itself Lipschitz and strongly convex, so guarantees like~\cite{lacoste2012simpler} apply. Below, we present three standard algorithms generalized to minimize a maximum of Lipschitz functions that, when squared, are smooth and/or strongly convex. Appendix~\ref{app:convergence-proofs} provides convergence theorems, summarized in Table~\ref{tab:rates}.

\noindent {\bf Methods for Finite Maximum Minimization.}
First, we consider applying the subgradient method to the half objective squared in~\eqref{eq:finite-minmax}. Letting $h(y) = \max_{i=1\dots m}\{h_i(y)\}$, this is defined with stepsizes $s_k$ as
\begin{equation} \label{eq:subgrad-method}
    y_{k+1} = y_k - s_k h(y_k)g_k, \quad g_k\in\partial h(y_k) \ . 
\end{equation}
Note utilizing the squared objective enables the method to benefit from any strong convexity present but requires additional care in the analysis as the square cannot be uniformly Lipschitz. In terms of computational cost, each iteration requires computing the function value of each $h_i$ and a subgradient of one $h_i$ attaining the maximum.

Second, as a direct improvement on the subgradient step, we consider the level projection step, which linearizes each component of the half objective squared and then projects onto the $\frac{1}{2}\bar h^2$-level set
\begin{equation}
    \mathtt{level\mbox{-}proj}(y,\bar h) := \argmin_{z}\left\{ \|z-y\|^2_2 \mid \frac{1}{2}h^2_i(y) + h_i(y)g_{i}^T(z-y) \leq \frac{1}{2}\bar h^2, \forall i=1, \dots ,n \right\} \ .  \label{eq:level-proj-step}
\end{equation}
Iterating this with a target accuracy $\bar h$ gives the level-projection method, defined as
\begin{equation} \label{eq:level-proj-method}
    y_{k+1} = \mathtt{level\mbox{-}proj}(y_k,\bar h) \ .
\end{equation}
This iteration is especially nice for feasibility problems as $\bar h=1$ corresponds to seeking a feasible solution. When $n=1$, this iteration is exactly the subgradient method with the Polyak stepsize rule. For $n=2$, this can be done in closed form while in general, computing the projection corresponds to a dimension $n$ quadratic program.  

Finally, to benefit from any smoothness in the squared objective, we consider the generalized gradient step at $y$ with $g_i\in\partial h_i(y)$ given by
\begin{equation}
    \mathtt{gen\mbox{-}grad}(y,s) :=\argmin_{z}\left\{\max_{i=1, \dots ,n}\left\{\frac{1}{2}h^2_i(y) + h_i(y)g_i^T(z-y)\right\} + \frac{1}{2s}\|z-y\|^2\right\}. \label{eq:general-grad-step}
\end{equation}
Note when $n=1$, this corresponds to a subgradient step. For $n=2$, this can be solved in closed-form. In general, computing this step corresponds to a dimension $n$ quadratic program (matching the complexity of the level-projection method). Iterating this with stepsizes $s_k$ gives the generalized gradient method, defined as $y_{k+1} = \mathtt{gen\mbox{-}grad}(y_k,s_k)$.
A direct improvement is the accelerated generalized gradient method for $L$-smooth optimization, defined as
\begin{equation} \label{eq:accel-method}
    \begin{cases}
    z_{k+1} = \mathtt{gen\mbox{-}grad}(y_k,1/L) \\
    t_{k+1}^2= (1-t_k)t_k^2+\frac{\mu}{L}t_{k+1} \\
    \beta_k = \frac{t_k(1-t_k)}{t^2_k+t_{k+1}} \\
    y_{k+1}=z_{k+1}+\beta_k(z_{k+1}-z_k) \ . 
    \end{cases}
\end{equation}

\begin{table}
	{\small
		\centering
			\renewcommand{\arraystretch}{2}
			\begin{tabular}{|c|c|c|c|c|}\hline
				Algorithm & Generic $\frac{1}{2}h^2_i$ & $\mu$-SC $\frac{1}{2}h^2_i$ & $L$-Smooth $\frac{1}{2}h^2_i$ & Both\\ \hline
				Subgradient Mthd & $O\left(\dfrac{MD}{\sqrt{T+1}}\right)$ & $O\left(\dfrac{M^2p_*}{\mu(T+1)}\right)$ & $O\left(\dfrac{MD}{\sqrt{T+1}}\right)$ & $O\left(\dfrac{M^2p_*}{\mu(T+1)}\right)$ \\
				Level-Proj Mthd & $O\left(\dfrac{MD}{\sqrt{T+1}}\right)$ & $O\left(\dfrac{M^2p_*}{\mu(T+1)}\right)$ & $O\left(\dfrac{MD}{\sqrt{T+1}}\right)$ & $O\left(\dfrac{M^2p_*}{\mu(T+1)}\right)$  \\
				Accel.~Gen.~Grad Mthd & - & - & $O\left(\dfrac{LD^2}{p_*(T+1)^2}\right)$ & $O\left((1-\sqrt{\mu/L})^T\right)$ \\
                \hline
			\end{tabular}
	}
 \caption{Convergence rates for minimizing $M$-Lipschitz functions $h_i$ with varied structure in $\frac{1}{2}h_i^2$ and minimizer $y^*$ with $\|y_0-y^*\|\leq D$.}  \label{tab:rates}
\end{table}

Applying these methods to the reformulations of feasibility and optimization problems gives accelerated, projection-free methods for smooth or strongly convex sets. 
\begin{theorem}[Accelerated Guarantees for Feasibility Problems] \label{thm:feas-rates}
    Consider the setting of~\eqref{eq:feas-constants}, and let $R=\min\{R(S_i - e_i)\}$. In general, applying the level projection method~\eqref{eq:level-proj-method} with $\bar h=1$ converges to feasibility with
    $$ \min_{k\leq T}\max_i \gamma_{S_i,e_i}(y_k) - 1 \leq
        \frac{\mathrm{dist}(y_0,\cap S_i)}{R\sqrt{T+1}} + \frac{\mathrm{dist}(y_0,\cap S_i)^2}{2R^2(T+1)} $$
    and if $\mu>0$ and $T+1 \geq \frac{32}{R^2\mu}\log(\mu\ \mathrm{dist}(y_0,\cap S_i)^2)$, this improves to
    $$ \min_{k\leq T}\max_i \gamma_{S_i,e_i}(y_k) - 1 \leq \frac{\sqrt{32}}{R^2\mu(T+1)} + \frac{64}{R^4\mu^2(T+1)^2}\ . $$
    If $\beta<\infty$ and $y^*$ is a minimizer of~\eqref{eq:Feas-Gauge-Problem}, applying the accelerated method~\eqref{eq:accel-method} converges with
    $$ \max_i \gamma_{S_i,e_i}(y_T) - h(y^*) \leq \frac{h(y_0)^2 - h(y^*)^2 + \gamma_0\|y_0-y^*\|^2}{2h(y^*)} \min\left\{\left(1-\sqrt{\frac{\mu}{L}}\right)^T,\ \frac{4L}{\left(2\sqrt{L} +T\sqrt{\gamma_0}\right)^2}\right\},$$
    where $\gamma_0=\frac{t_0(t_0L-\mu)}{1-t_0}$.
\end{theorem}
\begin{proof}
    The guarantees for the level projection method follow immediately from its convergence guarantees proven in Theorem~\ref{thm:level-proj-convergence}, substituting in $\bar h=1$ and the constants~\eqref{eq:feas-constants}. Similarly, the guarantees for the accelerated generalized gradient method follow immediately from Theorem~\ref{thm:acclerated-gradient-convergence}.
\end{proof}

\begin{theorem}[Accelerated Guarantees for Optimization Problems] \label{thm:OPT-rates}
    Consider the setting of~\eqref{eq:OPT-constants}, let $x^*$ maximize~\eqref{eq:OPT-Problem}, and $R=\min\{R(S_i), R(S_f)\}$. Let $x_0=y_0=0$, $y_k$ denote the iterates of some method minimizing the radial dual problem, and $x_k = y_k/h(y_k) \in \cap S_i$.
    In general, the iterates $y_k$ of the subgradient method~\eqref{eq:subgrad-method} with $s_k = \frac{D}{\|h(y_k)g_k\|\sqrt{T+1}}$ converges to optimality with
    $$ \frac{f(x^*) - \max_{k\leq T} f(x_k)}{f(x^*)} \leq  \frac{\|x_0-x^*\|}{R\sqrt{T+1}} + \frac{\|x_0-x^*\|^2}{2R^2(T+1)} $$
    and if $\alpha>0$ and $s_k=\frac{2}{\mu(k+2)+\frac{1}{R^4\mu(k+1)}}$, this improves to
    $$ \frac{f(x^*) - \max_{k\leq T} f(x_k)}{f(x^*)} \leq \frac{4}{R^2\mu(T+2)} + \frac{4\|x_0-x^*\|^2}{R^4\mu(T+1)(T+2)} \ . $$
    If $\beta<\infty$, applying the accelerated method~\eqref{eq:accel-method} converges with
    $$ \frac{f(x^*) -f(x_T)}{f(x^*)} \leq \frac{1}{2}\left(\frac{f(x^*)^2}{f(x_0)^2} - 1  + \gamma_0\|x_0-x^*\|^2\right)\min\left\{\left(1-\sqrt{\frac{\mu}{L}}\right)^T,\ \frac{4L}{\left(2\sqrt{L} +T\sqrt{\gamma_0}\right)^2}\right\} \ , $$
    where $\gamma_0=\frac{t_0(t_0L-\mu)}{1-t_0}$.
\end{theorem}
\begin{proof}
    Note by~\cite[Proposition 24]{grimmer2023radial1}, $y^*=x^*/f(x^*)$ minimizes $h$.
    Applying the subgradient method's convergence guarantees proven in Theorem~\ref{thm:subgradient-convergence} with $y_0=0$, establishes either
     $$ \min_{k\leq T} h(y_k) - h(y^*) \leq \frac{\|y^*\|}{R\sqrt{T+1}} + \frac{\|y^*\|^2}{2R^2h(y^*)(T+1)} $$
     or if $\alpha>0$
     $$ \min_{k\leq T} h(y_k) - h(y^*) \leq \frac{4h(y^*)}{R^2\mu(T+2)} + \frac{4\|y^*\|^2}{R^4\mu h(y^*)(T+1)(T+2)} \ . $$
    Similarly, the accelerated generalized gradient method by Theorem~\ref{thm:acclerated-gradient-convergence} has
     $$ h(y_T) - h(y^*) \leq \frac{h(y_0)^2 - h(y^*)^2 + \gamma_0\|y^*\|^2}{2 h(y^*)} \min\left\{\left(1-\sqrt{\frac{\mu}{L}}\right)^T,\ \frac{4L}{\left(2\sqrt{L} +T\sqrt{\gamma_0}\right)^2}\right\} \ . $$
    To translate these radial dual optimality guarantees to the original (primal) problem, observe that $x_k=y_k/h(y_k)$ and $h(y_k)\geq f^{\Gamma}(y)=\sup \{v>0 \mid v\cdot f(y_k/v)\leq 1\}$ have $f(x_k) \geq 1/h(y_k)$ and $x_k \in \cap S_i$ and $x^*$ has $f(x^*) = 1/h(y^*)$. Hence each result above upper bounds $1/f(x_k) - 1/f(x^*)$ by some $\epsilon$. The claims then follow as this implies $(f(x^*) - f(x_k))/f(x^*) \leq f(x_k) \epsilon \leq f(x^*)\epsilon$.
\end{proof}

These projection-free convergence guarantees closely relate to those of typical projected methods. Suppose for ease of comparison $x_0=0$. Then the classic projected subgradient method convergence rate in objective gap is $\frac{G\|x^*\|}{\sqrt{T}}$, whereas by Theorem~\ref{thm:OPT-rates}, our radial subgradient method has convergence rate is $\frac{f(x^*)\|x^*\|}{R(S_f)\sqrt{T}} + O(1/T)$. To related these, note from Lipschitz continuity that $\frac{f(0)}{G} \leq R(S_f)$. Thus, our radial subgradient method convergence bound has leading term bounded by
\begin{align*}
    \frac{f(x^*)\|x^*\|}{R(S_f)\sqrt{T}}
    &\leq \frac{f(x^*)}{f(0)}\frac{G\|x^*\|}{\sqrt{T}} \ .
\end{align*}
When initialized near optimal $f(0) \approx f(x^*)$, our radial subgradient method upper bound is at least as good as the classic subgradient bound. In general, which bound is stronger varies, depending on the quality of initialization and Lipschitz constants.

To compare the accelerated smooth convergence rates, for simplicity we focus on the dependence with respect to the objective smoothness constant $L_f$. Given access to orthogonal projections, accelerated gradient methods have objective gap converge at rate $O\left(\frac{L_{f} \|x^*\|^2}{T^2}\right)$. Considering only terms with $L_f$, our rate above is $O\left(\frac{f(x^*)D(\mathrm{hypo\ }f)^2}{R(\mathrm{hypo\ } f)^3}\frac{L_{f}  \|x^*\|^2}{T^2}\right)$. 
The factor $\frac{f(x^*)D(\mathrm{hypo\ }f)^2}{R(\mathrm{hypo\ } f)^3}$ relating these guarantees is always at least one as $f(x^*)$ and $D(\mathrm{hypo\ }f)$ both upper bound $R(\mathrm{hypo\ } f)$. So our proven guarantee is weaker by this geometric constant. In exchange for this factor, the considered radial methods avoid projections.

\noindent {\bf Further Methods and Oracle Models (When $n$ is Large).} 
Lastly, since the level projection and generalized gradient methods scale poorly with $n$ (i.e., requiring QP solves), we note two methods applicable when $n$ is large and each $h_i$ when squared is $L$-smooth. The accelerated smoothing method of~\cite{Nesterov2005,Beck2012} replaces the finite maximum by a $1/\epsilon$-smooth approximation. Then applying Nesterov's accelerated method gives an overall convergence rate of $O(L/T)$, improving on the subgradient method's $O(1/\sqrt{T})$. Renegar and Grimmer~\cite{RenegarGrimmer2022}'s restarting scheme showed such a method can attain a faster $O(1/\sqrt{\epsilon})$ rate, when $\mu$-strong convexity additionally holds. Alternatively, replacing the finite maximum by $\max_{\theta\in \Delta} \sum \theta_i h_i(x)$ where $\Delta$ is the simplex gives an equivalent convex-concave minimax optimization problem. Then the same $O(L/T)$ speed up can be attained by the extragradient method~\cite{Nemirovski2004}.

Other sophisticated nonsmooth minimization methods could also be explored. Classical bundle methods~\cite{Lemarechal1975,Wolfe1975}, which construct cutting plane models used to produce stable descent sequences, could be applied with convergence guarantees following from their recent literature~\cite{Lan2015,Du2017,Liang2021,atenas2023unified,diaz2023optimal}. An alternative scheme for minimizing finite maximums of smooth functions was recently proposed by Han and Lewis~\cite{han2023survey}. Their approach aims to maintain  (at least) $n$ points, each staying within a smooth region of the objective (where a single $h_i$ attains the maximum). Then, each iteration solves a relatively simple second-order cone program to update each point within its region.

\section{Applications and Numerics}\label{sec:appli}
In this section, we show the effectiveness of our gauge theory and the above algorithms on several smooth and/or strongly convex feasibility problems and constrained optimization problems. We run all methods in Section~\ref{sec:algs} on feasibility problems to verify the convergence rates in Table~\ref{tab:rates}. Then, considering constrained optimization problems, we compare our projection-free accelerated method applied to the radial dual~\eqref{eq:OPT-Gauge-Problem} against standard first-order and second-order solvers. Our numerical experiments are conducted on a two-core Intel i5-6267u CPU using Julia 1.8.5 and \texttt{JuMP.jl} to access solvers Gurobi 10.0.1, Mosek 10.0.2, COSMO 0.8.6, SCS 3.2.1\footnote{The source code is available at \url{https://github.com/nliu15/Gauges-and-Accelerated-Optimization-over-Smooth-Strongly-Convex-Sets}.}. In Appendix~\ref{append:jump}, we validate our numerics using \texttt{Convex.jl} instead of \texttt{JuMP.jl}, seeing similar performance.

Our numerics focus on ellipsoidal constraints $S_i=\{x \mid \|A_ix - b_i\|_{p_i}\leq 1\}$ since the strong convexity/smoothness can be directly controlled via $p_i$. Computing $\gamma_{S_i,e_i}$ has a closed form for $p_i\in\{1,2,3,4,\infty\}$ and otherwise corresponds to one-dimensional polynomial root finding, from which automatic differentiation can compute gradients. In general, reliance on numerical root finding may result in small errors on the order of machine precision. As occurs widely in nonsmooth optimization, small machine rounding inaccuracies can greatly change the algorithm trajectory as subgradients/normal vectors can change suddenly. This concern is only minor here as for the considered ellipsoids (and in general for any smooth sets), normal vectors vary continuously.

\subsection{Application to \texorpdfstring{$p$}{Lg}-norm Ellipsoid Feasibility Problems}\label{subsec:numerics1}

First, we illustrate the above algorithms by considering the $p$-norm ellipsoid feasibility problem
\begin{equation}
    \text{Find } x \text{ s.t. } x\in \{x \mid \|A_1x-b_1\|_{p_1}\leq \tau_1 \}\cap \{x \mid \|A_2x-b_2\|_{p_2}\leq \tau_2 \} \ ,
\end{equation}
setting $S_i=\{x \mid \|A_ix - b_i\|_{p_i}\leq 1\}$. Recall the discussion in Example~\ref{subsec:ex} showed ellipsoidal sets are smooth when $2\leq p < \infty$, strongly convex when $1<p\leq 2$, and both when $p=2$. We consider these three settings separately in dimensions $n=1600$, matching the speedups expected from Section~\ref{sec:algs}.

We generate synthetic ellipsoids corresponding to linear regression uncertainty sets with two differently corrupted data sets $i=1,2$. We generate $x_{true}\in\mathbb{R}^n$ and $A_i\in\mathbb{R}^{n\times n}$ with standard normal entries and $b_i = A_ix_{true} + \epsilon_i$ where $\epsilon_i$ has generalized normal distribution with shape parameter $\beta = p_i$. Then we select each $\tau_i$ such that $x_{true}\in S_i$ holds with probability $0.975$. Our feasibility problem can then be viewed as seeking any $x\in \cap S_i$ able to provide a reasonable estimate of $x_{true}$, performing well on each data set. Note setting the shape parameter $\beta<2$ gives a heavier tail than Gaussian data (with $\beta=1$ yielding a Laplace distribution) while $\beta>2$ gives a lighter tail (with $\beta=\infty$ yielding a uniform distribution).

Following~\eqref{eq:Feas-Gauge-Problem}, we approach finding a feasible point via $\min_y \max_i\{\gamma_{S_i,e_i}(y)\}$
where here each $e_i$ is computed by $35$ iterations of conjugate gradient applied to $A_ix=b_i$.  
We set $(p_1,p_2)=(1.5,1.8)$, $(p_1,p_2)=(2,2)$, $(p_1,p_2)=(3,4)$ to generate three cases with only strong convexity, both strong convexity and smoothness, and only smoothness. For each setting, the methods in Table~\ref{tab:rates} are considered: we implement subgradient methods with stepsize $\alpha_k = \eta$, $\eta/\sqrt{k+10}$, level methods with $\bar h$ as the optimal value and as one, and the accelerated generalized gradient method. Parameters $\eta,L,\mu$ for each method are tuned to two significant figures; see the source code for exact values.

\begin{figure}[t]
\begin{minipage}[t]{0.31\linewidth}
    \centering
	\includegraphics[width=\linewidth]{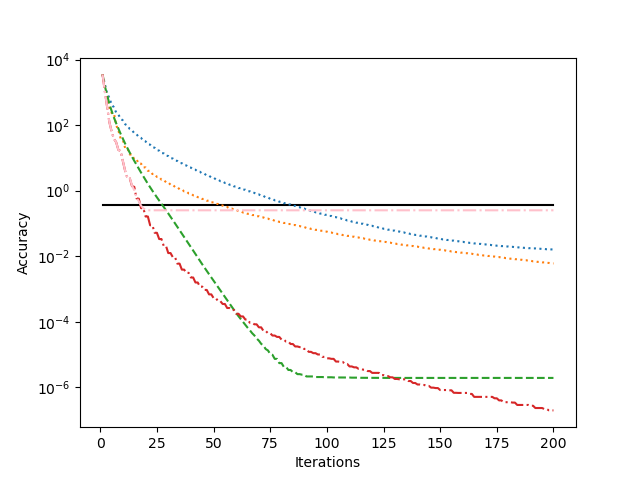}
	\caption*{(a) $(p_1,p_2)$=(1.5,1.8)}
\end{minipage}
\begin{minipage}[t]{0.31\linewidth}
	\centering
	\includegraphics[width=\linewidth]{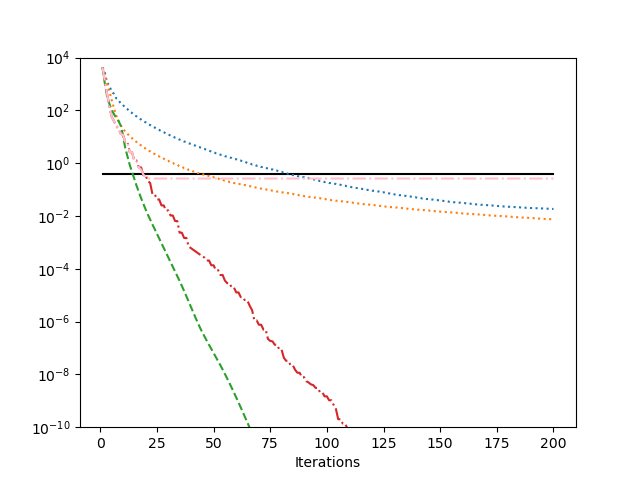}
	\caption*{(b) $(p_1,p_2)$=(2,2)}
\end{minipage}
\begin{minipage}[t]{0.365\linewidth}
	\centering
	\includegraphics[width=\linewidth]{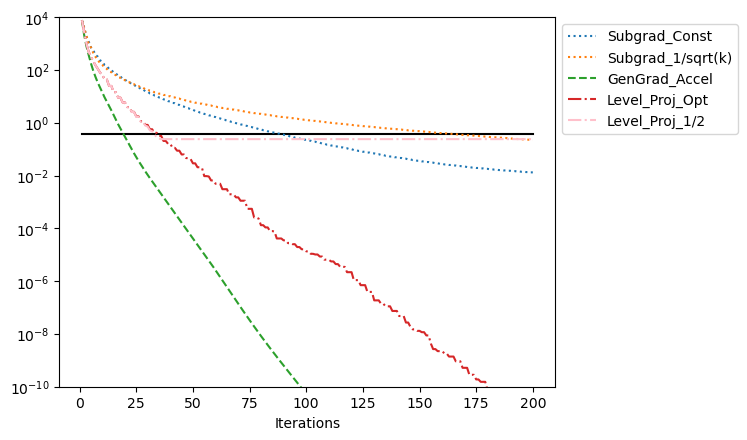}
	\caption*{\hskip-0.75cm (c) $(p_1,p_2)$=(3,4)}
\end{minipage}
\caption{The minimum accuracy of $\max_i\{\frac{1}{2}\gamma_{S_i,e_i}^2(x_k)\} - \max_i\{\frac{1}{2}\gamma_{S_i,e_i}^2(x^*)\}$ of finding the intersection point of two smooth/strongly convex $p$-norm ellipsoids. Note that all points below the black line at objective value $1/2$ correspond to feasible solutions.}
\label{fig:feasbility}
\end{figure}

The results are shown in Figure~\ref{fig:feasbility}. The solid black line denotes objective value one. So all solutions under this line correspond to the method having found a feasible point. In the strongly convex setting of Figure~\ref{fig:feasbility}(a), we can see that the level method with optimal value performs best. The subgradient method with constant stepsize and the accelerated generalized gradient method all perform reasonably up to a constant accuracy, after which nonsmoothness prevents further progress. For the smooth and strongly convex case shown in Figure~\ref{fig:feasbility}(b), the accelerated generalized gradient method and level method converges quickly while the subgradient methods are much slower. In the non-strongly convex smooth case, Figure~\ref{fig:feasbility}(c) looks similar to Figure~\ref{fig:feasbility}(b), with each method being somewhat slower. These numerics match our theory summarized in Table~\ref{tab:rates} except for the level method converges at a faster linear rate in smooth settings.

\subsection{Application to Trust Region Optimization Problems}\label{subsec:numerics2}
Finally, we compare the best of the previous first-order methods (the accelerated generalized gradient method) against standard solvers for synthetic constrained minimization problems. We consider constrained quadratic minimization problems over a $p$-norm ellipsoidal constraint set
\begin{equation} \label{eq:quad-SOCP}
    \begin{cases} \max_x & f(x) \\ \text{s.t. } & g(x) \leq 0
    \end{cases}=
    \begin{cases} \max_x &  1-\frac{1}{2}x^TQx -c^Tx  \\ \text{s.t. } & \|Ax-b\|^2_p\leq 1   \end{cases}
\end{equation}
for positive semidefinite $Q\in\mathbb{R}^{n\times n}$, and generic $A\in\mathbb{R}^{m\times n}, b\in\mathbb{R}^m, c\in\mathbb{R}^n$. For our numerics, we generate $Q, c, A, x_{feas}, \epsilon$ with standard normal entries and set $b=Ax_{feas}+\frac{1}{m} \epsilon$.
Each solver will be directly applied to~\eqref{eq:quad-SOCP}, first with $p=2$ (making this a simple second-order cone program) and then with $p=4$ (making this a more general conic program). 

For our projection-free method, we apply the accelerated generalized gradient method to the square of the radial reformulation~\eqref{eq:OPT-Gauge-Problem} to benefit from any smoothness and/or strong convexity present in the constraints.
(For this radial transformation to be well-defined, we need the origin to be feasible. To ensure this, we compute $e$ by approximately solving $Ax=b$ via $30$ iteration of the conjugate gradient method and then translate the problem to place $e$ at the origin.)
As in Section~\ref{subsec:numerics1}, we tune any algorithmic parameters (e.g., $\eta,L,\mu$) for each considered method to optimize final performance to two significant digits.

We compare our approach with six alternatives. We compare with the default configurations of Gurobi~\cite{gurobi} and Mosek~\cite{mosek} (second-order solvers) and of COSMO~\cite{cosmo} and SCS~\cite{scs} (first-order solvers). Further, we compare against two simple alternative first-order methods that utilize the functional constraint form of~\eqref{eq:quad-SOCP}: First, we consider a ``switching subgradient method'', which at each step either takes a gradient step on the objective or constraint function depending on whether $x_k$ is feasible, defined as
\begin{equation}
    x_{k+1} = x_k - \alpha_{k} \zeta_k, \ \zeta_k \in \begin{cases}  -\nabla f(x_k) & \text{if } g(x_k) \leq 0    \\ \nabla g(x_k) & \text{else,}          \end{cases}
\end{equation}
where stepsizes are tuned for each of the two cases of $\alpha_k = \begin{cases} \eta_1 / \sqrt{k+10} & \text{if } g(x_k) \leq 0\\ \eta_2 / \sqrt{k+10} & \text{else} \end{cases}$. Second, we consider directly maximizing the problem's Lagrangian $\mathcal{L}(x) = f(x) - \lambda^*g(x)$ assuming the dual optimal Lagrange multiplier $\lambda^*$ is given in advance. Then we compare with the following ``ideal Lagrangian accelerated method''
\begin{equation}
    \begin{cases}
        y_{k+1}=x_k-\frac{1}{L}\nabla \mathcal{L}(x_k)  \\
        x_{k+1}=y_{k+1} - \beta_k(y_{k+1}-y_k) \ . 
    \end{cases} 
\end{equation}


\begin{figure}[t]
\begin{minipage}[t]{0.31\linewidth}
        \captionsetup{justification=centering}
        \centering
	\includegraphics[width=\linewidth]{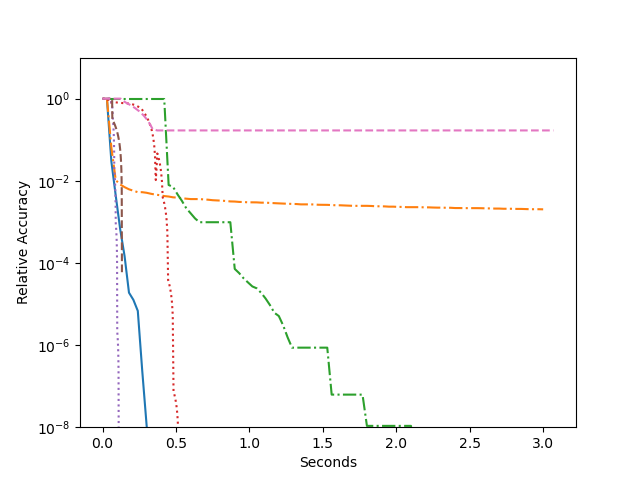}
	\caption*{(a) $(n,m)=(400,200)$}
\end{minipage}
\begin{minipage}[t]{0.31\linewidth}
        \captionsetup{justification=centering}
	\centering
	\includegraphics[width=\linewidth]{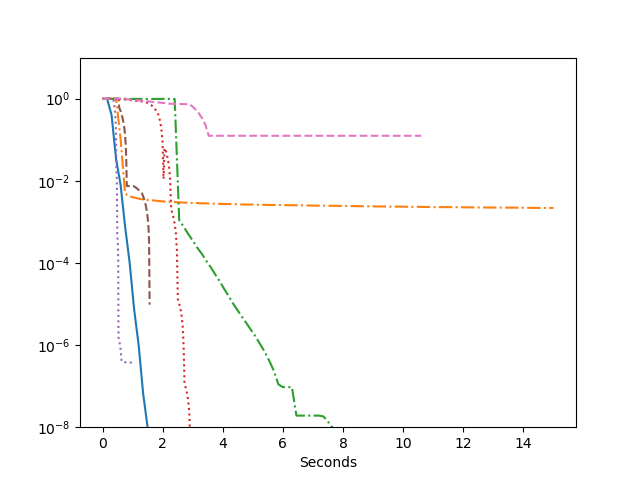}
        \caption*{(b) $(n,m)=(800,400)$}
\end{minipage}
\begin{minipage}[t]{0.36\linewidth}
        \captionsetup{justification=centering}
	\centering
	\includegraphics[width=\linewidth]{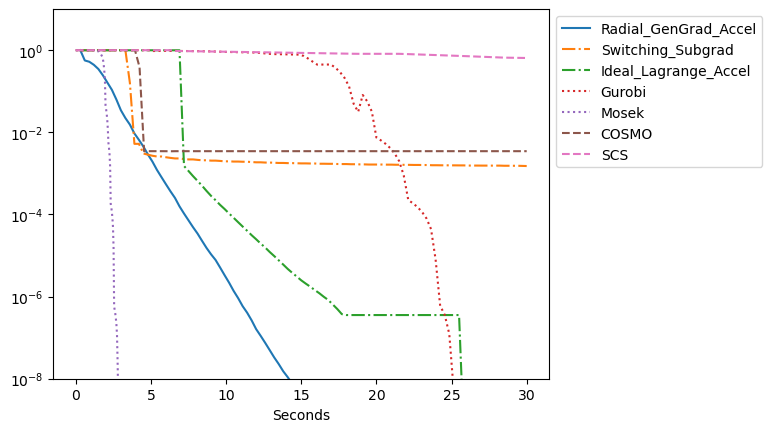}
        \caption*{\hskip-0.75cm (c) $(n,m)=(1600,800)$}
\end{minipage}
\caption{The minimum relative accuracy $|f(x_k)-f(x^*)|/|f(x_0)-f(x^*)|$ of~\eqref{eq:quad-SOCP} for varied problem dimensions $(n,m)$ over 3 seconds, 15 seconds, 30 seconds, respectively.}
\label{fig:ellips_cons}
\end{figure}



\noindent {\bf $p=2$-norm Ellipsoid Constrained Optimization.}
First, we consider the problem of quadratic optimization over a $p=2$-norm ellipsoid constraint, giving a smooth and strongly convex objective function and constraints,
\begin{equation} \label{eq:quad-ellip}
    \begin{cases} \max_x & 1 - \frac{1}{2}x^TQx - c^Tx \\ \text{s.t. } &\|Ax-b\|^2_2\leq 1 \end{cases}
\end{equation}
which has an equivalent unconstrained radially dual problem \footnote{This is computed directly by definition, refer to~\cite{grimmer2023radial2} Section 2.1.}
\begin{equation}
    \min_y \max\left\{ \frac{c^Ty +1 +\sqrt{(c^Ty+1)^2+2y^TQy}}{2}, \frac{-b^TAy + \sqrt{(b^TAy)^2-4(1-\|b\|_2^2)\|Ay\|^2_2)}}{1-\|b\|_2^2} \right\} \ .
\end{equation}

Note the second term above is exactly the gauge $\gamma_{\{x \mid \|Ax-b\|_2\leq 1\}}(y)$. As before, the smoothness and strong convexity of the objective and ellipsoid constraint correspond to both terms above being smooth and strongly convex when squared. As a result, this can be solved at a linear rate by an accelerated projection-free method.

Figure~\ref{fig:ellips_cons} shows the result of each method applied to sampled problem instances with dimensions $(n,m)=(400,200), (800,400), (1600,800)$. The two second-order solvers have high iteration costs but converge quickly once they start making progress. In this setting, our radial accelerated method is relatively competitive with Gurobi and COSMO, but always outperformed by Mosek.

\noindent {\bf $p=4$-norm Ellipsoid Constrained Optimization.}
Lastly, we consider a ``harder'' quadratic optimization problem with a quartic ellipsoid constraint, making the constraint set smooth but not strongly convex,
\begin{equation} \label{eq:quart}
    \begin{cases} \max_x & 1 - \frac{1}{2}x^TQx - c^Tx \\ \text{s.t. } &\|Ax-b\|^2_4\leq 1 \end{cases}
\end{equation}
which has an equivalent unconstrained radially dual problem
\begin{equation} \label{eq:quart_rad}
    \min_y \max\left\{ \frac{c^Ty +1 +\sqrt{(c^Ty+1)^2+2y^TQy}}{2}, \gamma_{S}(y) \right\} \ .
\end{equation}
where $S = \{x \mid \|Ax-b\|_4 \leq 1\}$. The gauge $\gamma_{S}$ has a closed form given by quartic formula, and only costs one matrix-vector multiplication with $A$ to evaluate. In this case, the two terms in~\eqref{eq:quart_rad} are only guaranteed to have smoothness when squared.

\begin{figure}[t]
\begin{minipage}[t]{0.31\linewidth}
        \captionsetup{justification=centering}
        \centering
	\includegraphics[width=\linewidth]{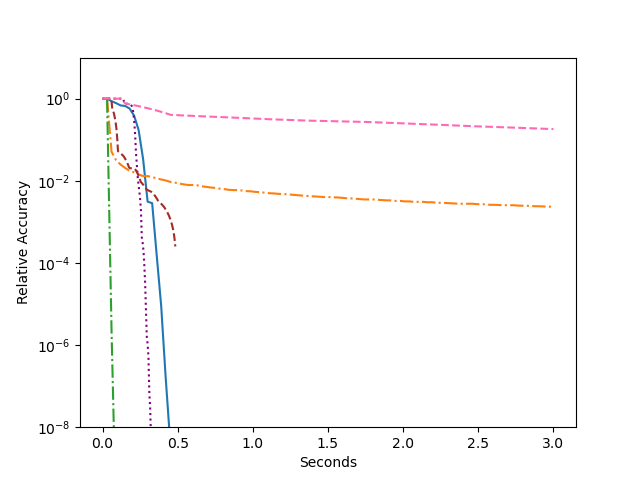}
	\caption*{(a) $(n,m)=(400,200)$}
\end{minipage}
\begin{minipage}[t]{0.31\linewidth}
        \captionsetup{justification=centering}
	\centering
	\includegraphics[width=\linewidth]{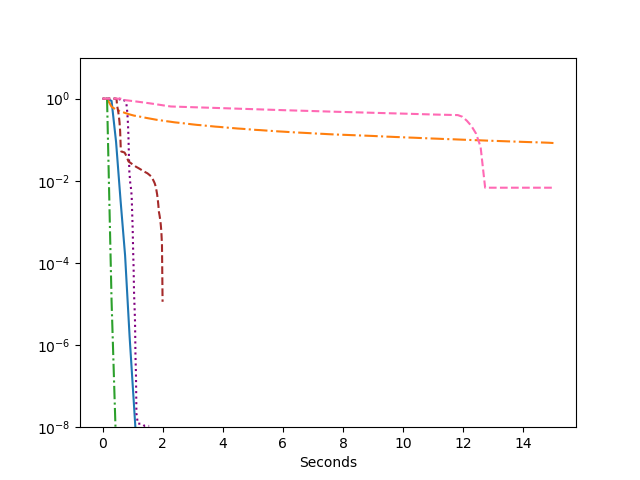}
        \caption*{(b) $(n,m)=(800,400)$}
\end{minipage}
\begin{minipage}[t]{0.36\linewidth}
        \captionsetup{justification=centering}
	\centering
	\includegraphics[width=\linewidth]{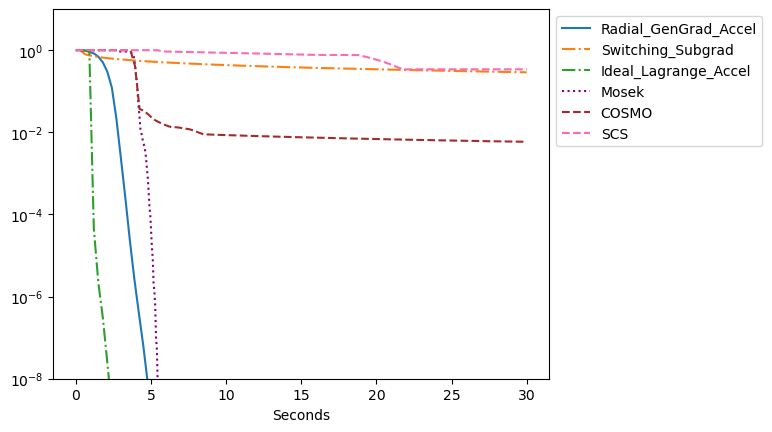}
        \caption*{\hskip-0.75cm (c) $(n,m)=(1600,800)$}
\end{minipage}
\caption{The minimum relative accuracy $|f(x_k)-f(x^*)|/|f(x_0)-f(x^*)|$ of~\eqref{eq:quart} for varied problem dimensions $(n,m)$ over 3 seconds, 15 seconds, 30 seconds, respectively.}
\label{fig:quartic_cons}
\end{figure}


Figure~\ref{fig:quartic_cons}\footnote{Gurobi is limited to linear and quardratic programs, it is not able to run on $p=4$-norm ellipsoids.} shows the first-order solvers have their performance fall off as the problem size grows.  The second-order solver's performance is matched by our radial method at all problem sizes, increasingly so as the problem dimension grows. This motivates future works, building more practical radial methods and testing their performance on real (non-synthetic) problem instances. Unlike the prior $p=2$ setting, the idealized Lagrangian accelerated method outperforms our radial method by a moderate factor (between two and five). Given this method's benefited from knowledge of the optimal Lagrange multiplier, this is not unexpected.

    \section{Conclusion} \label{sec:concl}

In this paper, we showed that $\alpha$-strongly convex and/or $\beta$-smooth sets always have $O(\alpha)$-strongly convex and/or $O(\beta)$-smooth gauge squared. As a result, feasibility and optimization problems over structured sets can be recast as structured optimization problems minimizing these gauges squared. To benefit from this, we proposed fast first-order methods targeting such squared strongly convex/smooth properties. Consequently, we derived accelerated convergence guarantees of $O(1/T)$ for problems over strongly convex sets, $O(1/T^2)$ for smooth sets, and accelerated linear convergence given both.
Numerically, we find these methods are very effective in sample synthetic experiment settings. This indicates future developments of gauge-based methods may provide a competitive alternative to current solvers based on ADMM or interior point approaches.

Additionally, future works may be able to identify further relationships between structured sets and their gauges. We expect that notions of uniform smoothness/convexity of a set will correspond to a H\"older/uniform smoothness of the set's gauge raised to an appropriate power.
Such analysis may explain the linear convergence we observe for our radial method in settings with $p=4$ norms (where strong convexity does not hold, but some uniform convexity notion may hold).




\section*{Statements and Declarations}

\paragraph{Acknowledgements}
The authors want to thank three anonymous referees and the associated editor for invaluable critical feedback, which much improved this work. Additionally, we thank Michael P.~Friedlander for their early comments on this work, which led us to improve the numerical comparisons using \texttt{JuMP.jl} instead of \texttt{Convex.jl}. 

\paragraph{Competing Interests and Funding}
This work does not have any funding to disclose or any conflicts of interest.

\bibliography{sn-bibliography}

\begin{appendices}
{\small
    \section{} \label{sec:append}
\subsection{Explanation of Examples~\ref{subsec:ex}} \label{append:ex}
A useful lemma of the equivalence of $\ell_p$-norm in finite dimensions ~\cite[Proposition 6.12]{folland1999real} is given below
\begin{lemma} \label{lem:norm-embeddings}
    For $x \in \mathbb{R}^n$, if $0< p <q \leq \infty$, then $\|x\|_q \leq \|x\|_p \leq n^{1/p-1/q}\|x\|_q$. 
\end{lemma}
Here we carry out detailed calculations supporting our claimed strong convexity and smoothness constants for $p$-norm balls $B_p=\{x \mid \|x\|_p\leq 1\}$, more generally translated $p$-norm balls $T_p=\{x \mid \|x-b\|_p\leq 1\}$, and their gauges, in $\mathbb{R}^n$.

Let $f(x) = \|x-b\|_p^p$ so that $T_p=\{x \mid f(x)\leq 1\}$. Here we compute the smoothness and strong convexity at any point $\bar y\in\bdry T_p$. Without loss of generality, $\bar y-b\geq 0$. Then $f$ has gradient and Hessian given by
\begin{align*}
    \nabla f(\bar y) &= p (\bar y-b)^{p-1} \ , \\
    \nabla^2 f(\bar y) &= p(p-1)\mathrm{diag}((\bar y-b)^{p-2})
\end{align*}
applying powers to vectors elementwise. Applying Lemma~\ref{lemma:level}, $T_p$ then must  be locally $\alpha(\bar y)$-strongly convex and $\beta(\bar y)$-smooth such that
\[
\text{for }p\le 2 \quad \alpha(\bar y)=\frac{p(p-1)}{\|\nabla f(\bar y)\|_2\,\|\bar y-b\|_\infty^{\,2-p}},\qquad
\text{while for }p\ge 2\quad
\beta(\bar y)=\frac{p(p-1)\|\bar y-b\|_\infty^{\,p-2}}{\|\nabla f(\bar y)\|_2}.
\]
Using $\|\bar y-b\|_\infty\le 1$ and Lemma~\ref{lem:norm-embeddings} gives the pointwise bounds
\[
\alpha(\bar y)\ \ge\ \alpha_*:=(p-1)\,n^{\frac12-\frac1p},\qquad
\beta(\bar y)\ \le\ \beta^*:=(p-1)\,n^{\frac12-\frac1p}.
\]
Our main theorems then give local smoothness and strong convexities for $\frac{1}{2}\gamma_{B_p}^2(y) = \frac{1}{2}\|y\|^2_p$. Computing the inf/sup of these quantities over all boundary points $\bar y$ gives the following strong convexity and smoothness constants
\begin{align*}
    \mu &= \inf_{\|\bar y\|_p=1} \frac{\|\nabla f(\bar y)\|^3}{2(\nabla f(\bar y)^T\bar y)^3}\left( \frac{\nabla f(\bar y)^T\bar y}{\|\nabla f(\bar y)\|}+\alpha(\bar y)\|\bar y\|^2-\sqrt{\left(\frac{\nabla f(\bar y)^T\bar y}{\|\nabla f(\bar y)\|}+\alpha(\bar y)\|\bar y\|^2\right)^2-4\alpha(\bar y)\left(\frac{\nabla f(\bar y)^T\bar y}{\|\nabla f(\bar y)\|}\right)^3}\right) \ , \\
    L &= \sup_{\|\bar y\|_p=1} \frac{\|\nabla f(\bar y)\|^3}{2(\nabla f(\bar y)^T\bar y)^3}\left( \frac{\nabla f(\bar y)^T\bar y}{\|\nabla f(\bar y)\|}+\beta(\bar y)\|\bar y\|^2+\sqrt{\left(\frac{\nabla f(\bar y)^T\bar y}{\|\nabla f(\bar y)\|}+\beta(\bar y)\|\bar y\|^2\right)^2-4\beta(\bar y)\left(\frac{\nabla f(\bar y)^T\bar y}{\|\nabla f(\bar y)\|}\right)^3}\right)\ .
\end{align*}

For the case of $p$-norm balls (i.e., $b=0$), numerically computing these bounds shows the simple formulas of $\mu=L=(p-1)n^{\frac{1}{2}-\frac{1}{p}}$ are not tight. For example, when $n=2,p=3$, the smoothness constant above has approximate value $L\approx 2.1424$ whereas $(p-1)n^{\frac{1}{2}-\frac{1}{p}} = 2^{7/6}\approx 2.2449$.

For generic translated balls, here we provide relatively simple lower and upper bounds for $\mu$ and $L$ above. 

\noindent\emph{Strong convexity.}
By the gauge strong convexity bound (the same inequality as in~\eqref{eq:gauge-strong-convexity-bound}),
\[
\mu\ \ge\ \inf_{\|x-b\|_p=1}\ \frac{\alpha(x)}{\frac{\nabla f(x)^T x}{\|\nabla f(x)\|_2}+\alpha(x)\|x\|_2^2}.
\]
Using $\frac{\nabla f(x)^T x}{\|\nabla f(x)\|_2}\le \|x\|_2$ (Cauchy–Schwarz) and the monotonicity of
$\alpha\mapsto \frac{\alpha}{A+\alpha t}$ for $A,t\ge 0$, we obtain
\[
\mu\ \ge\ \inf_{\|x-b\|_p=1}\ \frac{\alpha_*}{\|x\|_2+\alpha_*\,\|x\|_2^2}.
\]
Finally, $\|x\|_2\le \|x-b\|_2+\|b\|_2\le \|x-b\|_p+\|b\|_2=1+\|b\|_2$ yields the explicit bound
\begin{equation}\label{eq:translated-mu}
\quad
\mu\ \ge\ \frac{\alpha_*}{(1+\|b\|_2)\ +\ \alpha_*(1+\|b\|_2)^2}
\ =\ \frac{(p-1)\,n^{\frac12-\frac1p}}{(1+\|b\|_2)\ +\ (p-1)\,n^{\frac12-\frac1p}(1+\|b\|_2)^2}\ .
\quad
\end{equation}
By Lemma~\ref{lem:struc_trans}, for $E_p=\{y:\ Ay-b\in B_p\}$ we get
\[
\text{$\frac12\gamma_{E_p}^2$ is }\quad
\lambda_{\min}(A^TA)\times \mu\text{-strongly convex,}
\]
so~\eqref{eq:translated-mu} multiplied by $\lambda_{\min}(A^TA)$ provides a global constant.

\medskip
\noindent\emph{Smoothness.}
From the gauge smoothness bound (the same inequality as in~\eqref{eq:gauge-smooth-bound}), on $\partial T_p$,
\[
L\ \le\ \sup_{\|x-b\|_p=1}\ \frac{\frac{\nabla f(x)^T x}{\|\nabla f(x)\|_2}+\beta(x)\|x\|_2^2}{\left(\frac{\nabla f(x)^T x}{\|\nabla f(x)\|_2}\right)^3}.
\]
Using $\|x-b\|_\infty\le 1$, $\|x-b\|_{2(p-1)}\le 1$, Lemma~\ref{lem:norm-embeddings}, and convexity
$\nabla f(x)^Tx\ge f(x)-f(0)=1-\|b\|_p^p$, we get
\begin{equation}\label{eq:translated-L}
\quad
L\ \le\ \frac{p^2\,(1-\|b\|_p^p)\ +\ p^3(p-1)\,n^{\,1-\frac{2}{p}}\,\big(1+\|b\|_p\big)^2}{(1-\|b\|_p^p)^3}\ .
\quad
\end{equation}
If we only want a simpler expression under the standing assumption $\|b\|_p<1$, then
$(1+\|b\|_p)^2\le 4$ and $1-\|b\|_p^p \leq 1 \leq \frac{1}{2}p(p-1)n^{1-\frac{2}{p}}$, hence~\eqref{eq:translated-L} reduces to
\[
L\ \le\ \frac{p^2\, \frac{1}{2}p(p-1)n^{1-\frac{2}{p}}\ +\ 4p^3(p-1)\,n^{\,1-\frac{2}{p}}}{(1-\|b\|_p^p)^3}\ = \frac{9}{2}\left(\frac{p}{1-\|b\|_p^p}\right)^3(p-1)n^{1-\frac{2}{p}}.
\]
Again, Lemma~\ref{lem:struc_trans} multiplies this by $\lambda_{\max}(A^TA)$ for the set $E_p$.

\subsection{Deferred Proofs of Lemmas and Propositions}

\subsubsection{Proof of Proposition~\ref{prop:smooth_func}} \label{proof:smooth_func}
The equivalence $\mathbf{(a)\Leftrightarrow(b)}$ is given by the equivalence in~\cite[Theorem 2.1.5 $\mathbf{(2.1.9)\Leftrightarrow (2.1.11)}$]{nesterov2018lectures}, and the equivalence $\mathbf{(a)\Leftrightarrow(c)}$ is given by the equivalence in~\cite[Proposition 12.60 $\mathbf{(e)\Leftrightarrow(d)}$]{rockafellar2009variational}.  The implication $\mathbf{(a) \Rightarrow (d)}$ is trivial as global smoothness implies local smoothness. The final implication $\mathbf{(d) \Rightarrow (b)}$ is nontrivial:

Assuming local smoothness $\mathbf{(d)}$, fix any $\Tilde{L}>L$. Our proof of $\mathbf{(a)}$ relies on the following claim.
\begin{claim} \label{claim:point_lip}
    If $f$ is globally $\Delta$-smooth and locally $\Tilde{L}$-smooth everywhere, then $f$ is globally $\left(\frac{1}{2\Delta}+\frac{1}{2\tilde{L}}\right)^{-1}$-smooth.\footnote{Both $\Delta$ and $\tilde {L}$ can be $\infty$.}
\end{claim}
Note, vacuously, $f$ is $\Delta_0=\infty$-smooth globally. Then inductively applying the claim, $f$ must be $\Delta_k$-smooth for every $k$ where $\Delta_k$ satisfies $\frac{1}{\Delta_{k+1}}=\frac{1}{2\Delta_k}+\frac{1}{2\Tilde{L}}$. This recurrence with initial condition $\Delta_0=\infty$, ensures $\frac{1}{\Delta_{n}}=\frac{1}{\Tilde{L}}-\frac{1}{2^{n}}\frac{1}{\Tilde{L}}$.
Hence $\Delta_k \rightarrow \Tilde{L}$ and so $f$ is globally $\Tilde{L}$-smooth. Thus it must be globally $L$-smooth. All that remains is to prove the claim.

\noindent {\bf Proof of Claim~\ref{claim:point_lip}.}
We say $f$ has $L$-Lipschitz gradient at $x\in\mathcal{E}$ if for some $\eta>0$, every $y\in B(x,\eta)$ has $\|\nabla f(y)-\nabla f(x)\| \leq L\|y-x\|$. We say $f$ has pointwise $L$-Lipschitz gradient if it has $L$-Lipschitz gradient at every $x\in\mathcal{E}$. Note as a simple result, having pointwise $L$-Lipschitz gradient implies $f$ has uniformly $L$-Lipschitz gradient~\cite[Corollary 2.4]{durand2010pointwise}. Hence, our proof aims to show $f$ has pointwise $\left(\frac{1}{2\Delta}+\frac{1}{2\Tilde{L}}\right)^{-1}$-Lipschitz gradient. To that end, consider some $x\in\mathcal{E}$ and let $\eta>0$ be such that condition $\mathbf{(d)}$ holds at $x$.

Consider any $y\in B(x, \eta/2)$ and $z\in B(x, \eta)$.
First, note by assuming $f$ is $\Delta$-smooth, we have the following standard inequality
\begin{equation} \label{eq:main_bound_1}
f(y) \geq f(x)+\nabla f(x)^T(y-x)+\frac{1}{2\Delta}\|\nabla f(y)-\nabla f(x)\|^2 \ .
\end{equation}
Second, note that convexity and the local $\Tilde{L}$-smoothness of $f$ ensures
\begin{equation}
f(y)+\nabla f(y)^T(z-y) \leq f(z) \leq f(x)+\nabla f(x)^T(z-x)+\frac{\Tilde{L}}{2}\|z-x\|^2 \ .  \label{eq:convex_and_local}
\end{equation}
Considering $z=x+\eta \frac{\nabla f(y)-\nabla f(x)}{\|\nabla f(y)-\nabla f(x)\|}$, convexity and~\eqref{eq:convex_and_local} implies
$$ \nabla f(x)^T(y-x) \leq f(y) -f(x) \leq \nabla f(y)^T(y-x) -\eta\|\nabla f(y)-\nabla f(x)\| +\frac{\Tilde{L}\eta^2}{2}\ . $$
Therefore $\|\nabla f(y)-\nabla f(x)\| \leq \Tilde{L}\eta$ since $(\nabla f(y)-\nabla f(x))^T(y-x) \leq \|\nabla f(y)-\nabla f(x)\| \eta/2$ by Cauchy-Schwarz. Thus we can consider $z=x+(\nabla f(y)-\nabla f(x))/\tilde{L} \in B(x,\eta)$. Then~\eqref{eq:convex_and_local} implies
\begin{equation} \label{eq:main_bound_2}
    f(y) \leq f(x) +\nabla f(y)^T(y-x) -\frac{1}{2\Tilde{L}}\|\nabla f(y)-\nabla f(x)\|^2 \ . 
\end{equation}
Combining~\eqref{eq:main_bound_1} and~\eqref{eq:main_bound_2} gives 
$(\nabla f(y)-\nabla f(x))^T(y-x)\geq \left(\frac{1}{2\Delta}+\frac{1}{2\Tilde{L}}\right)\|\nabla f(y)-\nabla f(x)\|^2 . $
Again applying Cauchy-Schwarz gives the desired Lipschitz gradient condition at $x$, from which we conclude $f$ has $\left(\frac{1}{2\Delta}+\frac{1}{2\Tilde{L}}\right)^{-1}$-Lipschitz gradient globally.

\subsubsection{Proof of Proposition~\ref{prop:sc_func}} \label{proof:sc_func}
The equivalence $\mathbf{(a)\Leftrightarrow(b)}$ is given by the equivalence in~\cite[Theorem 5.24 $\mathbf{(ii) \Leftrightarrow (iii)}$]{beck2017first}.
The equivalence $\mathbf{(a)\Leftrightarrow(c)}$ follows by combining a few standard results. Note $h$ being $\mu$-strongly convex is equivalent to the Fenchel conjugate $h^*$ being $\frac{1}{\mu}$-smooth by~\cite[Theorem 12.60 $\mathbf{(a)\Leftrightarrow(b)}$]{rockafellar2009variational}. Moreover, $h^*$ is $\frac{1}{\mu}$-smooth if and only if there exists a closed convex function $g$ such that $g+h^*=\frac{1}{2\mu}\|\cdot\|^2$ by~\cite[Lemma 4 $\mathbf{[1]\Leftrightarrow[2]}$]{zhou2018fenchel}. This statement about sums of conjugates is equivalent to the existence of $g^*$ with the following inf-convolution $g^*\ \square \ h=\frac{\mu}{2}\|\cdot\|^2$, by~\cite[Theorem 11.23(a)]{rockafellar2009variational}.

Finally, we show the equivalence with the local strong convexity condition $\mathbf{(d)}$. The implication $\mathbf{(a)\Rightarrow(d)}$ is trivial as global strong convexity implies local strong convexity. Again the local to global implication is nontrivial: We leverage our previous local to global result for smooth functions from Proposition~\ref{prop:smooth_func}(d). From this, it suffices to show that $f$ being locally $\Tilde{\mu} <\mu$-strongly convex implies $f^*$ is locally $1/\Tilde{\mu}$-smooth since one can then conclude $f^*$ is $1/\mu$-smooth and hence $f$  is $\mu$-strongly convex by~\cite[Theorem 12.60 $\mathbf{(a)\Leftrightarrow(b)}$]{rockafellar2009variational}. 

Consider any $x\in\mathcal{E}$, $g\in \partial f(x)$, $\Tilde{\mu} < \mu$, and $\eta>0$ such that the local strong convexity condition $\mathbf{(d)}$ holds w.r.t~$(x,g)$. For $z\in B(x,\eta)$, local strong convexity of $f$ gives the following lower bound
$$ f(z) \geq h(z) := f(x)+g^T(z-x)+\frac{\Tilde{\mu}}{2}\|z-x\|^2 \ . $$
For $z\not\in B(x,\eta)$, convexity of $f$ then gives the following lower bound 
$$ f(z) \geq l(z) := f(x)+g^T(z-x)+\frac{\Tilde{\mu} \eta}{2}\|z-x\| $$
since $\frac{\eta}{\|z-x\|} f(z) + (1-\frac{\eta}{\|z-x\|})f(x) \geq f(x + \frac{\eta(z-x)}{\|z-x\|}) \geq h(x + \frac{\eta(z-x)}{\|z-x\|})$. Hence $f \geq \min\{h,l\}$ and so $f^* \leq \max\{h^*,l^*\}$. Computing these two conjugates shows this is exactly local $1/\Tilde{\mu}$-smoothness of $f^*$ w.r.t.~$(g,x)$ with radius $\Tilde{\mu}\eta/2$ as needed:
\begin{align*}
    h^*(w) &= f^*(g)+x^T(w-g)+\frac{1}{2\Tilde{\mu}}\|w-g\|^2 \ , \\
    l^*(w) &= f^*(g)+x^T(w-g)+\iota_{B}(w)
\end{align*}
where $\iota_B(x)=\begin{cases}0 & \text{ if } x\in B\\ +\infty & \text{ if } x \notin B \end{cases}$ denotes the indicator for the ball $B(g,\Tilde{\mu}\eta/2)$.

\subsubsection{Proof of Proposition~\ref{prop:smooth_set}}\label{pf:smooth_set}
 $\mathbf{(a) \Rightarrow (b)}$.
For $x,y \in \bdry S$ and unit normal vectors $\zeta_x \in N_S(x), \zeta_y \in N_S(y)$, $\beta$-smoothness ensures $B(x-\frac{1}{\beta}\zeta_x, \frac{1}{\beta}) \subseteq S$ and $B(y-\frac{1}{\beta}\zeta_y, \frac{1}{\beta}) \subseteq S$. Note $x-\frac{1}{\beta}\zeta_x+\frac{1}{\beta}\zeta_y \in B(x-\frac{1}{\beta}\zeta_x, \frac{1}{\beta})$ and $y-\frac{1}{\beta}\zeta_y+\frac{1}{\beta}\zeta_x \in B(y-\frac{1}{\beta}\zeta_y, \frac{1}{\beta})$. Then it follows from the definition of normals that
\begin{align*}
    \zeta_x^T\left(\left(y-\frac{1}{\beta}\zeta_y+\frac{1}{\beta}\zeta_x\right)-x\right) \leq 0 \ , \\
    \zeta_y^T\left(\left(x-\frac{1}{\beta}\zeta_x+\frac{1}{\beta}\zeta_y\right)-y\right) \leq 0 \ . 
\end{align*}
Summing these inequalities yields $\|\zeta_x-\zeta_y\|^2 \leq \beta(\zeta_x-\zeta_y)^T(x-y)$.

\noindent $\mathbf{(b) \Rightarrow (a)}$.
Fix $x\in\bdry S$ and a unit $\zeta_x\in N_S(x)$, and set $c:=x-\frac1\beta \zeta_x$.
We prove $B(c,1/\beta)\subseteq S$ by contradiction. Supposing not, then there exists
$y\in\bdry S$ with $d:=\|c-y\|<1/\beta$. Choose $y$ to be a nearest boundary point to $c$
(i.e., $y\in\arg\min_{v\in\bdry S}\|v-c\|$). By optimality and convexity of $S$,
\[
(z-y)^T(c-y)\leq 0\quad\forall z\in S\qquad\Rightarrow\qquad \zeta_y:=\frac{c-y}{\|c-y\|}\in N_S(y),
\]
and $\zeta_y$ is unit. Apply (b) to $x,y$ and $\zeta_x,\zeta_y$:
\[
2(1-\zeta_x^T\zeta_y)\ \le\ \beta\big(\zeta_x-\zeta_y\big)^T(x-y)
= \beta\Big(\tfrac1\beta+\zeta_x^T(c-y)-\tfrac1\beta\,\zeta_x^T\zeta_y-d\Big)
= (1-\zeta_x^T\zeta_y)+\beta\big(\zeta_x^T(c-y)-d\big).
\]
Hence
\begin{equation}\label{eq:key}
1-\zeta_x^T\zeta_y \leq \beta\big(\zeta_x^T(c-y)-d\big).
\end{equation}
By Cauchy--Schwarz, $\zeta_x^T(c-y)\le\|c-y\|=d$, so the right-hand side of \eqref{eq:key} is
$\le 0$. Therefore $1-\zeta_x^T\zeta_y\le 0$, i.e., $\zeta_y=\zeta_x$, and
equality must hold in Cauchy--Schwarz, so $\zeta_x^T(c-y)=d$.

But $\zeta_y=\zeta_x\in N_S(x)\cap N_S(y)$ implies $\zeta_x^T x=\sigma_S(\zeta_x)=\zeta_x^T y$.
Thus
\[
\zeta_x^T(c-y)=\zeta_x^T(x-\tfrac1\beta\zeta_x-y)
= \underbrace{\zeta_x^T x-\zeta_x^T y}_{=\,0}-\tfrac1\beta=-\tfrac1\beta,
\]
which contradicts $\zeta_x^T(c-y)=d>0$.

\noindent $\mathbf{(a) \Rightarrow (c)}$.
Define the Minkowski Difference of two sets as $A \ominus B=\{x \in \mathcal{E}: x+B \subseteq A\}$ and the support function of a set as $\sigma_S(a)=\sup _{x \in S} a^Tx$. Note for closed convex sets $A,B$, the support function has
\begin{align} 
\sigma_{A+B} &= \sigma_{A}+\sigma_{B}\ , \label{sup_sum}\\
A \subseteq B &\Longleftrightarrow \sigma_A \leq \sigma_{B}\ . \label{sup_subset}
\end{align}
Letting $S_0=S \ominus B(0, \frac{1}{\beta})$, it is immediate that $S_0+B(0,\frac{1}{\beta}) \subseteq S$. Then we need to show $S_0+B(0,\frac{1}{\beta}) \supseteq S$. By properties~\eqref{sup_sum} and~\eqref{sup_subset}, it suffices to show $\sigma_{S_0}(d)+\sigma_{B(0,\frac{1}{\beta})}(d) \geq \sigma_{S}(d)$ for every $d \in \mathcal{E}$.

Suppose $\sigma_S(d)$ is finite. Then, there exists a sequence $\{x_i\}$ with $x_i \in \bdry S$ such that $\lim_{i \rightarrow \infty} d^Tx_i =\sigma_S(d)$. For each $x_i$, letting $\zeta_i \in N_S(x_i)$ be some unit normal vector, the assumed smoothness ensures $B(x_i-\frac{1}{\beta}\zeta_i, \frac{1}{\beta}) \subseteq S$. Hence $x_i-\frac{1}{\beta}\zeta_i \in S\ominus B(0,\frac{1}{\beta})=S_0$. So
$$ \sigma_S(d) = \lim_{i \rightarrow \infty} d^Tx_i = \lim_{i \rightarrow \infty} d^T(x_i-\frac{1}{\beta}\zeta_i)+\frac{1}{\beta}d^T\zeta_i \leq \sigma_{S_0}(d)+\sigma_{B(0,\frac{1}{\beta})}(d)\ .
$$
Next, suppose $\sigma_S(d)=\infty$, then $S$ is unbounded. By~\cite[Proposition 1.4.2(a)]{bertsekas2009convex} there exists a ray $\{ z + t \Delta \mid t \geq 0\}\subseteq S$ with $d^T\Delta >0$. Letting $\zeta_z \in N_S(z)$ be a unit normal vector, we have $B(z-\frac{1}{\beta}\zeta_z, \frac{1}{\beta}) \subseteq S$. By convexity, one has $\cup_{t\geq 0} B(z-\frac{1}{\beta}\zeta_z+t\Delta, \frac{1}{\beta}) \subseteq S$. Consequently, $z-\frac{1}{\beta}\zeta_z+ t\Delta \in S_0$ for any $t\geq 0$. Thus
$$ \sigma_{S_0}(d) \geq \lim_{\lambda \rightarrow \infty} d^T(z-\frac{1}{\beta}\zeta_z+\lambda \Delta) =\infty = \sigma_S(d)\ .
$$



\noindent $\mathbf{(c) \Rightarrow (a)}$.
Let $S_0$ be a closed convex set such that $S_0+B(0,\frac{1}{\beta})=S$. Then for any $x \in \bdry S$ and unit $\zeta_x \in N_S(x)$,~\eqref{sup_sum} ensures $\sigma_{S_0}(\zeta_x)+\sigma_{B(0,\frac{1}{\beta})}(\zeta_x)=\sigma_S(\zeta_x)$.
So every $z\in S_0$ has $\zeta_x^Tz \leq \sigma_{S_0}(\zeta_x)=\zeta_x^Tx-\frac{1}{\beta}$. Since $x \in S$, there exists some $\bar z\in S_0$ such that $x\in \bar z+B(0,\frac{1}{\beta})$, which implies $\|x-\bar z\| \leq \frac{1}{\beta}$. Hence $\|x-\bar z -\frac{1}{\beta}\zeta_x\|^2=\|x-\bar z\|^2-\frac{2}{\beta}\zeta_x^T(x-\bar z)+\frac{1}{\beta^2} \leq 0$. It follows that $\|x-\bar z -\frac{1}{\beta}\zeta_x\|=0$, so $\bar z=x-\frac{1}{\beta}\zeta_x \in S_0$. Consequently, $(x-\frac{1}{\beta}\zeta_x)+B(0,\frac{1}{\beta})=B(x-\frac{1}{\beta}\zeta_x, \frac{1}{\beta}) \subseteq S$ and so $S$ is $\beta$-smooth.

\subsubsection{Proof of Lemma~\ref{lemma:epigraph}}    \label{pf: epigraph}
This result follows rather directly from relating the local quadratic upper and lower bounds from smoothness and strong convexity of $h$ to balls inside its epigraph. Note given any $g\in\partial h(x)$, $(\zeta,\delta) = (g,-1)/\|(g,-1)\|$ is a unit normal $(\zeta,\delta)\in N_{\mathrm{epi\ }h}(x,h(x))$. Further note $\|\zeta\|^2=1-\delta^2$. First, we address smoothness. Consider any $\tilde L > L$ and $\eta>0$ such that all $y\in B(x,\eta)$ have
\begin{equation}\label{eq:quad-bound}
     h(y) \leq q(y) := h(x) + g^T(y-x) + \frac{\tilde L}{2}\|y-x\|^2 \ .
\end{equation}
Local $\tilde \beta$-smoothness of $\epi h$ w.r.t. $((x,h(x)),(\zeta,\delta))$ corresponds to the following upper bound holding for all $y$ near $x$
\begin{equation}\label{eq:ball-bound}
    h(y) \leq b(y) := h(x) - \delta/\tilde \beta - \sqrt{1/\tilde \beta^2 - \|y-(x-\zeta/\tilde \beta)\|^2}\ .
\end{equation}
Noting $q(x) = b(x) = h(x)$ and $\nabla q(x) = \nabla b(x) = g$, the target inequality~\eqref{eq:ball-bound} follows from~\eqref{eq:quad-bound} if $\nabla^2 b(x)\succeq \nabla^2 q(x)$. These two Hessians can be computed as
\begin{align*}
    \nabla^2 q(x) &= \tilde L I \qquad \text{and} \qquad
    \nabla^2 b(x) = \frac{\tilde \beta}{|\delta|}I+\frac{\tilde \beta\zeta\zeta^T}{|\delta|^{3}}\ .
\end{align*}
Thus $\nabla^2 b(x)-\nabla^2 q(x)$ is positive semidefinite whenever $\tilde L - \tilde \beta/|\delta| \geq 0 $.
So local smoothness holds with any $\tilde \beta\geq \tilde L|\delta|$-smoothness holds, and so $L\delta$ local smoothness holds.

Similarly, the claimed $\mu\delta^3$-strong convexity follows by showing a local quadratic lower bound $q$ on $h$ yields as ball lower bound. The same calculations as done above hold, now needing $\nabla^2 b(x)\preceq \nabla^2 q(x)$ for
\begin{align*}
     q(y) &:= h(x) + g^T(y-x) + \frac{\tilde \mu}{2}\|y-x\|^2 \ , \\
     b(y) &:= h(x) - \delta/\tilde \alpha - \sqrt{1/\tilde \alpha^2 - \|y-(x-\zeta/\tilde \alpha)\|^2}\ .
\end{align*}
Similarly, these two Hessians can be computed as
\begin{align*}
    \nabla^2 q(x) &= \tilde \mu I \qquad \text{and} \qquad
    \nabla^2 b(x) = \frac{\tilde \alpha}{|\delta|}I+\frac{\tilde \alpha\zeta\zeta^T}{|\delta|^{3}}\ .
\end{align*}

The minimum eigenvalue of $\nabla^2 q(x)-\nabla^2 b(x)$ is $\tilde \mu - \tilde\alpha/|\delta|^3$ and so any $\tilde \alpha \leq  \tilde \mu|\delta|^3$-strong convexity of $\epi h$ w.r.t. $((x,h(x)),(\zeta,\delta))$ holds locally.

\subsubsection{Proof of Lemma~\ref{lemma: hessian formular}}    \label{pf: hessian formular}
Consider any $y\in \mathbb{E}$ with $\gamma_S(y)\neq 0$.
Let $\bar y = y/\gamma_S(y)$ be the corresponding point on the boundary of $S$. Note that $\zeta \in N_S(\bar y)$ if and only if the halfspace $\mathcal{H} = \{z \mid \zeta^T(z-\bar y) \leq 0\} \supseteq{S}$, or equivalently $\gamma_{S} \geq \gamma_{\mathcal{H}} $. Then it yields
\begin{align*}
    N_S(\bar y) &= \{\zeta \mid \gamma_S(z) \geq \gamma_{\mathcal {H}}(z), \forall z \} \\
    &=\{\zeta \mid \gamma_S(z) \geq \gamma_S(\bar y)+\frac{\zeta^T(z-\bar y)}{\zeta^T \bar y}, \forall z \} \\
    &=\{\zeta \mid \gamma_S(z) \geq \gamma_S(y)+\frac{\zeta^T(z-y)}{\zeta^T \bar y}, \forall z\} \\
    &=\{\zeta \mid \frac{\zeta}{\zeta^T \bar y} \in \partial \gamma_S(y)\}
\end{align*}
where the second equality is given by Example~\ref{subsec:ex}, $\gamma_{\mathcal{H}}(z)=\frac{\zeta^Tz}{\zeta^T \bar y} = \gamma_S(\bar y)+ \frac{\zeta^T(z-\bar y)}{\zeta^T \bar y}$, and the third equality is by rescaling.

Then by the chain rule and this gauge subgradient formula, 
$$ \frac{1}{2}\partial \gamma_S^2(y) = \left\{\gamma_S(y) \frac{\zeta}{\zeta^T \bar y} \mid \zeta\in N_S(\bar y)  \right\}.$$

Define $\bar \zeta = r\zeta$. Then the gauge of this ball is given by
    \begin{align*}
        \gamma _B(z)&=\inf \{\lambda>0 \mid \frac{z}{\lambda}  \in B\} \\
        &=\inf \{\lambda>0 \mid \|z-\lambda (\bar y-\bar \zeta)\|^2\leq \lambda ^2r^2\} \\
        &=\inf \{\lambda>0 \mid (2\bar \zeta^T \bar y -\|\bar y\|^2)\lambda ^2+2\lambda (\bar y -\bar \zeta)^Tz - \|z\|^2 \geq 0\} \\
        & = \left\{
    \begin{aligned}
    & \frac{\|z\|^2}{2(\bar y -\bar \zeta)^Tz}, & \|\bar y\|^2=2\zeta^T \bar y \\
    & \frac{-(\bar y -\bar \zeta)^Tz+\sqrt{((\bar y -\bar \zeta)^Tz)^2+(2\bar \zeta^T \bar y -\|\bar y\|^2)\|z\|^2}}{2\bar \zeta^T \bar y -\|\bar y\|^2}, & \|\bar y\|^2 \neq 2\zeta^T \bar y \ .
    \end{aligned}
    \right.
    \end{align*}
Note that these two cases are independent of the choice of $z$. We prove the lemma separately for each of these two cases.

\noindent {\bf Case 1:} $\|\bar y\|^2 = 2\zeta^T \bar y$. 
\begin{enumerate}[label=(\alph*)]
     \item Observe that at the point $z=y$,
\begin{align*}
\gamma_B(y) = \frac{\|y\|^2}{2( \bar y -\bar \zeta)^Ty} 
=\gamma_S(y)\frac{\|\bar y\|^2}{2(\|\bar y\|^2-\bar \zeta^T\bar y)} 
=\gamma_S(y) \ .
\end{align*}

\item To show $\nabla \frac{1}{2}\gamma_B^2(y) \in \partial \frac{1}{2}\gamma_S(y)^2$, it suffices to show $\nabla \gamma_B(y) \in \partial \gamma_S(y)$. We have
$$\nabla \gamma_B(z) = \frac{2(\bar y - \bar \zeta)^Tz z-\|z\|^2(\bar y - \bar \zeta)}{2\left((\bar y- \bar \zeta)^Tz\right)^2} \ .
$$
Then at the point $z=y$
\begin{align*}
    \nabla \gamma_B(y) = \frac{2(\bar y - \bar \zeta)^T  y y-\|y\|^2(\bar y - \bar \zeta))}{2\left((\bar y- \bar \zeta)^T y\right)^2} 
    =\frac{2(\bar y - \bar \zeta)^T\bar y \bar y-\|\bar y\|^2(\bar y - \bar \zeta))}{2\left((\bar y- \bar \zeta)^T \bar y\right)^2}  
    =\frac{\bar \zeta}{\bar \zeta^T \bar y} \in \partial \gamma_S(y) \ .
\end{align*}

\item We can then compute $\frac{1}{2}\gamma^2_B$, the gradient and Hessian at $z$ as
\begin{align*}
\begin{split}
\frac{1}{2}\gamma_B^2(z) &=\frac{\|z\|^4}{8\left((\bar y -\bar \zeta)^Tz\right)^2} \ ,\\
\nabla \frac{1}{2}\gamma_B^2(z) &= \frac{2(\bar y -\bar \zeta)^Tz\|z\|^2z-\|z\|^4(\bar y- \bar \zeta)}{4\left((\bar y-\bar \zeta)^Tz\right)^3} \ , \\
\nabla ^2 \frac{1}{2}\gamma_B^2(z) &= \frac{1}{4\left((\bar y-\bar \zeta)^Tz\right)^4}\left(4((\bar y-\bar \zeta)^Tz)^2zz^T-4\|z\|^2(\bar y-\bar \zeta)^Tz(z(\bar y-\bar \zeta)^T+(\bar y-\bar \zeta)z^T)\right.  \\ &\left.+3\|z\|^4(\bar y-\bar \zeta)(\bar y-\bar \zeta)^T+2((\bar y-\bar \zeta)^Tz)^2\|z\|^2 I\right) \ .
 \end{split}
\end{align*}

Then at the point $z=y$, 
$$\nabla ^2 \frac{1}{2}\gamma_B^2(y) = \frac{1}{(\bar \zeta^T \bar y)^2}\left(3\bar \zeta \bar \zeta^T - (\bar \zeta \bar y^T+\bar y \bar \zeta^T) +\bar \zeta \bar y I\right) \ .
$$
\end{enumerate}

\noindent {\bf Case 2:} $\|\bar y\|^2 \neq 2\zeta^T \bar y$. Note
\begin{equation} \label{eq:sum_square}
((\bar y -\bar \zeta)^Ty)^2+(2\bar \zeta^T \bar y -\|\bar y\|^2)\|y\|^2=(\bar \zeta ^T \bar y)^2 \ .
\end{equation}

\begin{enumerate}[label=(\alph*)]
\item At the point $z=y$,  
\begin{align*}
    \gamma_B(y)=\frac{-(\bar y -\bar \zeta)^Ty+\sqrt{((\bar y -\bar \zeta)^Ty)^2+(2\bar \zeta^T \bar y -\|\bar y\|^2)\|y\|^2}}{2\bar \zeta^T \bar y -\|\bar y\|^2} 
    =\gamma_S(y)\frac{\bar \zeta^T\bar y-\|\bar y\|^2+\bar \zeta^T\bar y}{2\bar \zeta^T \bar y -\|\bar y\|^2} 
    =\gamma _S(y) \ .
\end{align*}

\item Compute $\nabla \gamma_B(z)$ as
$$\nabla \gamma_B(z) = \frac{1}{2\bar \zeta ^T \bar y -\|\bar y\|^2}\left(-(\bar y- \bar \zeta)+\frac{(\bar y -\bar \zeta)^Tz(\bar y-\bar \zeta)+(2\bar \zeta^T\bar y -\|\bar y\|^2)z}{\sqrt{\left((\bar y -\bar \zeta)^Tz\right)^2+(2\bar \zeta^T \bar y -\|\bar y\|^2)\|z\|^2}}\right) \ .
$$
Then at the point $z=y$,
\begin{align*}
    \nabla \gamma_B(y) &= \frac{1}{2\bar \zeta ^T \bar y -\|\bar y\|^2}\left(-(\bar y- \bar \zeta)+\frac{(\bar y -\bar \zeta)^Ty(\bar y-\bar \zeta)+(2\bar \zeta^T\bar y -\|\bar y\|^2)y}{\sqrt{\left((\bar y -\bar \zeta)^Ty\right)^2+(2\bar \zeta^T \bar y -\|\bar y\|^2)\|y\|^2}}\right) \\
    &=\frac{1}{2\bar \zeta ^T \bar y -\|\bar y\|^2}\left(-(\bar y- \bar \zeta)+\frac{\bar \zeta ^T \bar y \bar y-(\|\bar y\|^2-\bar \zeta ^T \bar y)\bar \zeta}{\bar \zeta^T \bar y}\right) \\
    &=\frac{1}{2\bar \zeta ^T \bar y -\|\bar y\|^2}\left(\frac{(2\bar \zeta ^T \bar y -\|\bar y\|^2)\bar \zeta}{{ \bar \zeta ^T \bar y}}\right) \\
    &=\frac{\bar \zeta}{\bar \zeta^T \bar y} \in \partial \gamma_S(y) \ .
\end{align*}

\item We can then compute $\frac{1}{2}\gamma^2_B$, the gradient and Hessian as
$$\frac{1}{2}\gamma_B^2(z)=\frac{\left((\bar y -\bar \zeta\right)^Tz)^2-\left((\bar y -\bar \zeta\right)^Tz)\sqrt{\left((\bar y -\bar \zeta\right)^Tz)^2+(2\bar \zeta^T \bar y -\|\bar y\|^2)\|z\|^2}+ \frac{1}{2}(2\bar \zeta^T \bar y -\|\bar y\|^2)\|z\|^2}{(2\bar \zeta^T \bar y -\|\bar y\|^2)^2 } \ ,
$$
\begin{align*}
    \nabla \frac{1}{2}\gamma_B^2(z) &= \frac{1}{(2\bar \zeta^T \bar y -\|\bar y\|^2)^2 }\left(2(\bar y-\bar \zeta)^Tz(\bar y - \bar \zeta)+(2\bar \zeta^T \bar y-\|\bar y\|^2)z \right. \\
    & \quad \left.-\frac{\left(2\left((\bar y- \bar \zeta)^Tz\right)^2+(2\bar \zeta^T \bar y-\|\bar y\|^2)\|z\|^2\right)(\bar y-\bar \zeta)+(2\bar  \zeta^T\bar y - \|\bar y\|^2)(\bar y -\bar \zeta)^Tz z}{\sqrt{\left((\bar y -\bar \zeta)^Tz\right)^2+(2\bar \zeta^T \bar y -\|\bar y\|^2)\|z\|^2}}\right) \ ,
\end{align*}
\begin{align*}
    \nabla^2 \frac{1}{2}\gamma_B^2(z) &= \frac{1}{(2\bar \zeta^T \bar y -\|\bar y\|^2)^2 }\left((2-\frac{2\left((\bar y -\bar \zeta)^Tz\right)^3+3(2\bar \zeta^T\bar y- \|\bar y\|^2)\|z\|^2(\bar y -\bar \zeta)^Tz}{\left(\left((\bar y -\bar \zeta)^Tz\right)^2+(2\bar \zeta^T \bar y -\|\bar y\|^2)\|z\|^2\right)^\frac{3}{2}}\right)(\bar y-\bar \zeta)(\bar y - \bar \zeta)^T  \\
    & \quad -\frac{\|z\|^2\left((\bar y -\bar \zeta)z^T+z(\bar y-\bar \zeta)^T\right)}{\left(\left((\bar y -\bar \zeta)^Tz\right)^2+(2\bar \zeta^T \bar y -\|\bar y\|^2)\|z\|^2\right)^\frac{3}{2}}+\frac{(\bar y- \bar \zeta)^Tz zz^T}{\left(\left((\bar y -\bar \zeta)^Tz\right)^2+(2\bar \zeta^T \bar y -\|\bar y\|^2)\|z\|^2\right)^\frac{3}{2}} \\
    & \quad +\left(\frac{1}{2\bar \zeta^T \bar y -\|\bar y\|^2}-\frac{(\bar y- \bar \zeta)^Tz}{(2\bar \zeta^T \bar y -\|\bar y\|^2)\sqrt{\left((\bar y -\bar \zeta)^Tz\right)^2+(2\bar \zeta^T \bar y -\|\bar y\|^2)\|z\|^2}}\right)I \ .
\end{align*}
At point $z=y$, the numerator in the above Hessian formula can be notably simplified by~\eqref{eq:sum_square}. After such simplifications, the Hessian matrix at $y$ is given by 
\begin{align*}
    \frac{1}{2}\nabla^2 \gamma_B^2(y) = \frac{1}{\left(\bar \zeta^T \bar{y}\right)^3}\left(\left(\bar \zeta^T \bar{y}+\|\bar{y}\|^2\right)\bar\zeta\bar\zeta^T-\bar \zeta^T \bar{y}(\bar\zeta\bar{y}^T+\bar{y}\bar\zeta^T)+\left(\bar \zeta^T \bar{y}\right)^2I \right) \ .
\end{align*}
\end{enumerate}

\subsubsection{Proof of Lemma~\ref{lemma: hessian eigenvalues}} \label{pf: hessian eigenvalues}
To prove the lemma, we need the following spectral result: 
\begin{lemma}\label{lem:eigenvalues}
    Assume for any coefficients $C_1,C_2,C_3$, the matrix $$M=C_1 aa^T + C_2 (ab^T+ba^T) + C_3 bb^T$$
    has eigenvalues $\lambda_1, \lambda_2, \cdots, \lambda_n$ with $\lambda_1 \leq \lambda_2 \leq \cdots \leq \lambda_n$. Then the eigenvalues would be
\begin{align*}
    \lambda_1 &= \frac{1}{2}\left(C_1\|a\|^2+2C_2a^Tb+C_3\|b\|^2-\sqrt{(C_1\|a\|^2+2C_2a^Tb+C_3\|b\|^2)^2-4(C_1C_3-C_2^2)(\|a\|^2\|b\|^2-(a^Tb)^2)}\right)\\
    \lambda_i&=0, \qquad \text{for }i=2,...,n-1\\
     \lambda_n &= \frac{1}{2}\left(C_1\|a\|^2+2C_2a^Tb+C_3\|b\|^2+\sqrt{(C_1\|a\|^2+2C_2a^Tb+C_3\|b\|^2)^2-4(C_1C_3-C_2^2)(\|a\|^2\|b\|^2-(a^Tb)^2)}\right).
\end{align*}
\end{lemma}
\begin{proof}
    Note that if $a$ and $b$ are linearly independent, the matrix is at most a rank two matrix. Assume for the nonzero eigenvalue $\lambda$ the corresponding eigenvector $v=ra+sb$, then $Mv=\lambda v$. Then plug in $M$ and $v$, the equation is a polynomial respect to $a$ and $b$:
    $$\left(C_1r\|a\|^2+(C_2r+C_1s)a^Tb+C_2s\|b\|^2-\lambda r\right)a + \left(C_2r\|a\|^2+(C_3r+C_2s)a^Tb+C_3s\|b\|^2-
    \lambda s\right)b=0.
    $$
    Since $a$ and $b$ are independent, then the coefficients of $a$ and $b$ must be 0, 
    \begin{align*}
        C_1r\|a\|^2+(C_2r+C_1s)a^Tb+C_2s\|b\|^2-\lambda r=&0 \\
        C_2r\|a\|^2+(C_3r+C_2s)a^Tb+C_3s\|b\|^2-\lambda s=&0 \ .
    \end{align*}
    This is equivalent to 
    $$\begin{bmatrix} C_1\|a\|^2+C_2a^Tb-\lambda & C_1a^Tb +C_2\|b\|^2\\ C_2\|a\|^2+C_3a^Tb & C_2a^Tb+C_3\|b\|^2-\lambda \end{bmatrix} \begin{bmatrix} r \\ s \end{bmatrix} = 0 \ .
    $$
    Therefore, the two eigenvalues are the roots of
    $$\lambda ^2 -\left(C_1\|a\|^2+2C_2a^Tb+C_3\|b\|^2\right)\lambda +\left(C_1C_3-C_2^2\right)\left(\|a\|^2\|b\|^2-(a^Tb)^2\right)=0 \ .
    $$
    
    If $a$ and $b$ are linearly dependent, then the matrix is a rank one matrix. The only nonzero eigenvalue is $ C_1\|a\|^2+2C_2a^Tb+C_3\|b\|^2$, which matches the formula.
\end{proof}
Since our Hessian matrix $\frac{1}{2}\nabla ^2 \gamma_S^2(y)$ is the sum of matrix in Lemma~\ref{lemma: hessian formular} and identity matrix, then the eigenvalues are
\begin{align*}
    \lambda_1 =&\frac{\|\bar \zeta\|^2}{2(\bar \zeta^T \bar y)^3}\left(\bar \zeta^T \bar y+\|\bar y\|^2-\frac{2(\bar \zeta ^T \bar y)^2}{\|\bar \zeta\|^2}-\sqrt{(\bar \zeta ^T\bar y+\|\bar y\|^2)^2-4\frac{(\bar \zeta^T \bar y)^3}{\|\bar \zeta\|^2}}\right)+\frac{1}{\bar \zeta^T \bar y}  \\
    =&\frac{\|\bar \zeta\|^2}{2(\bar \zeta^T \bar y)^3}\left(\bar \zeta^T \bar y+\|\bar y\|^2-\sqrt{(\bar \zeta ^T\bar y+\|\bar y\|^2)^2-4\frac{(\bar \zeta^T \bar y)^3}{\|\bar \zeta\|^2}}\right) \ , \\
    \lambda_n =& \frac{\|\bar \zeta\|^2}{2(\bar \zeta^T \bar y)^3}\left(\bar \zeta^T \bar y+\|\bar y\|^2-\frac{2(\bar \zeta ^T \bar y)^2}{\|\bar \zeta\|^2}+\sqrt{(\bar \zeta ^T\bar y+\|\bar y\|^2)^2-4\frac{(\bar \zeta^T \bar y)^3}{\|\bar \zeta\|^2}}\right)+\frac{1}{\bar \zeta^T \bar y} \\
    =&\frac{\|\bar \zeta\|^2}{2(\bar \zeta^T \bar y)^3}\left(\bar \zeta^T \bar y+\|\bar y\|^2+\sqrt{(\bar \zeta ^T\bar y+\|\bar y\|^2)^2-4\frac{(\bar \zeta^T \bar y)^3}{\|\bar \zeta\|^2}}\right) \ , \\
    \lambda_i=& 0 +\frac{1}{\bar \zeta^T \bar y} =\frac{1}{\bar \zeta^T \bar y}, \qquad \text{for }i=2,...,n-1 \ ,
\end{align*}
where $\bar \zeta = r\zeta$.

\subsection{Deferred Proofs of Structured Finite Maximum Convergence Rates}\label{app:convergence-proofs}
Below, we present convergence theorems and proofs for the three methods of minimizing a nonnegative finite maximum~\eqref{eq:finite-minmax} introduced in Section~\ref{sec:algs}. Each of these algorithms can be viewed as minimizing the square of this problem. The near optimality of the squared problem relates to the near optimality of~\eqref{eq:finite-minmax} via
\begin{equation}\label{eq:Squared-Optimality-Conversion}
    h(y)-p_* = \frac{h^2(y)-p_*^2}{h(y)+p_*}\leq \frac{\frac{1}{2}h^2(y)-\frac{1}{2}p_*^2}{p_*} \ .
\end{equation}

\subsubsection{Subgradient Method Convergence Theory}
First, we consider the subgradient method~\eqref{eq:subgrad-method}.
Under appropriate choice of a stepsize sequence $s_k$, we use the non-Lipschitz convergence results of~\cite{grimmer2019} to show convergence at a rate of $O(1/\sqrt{T})$ for generic convex minimization and $O(1/T)$ for problems that are strongly convex when squared. To the best of our knowledge, such simple subgradient methods do not benefit from the smoothness of squared objective components.
\begin{theorem}\label{thm:subgradient-convergence}
    For any convex, nonnegative, $M$-Lipschitz functions $h_i$, the subgradient method with $s_k=\frac{D}{\|h(y_k)g_k\|\sqrt{T+1}}$ has
    $$ \min_{k\leq T}\{h(y_k) - p_*\} \leq \frac{MD}{\sqrt{T+1}} + \frac{M^2D^2}{2p_*(T+1)} $$
    provided some minimizer $y^*$ of~\eqref{eq:finite-minmax} has $\|y_0-y^*\|\leq D$.
    If additionally, each $\frac{1}{2}h_i^2$ is $\mu$-strongly convex, then selecting $s_k = \frac{2}{\mu(k+2)+\frac{4M^4}{\mu(k+1)}}$ has
    $$ \min_{k\leq T}\{h(y_k) - p_*\} \leq \frac{4M^2p_*}{\mu(T+2)} + \frac{4M^4D^2}{\mu p_*(T+1)(T+2)} \ .$$
\end{theorem} 
\begin{proof}
Notice $y_{k+1} = y_k - s_k h(y_k)g_k$ is the classic subgradient method applied to minimize $\frac{1}{2}h^2$. 
Since Lipschitz continuity ensures the upper bound $h(y) - p_*\leq M\|y-y^*\|$ for all $y$, the squared problem has upper bound
\begin{equation}
    \frac{1}{2}h^2(y) - \frac{1}{2}p^2_* = \frac{(h(y)-p_*)(h(y)+p_*)}{2} \leq \frac{2Mp_*\|y-y^*\| + M^2\|y-y^*\|^2}{2}= \mathcal{D}(\|y-y^*\|) \label{eq:general-upper-bound}
\end{equation}
where $\mathcal{D}(t) = Mp_*t + M^2t^2/2$. Then Theorem 1.2 of~\cite{grimmer2019} ensures that the subgradient method with normalized stepsizes $D/\sqrt{T+1}$ achieves convergence at a rate of
$$ \min_{k\leq T}\left\{\frac{1}{2}h^2(y_k) - \frac{1}{2}p^2_*\right\} \leq \frac{Mp_*D}{\sqrt{T+1}} + \frac{M^2D^2}{2(T+1)} \ . $$
From this,~\eqref{eq:Squared-Optimality-Conversion} gives the first claimed convergence rate.
Moreover, we can verify the following generalized subgradient norm bound for any $h_i(y)g_i\in\partial \frac{1}{2}h_i^2(y)$, using that $\|g_i\|\leq M$
\begin{equation}
    \|h_i(y)g_i\|^2 \leq M^2h_i^2(y) = M^2p_*^2 + 2M^2\left(\frac{1}{2}h^2_i(y) - \frac{1}{2}p_*^2\right). \label{eq:general-subgrad-bound}
\end{equation}
Since every subgradient of $h(y)$ is a convex combination of subgradients of the individual $h_i$, the same bound holds for each subgradient of $h$. Then in the strongly convex case, Theorem 1.7 of~\cite{grimmer2019} \footnote{The proof of Theorem 1.7 of~\cite{grimmer2019} gives a stronger statement which has been ultilized here.} ensures convergence  for the average of the iterates at a rate of
$$ \frac{2}{(T+1)(T+2)}\sum_{k\leq T}(k+1) \left(\frac{1}{2}h^2(y_k) - \frac{1}{2}p^2_*\right) \leq \frac{4M^2p_*^2}{\mu(T+2)} + \frac{4M^4D^2}{\mu(T+1)(T+2)} \ . $$
Lower bounding this average by the minimum objective gap and applying~\eqref{eq:Squared-Optimality-Conversion} gives the second claimed convergence rate.
\end{proof}

\subsubsection{Accelerated Generalized Gradient Method Convergence Theory}
We consider accelerated methods iteratively applying the ``generalized gradient step''~\eqref{eq:general-grad-step} following the development of Nesterov~\cite{nesterov2018lectures} (this can also be viewed as a prox linear step~\cite{dima2019} on the composition of the maximum function $\max\{t_1,\dots t_m\}$ with the smooth mapping $y\mapsto \frac{1}{2}(h^2_1(y),\dots,h^2_m(y))$). Computing this step can be formulated as a quadratic program of dimension $m$. When $m=1$, this is exactly a (sub)gradient step on $\frac{1}{2}h^2$ with stepsize $\alpha$. For $m=2$, this can be computed by optimality condition $\partial\left\{\max_{i=1,2}\left\{\frac{1}{2}h^2_i(y) + h_i(y)g_i^T(z-y)\right\} + \frac{1}{2\alpha}\|z-y\|^2\right\}=0$. It has the following three cases: 
\begin{enumerate}
    \item[(a)] If $\frac{1}{2}h^2_1(y) + h_1(y)g_1^T(z-y) > \frac{1}{2}h^2_2(y) + h_2(y)g_2^T(z-y)$, optimality condition gives $z=y-\alpha h_1(y)g_1$. \\ 
    \item[(b)] Similarly, if $\frac{1}{2}h^2_1(y) + h_1(y)g_1^T(z-y) < \frac{1}{2}h^2_2(y) + h_2(y)g_2^T(z-y)$, we obtain $z=y-\alpha h_2(y)g_2$. \\
    \item[(c)] If $\frac{1}{2}h^2_1(y) + h_1(y)g_1^T(z-y) = \frac{1}{2}h^2_2(y) + h_2(y)g_2^T(z-y)$, optimality condition gives $z=y-\alpha \left((1-\theta)h_1(y)g_1+\theta h_2(y)g_2\right)$, where $\theta =  \frac{\frac{1}{\alpha}(h_2^2(y)-h_1^2(y))+(h_1(y)g_1-h_2(y)g_2)^T(h_1(y)g_1)}{\|h_1(y)g_1-h_2(y)g_2\|_2^2}$ is calculated by substituting $z$ into the equality condition.
\end{enumerate}
  Then the closed form is given by
$$ \mathtt{gen\mbox{-}grad}(y,\alpha) = \begin{cases}
y^{(1)} & \text{ if } \frac{1}{2}h^2_1(y) + h_1(y)g_1^T(y^{(1)}-y) > \frac{1}{2}h^2_2(y) + h_2(y)g_2^T(y^{(1)}-y)\\
y^{(2)} & \text{ if } \frac{1}{2}h^2_1(y) + h_1(y)g_1^T(y^{(2)}-y) < \frac{1}{2}h^2_2(y) + h_2(y)g_2^T(y^{(2)}-y)\\
y^{(3)} & \text{ otherwise}
\end{cases}$$
where $$\begin{cases}y^{(1)} := y-\alpha h_1(y)g_1 \\ 
y^{(2)} := y-\alpha h_2(y)g_2 \\ 
y^{(3)}:=y-\alpha \left((1-\theta)h_1(y)g_1+\theta h_2(y)g_2\right)
\end{cases}$$ 
and $\theta =  \frac{\frac{1}{\alpha}(h_2^2(y)-h_1^2(y))+(h_1(y)g_1-h_2(y)g_2)^T(h_1(y)g_1)}{\|h_1(y)g_1-h_2(y)g_2\|_2^2}$. For $i=1,2$, $y^{(i)}$ is a subgradient step on $\frac{1}{2}h_i^2$.

The accelerated generalized gradient method~\eqref{eq:accel-method} attains the following stronger convergence guarantees whenever the squared components $\frac{1}{2}h_i^2$ are all $L$-smooth. (A parameter-free version of this iteration may be possible by generalizing the backtracking linesearch ideas of Nesterov's universal fast gradient method~\cite{Nesterov2015}. Such a universal method may attain both the optimal nonsmooth rates of the direct subgradient method and smooth rates of the accelerated generalized gradient method simultaneously.)
\begin{theorem}\label{thm:acclerated-gradient-convergence}
    For any convex, nonnegative, $M$-Lipschitz functions $h_i$ where $\frac{1}{2}h_i^2$ is $L$-smooth and $\mu$-strongly convex, the accelerated generalized gradient method~\eqref{eq:accel-method} has
    $$ h(y_T)-p_* \leq \frac{\frac{1}{2}h^2(y_0) - \frac{1}{2}p^2_* + \frac{\gamma_0}{2}\|y_0-y^*\|^2}{p_*}\min\left\{\left(1-\sqrt{\frac{\mu}{L}}\right)^k, \frac{4L}{(2\sqrt{L} +k\sqrt{\gamma_0})^2}\right\} $$
    where $y^*$is the minimizer of~\eqref{eq:finite-minmax} and $\gamma_0=\frac{t_0(t_0L-\mu)}{1-t_0}$.
\end{theorem}
\begin{proof}
By~\cite[Theorem 2.3.5]{nesterov2018lectures}, this scheme converges on the squared problem at a rate of
$$ \frac{1}{2}h^2(y_k) - \frac{1}{2}p^2_* \leq \left[\frac{1}{2}h^2(y_0) - \frac{1}{2}p^2_* + \frac{\gamma_0}{2}\|y_0-y^*\|^2\right] \min\left\{\left(1-\sqrt{\frac{\mu}{L}}\right)^k, \frac{4L}{(2\sqrt{L} +k\sqrt{\gamma_0})^2}\right\}\ .$$
Then~\eqref{eq:Squared-Optimality-Conversion} gives the claimed rate for the non-squared, original problem.
\end{proof}

\subsubsection{Level Projection Method Convergence Theory}
Finally, we consider the level-projection method~\eqref{eq:level-proj-method}.
When $m=1$,~\eqref{eq:level-proj-step} with $\bar h=p_*$ is exactly a (sub)gradient step with the Polyak stepsize rule. When $m=2$, this level projection step can still be directly computed in the following three cases:
\begin{enumerate}
    \item[(a)] If $\frac{1}{2}h^2_1(y) + h_1(y)g_{1}^T(z-y) > \frac{1}{2}h^2_2(y) + h_2(y)g_{2}^T(z-y)$, then we need to solve this projection problem
    \begin{equation*}
    \begin{cases}
        \min_z& \frac{1}{2}\|z-y\|^2\\
        \mathrm{s.t.} & \frac{1}{2}h^2_1(y) + h_1(y)g_{1}^T(z-y) = \frac{1}{2}\bar h^2 \ .
    \end{cases}
    \end{equation*}
    Optimality condition ensures $z-y=\lambda h_1(y)g_1$ for some $\lambda>0$. Then solving for $\lambda$ by substituting $z-y$ into the constraint, we obtain $\lambda= \frac{\frac{1}{2}h^2_1(y) - \frac{1}{2}\bar h^2}{\|h_1(y)g_{1}\|^2}$.
    \item[(b)] Similarly, if $\frac{1}{2}h^2_1(y) + h_1(y)g_{1}^T(z-y) < \frac{1}{2}h^2_2(y) + h_2(y)g_{2}^T(z-y)$, the same projection calculation for $i=2$ provides $\lambda= \frac{\frac{1}{2}h^2_2(y) - \frac{1}{2}\bar h^2}{\|h_2(y)g_{2}\|^2}$.
    \item[(c)] If $\frac{1}{2}h^2_1(y) + h_1(y)g_{1}^T(z-y) = \frac{1}{2}h^2_2(y) + h_2(y)g_{2}^T(z-y)$, then we need to calculate the following projection 
     \begin{align*}
    \begin{cases}
        \min_z& \frac{1}{2}\|z-y\|^2\\
        \mathrm{s.t.} & \frac{1}{2}h^2_1(y) + h_1(y)g_{1}^T(z-y) = \frac{1}{2}\bar h^2  \\
        &\frac{1}{2}h^2_2(y) + h_2(y)g_{2}^T(z-y) = \frac{1}{2}\bar h^2 \ . 
    \end{cases}
    \end{align*}
    Optimality condition ensures $z-y=\lambda_1 h_1(y)g_1+\lambda_2 h_2(y)g_2$ for some $\lambda_1, \lambda_2>0$. Then solving for $\lambda$ by substituting $z-y$ into the constraints, we obtain 
    $$ \begin{bmatrix} \lambda_1 \\ \lambda_2 \end{bmatrix}=\begin{bmatrix} \|h_1(y)g_{1}\|^2 &{ (h_1(y)g_{1})}^T(h_2(y)g_{2}) \\ (h_1(y){ g_{1}})^T (h_2(y)g_{2}) & \|h_2(y)g_{2}\|^2 \end{bmatrix}^{-1} \begin{bmatrix} \frac{1}{2}h^2_1(y)-\frac{1}{2}\bar{h}^2 \\ \frac{1}{2}h^2_2(y)-\frac{1}{2}\bar{h}^2 \end{bmatrix}. $$
\end{enumerate}
The closed form is given by:
$$ \mathtt{level\mbox{-}proj}(y,\bar h) = \begin{cases}
y^{(1)} & \text{ if } \frac{1}{2}h^2_1(y) + h_1(y)g_{1}^T(y^{(1)}-y) > \frac{1}{2}h^2_2(y) + h_2(y)g_{2}^T(y^{(1)}-y)\\
y^{(2)} & \text{ if } \frac{1}{2}h^2_1(y) + h_1(y)g_{1}^T(y^{(2)}-y) < \frac{1}{2}h^2_2(y) + h_2(y)g_{2}^T(y^{(2)}-y)\\
y^{(3)} & \text{ otherwise.}
\end{cases}$$
where $$\begin{cases}y^{(1)} := y - \frac{(\frac{1}{2}h^2_1(y) - \frac{1}{2}\bar h^2)h_1(y)g_{1}}{\|h_1(y)g_{1}\|^2} \\ 
y^{(2)} := y - \frac{(\frac{1}{2}h^2_2(y) - \frac{1}{2}\bar h^2)h_2(y)g_{2}}{\|h_2(y)g_{2}\|^2} \\ 
y^{(3)}:=y-(\lambda_1h_1(y)g_{1}+\lambda_2h_2(y)g_{2})
\end{cases}$$
and the coefficients $\lambda_1$ and $\lambda_2$ are defined as
$$ \begin{bmatrix} \lambda_1 \\ \lambda_2 \end{bmatrix}=\begin{bmatrix} \|h_1(y)g_{1}\|^2 &{ (h_1(y)g_{1})}^T(h_2(y)g_{2}) \\ (h_1(y){ g_{1}})^T (h_2(y)g_{2}) & \|h_2(y)g_{2}\|^2 \end{bmatrix}^{-1} \begin{bmatrix} \frac{1}{2}h^2_1(y)-\frac{1}{2}\bar{h}^2 \\ \frac{1}{2}h^2_2(y)-\frac{1}{2}\bar{h}^2 \end{bmatrix}. $$
Iterating this method for any $m$ and $\bar h \geq p_*$ will converge towards this target objective level as follows.
\begin{theorem}\label{thm:level-proj-convergence}
    For any convex, nonnegative, $M$-Lipschitz functions $h_i$ and $\bar h \geq p_*$, the level-set projection method $y_{k+1}=\mathtt{level\mbox{-}proj}(y_k,\bar h)$  has
    $$ \min_{k=0\dots T}\left\{h(y_k) - \bar h\right\} \leq \frac{MD}{\sqrt{T+1}} + \frac{2M^2D^2}{\bar h(T+1)} $$
    provided some $\bar y$ with $f(\bar y)\leq \bar h$ has $\|y_0-\bar y\|\leq D$.
    If additionally each $\frac{1}{2}h_i^2$ is $\mu$-strongly convex and $T+1 \geq \frac{32M^2}{\mu}\log(\mu D^2/\bar h^2)$, then
    $$ \min_{k\leq T}\{h(y_k) - \bar h\} \leq \frac{\sqrt{32}M^2\bar h}{\mu (T+1)} + \frac{64M^4\bar h}{\mu^2(T+1)^2} \ . $$
\end{theorem}
\begin{proof}
Let $\mathcal{\bar Y} = \{y \mid \frac{1}{2}h^2(y) \leq \frac{1}{2}\bar h^2\}\neq\emptyset$ denote the target level set, and $\mathcal{Y}_k = \{y \mid m_k(y)\leq \frac{1}{2}\bar h^2\}\supseteq \mathcal{\bar Y}$ denote our model's target level set (the containment follows from $m_k$ lower bounding $\frac{1}{2}h^2$). Each orthogonal projection step ensures $y_{k}-y_{k+1}$ is normal to $\mathcal{Y}_k$ at $y_{k+1}$. Consequently every $\bar y\in \mathcal{\bar Y}\subseteq\mathcal{Y}_k$ must have $(y_k - y_{k+1})^T(y_{k+1}-\bar y)\geq 0$. Then
\begin{align}
    \|y_{k+1}-\bar y\|^2 &= \|y_k - \bar y\|^2  -  2(y_{k} - y_{k+1})^T(y_{k+1}-\bar y) - \|y_{k+1}-y_k\|^2 \nonumber \\
    &\leq \|y_k - \bar y\|^2 - \|y_{k+1}-y_k\|^2 \nonumber \\
    &\leq \|y_k - \bar y\|^2 - \frac{\left(\frac{1}{2}h^2(y_k) - \frac{1}{2}\bar h^2\right)^2}{M^2_k} \nonumber  \\
    &= \|y_k - \bar y\|^2 - \frac{\left(\frac{1}{2}h^2(y_k) - \frac{1}{2}\bar h^2\right)^2}{M^2\bar h^2 + 2M^2\left(\frac{1}{2}h^2(y_k) - \frac{1}{2}\bar h^2\right)} \ . \label{eq:level-proj-induction}
\end{align}
where the last inequality follows as $m_k$ is $M_k$-Lipschitz (by~\eqref{eq:general-subgrad-bound}) with $m_k(y_k)-m_k(y_{k+1}) = \frac{1}{2}h^2(y_k) - \frac{1}{2}\bar h^2$, and so $\|y_{k+1}-y_k\| \geq (\frac{1}{2}h^2(y_k) - \frac{1}{2}\bar h^2)/M_k$. Inductively applying~\eqref{eq:level-proj-induction} yields
\begin{align*}
    & \min_{k=0\dots T}\left\{\frac{\left(\frac{1}{2}h^2(y_k) - \frac{1}{2}\bar h^2\right)^2}{M^2\bar h^2 + 2M^2\left(\frac{1}{2}h^2(y_k) - \frac{1}{2}\bar h^2\right)}\right\} \leq \frac{\|y_0-\bar y\|^2}{T+1} \\
    \implies & \min_{k=0\dots T}\left\{\frac{1}{2}h^2(y_k) - \frac{1}{2}\bar h^2\right\} \leq \frac{M\bar h\|y_0-\bar y\|}{\sqrt{T+1}} + \frac{2M^2\|y_0-\bar y\|^2}{T+1} \ . 
\end{align*}

Supposing $\mu$-strong convexity holds, observe that for $\bar y\in \mathcal{\bar Y}$ closest to $y_k$, $\frac{1}{2}h^2(y_k) - \frac{1}{2}\bar h^2 \geq \frac{\mu}{2}\|y_k-\bar y\|^2$. Then letting $D_k = \mathrm{dist}(y_k, \mathcal{\bar Y})^2$,~\eqref{eq:level-proj-induction} can then be relaxed to the recurrence relation
$$ D_{k+1} \leq D_k - \frac{\mu^2 D_k^2}{4(M^2\bar h^2 + \mu M^2D_k)} \ .$$
Note that this recurrence implies $D_k$ is monotonically decreasing. Let $K$ denote the first index with $D_k \leq \bar h^2/\mu$. All $k<K$ must have a geometric decrease of $D_{k+1}\leq (1-\mu/8M^2 )D_k$. Thus $K\leq (8M^2/\mu) \log(\mu D_0/\bar h^2)$. For all $k\geq K$, the recurrence ensures $D_{k+1}\leq D_k - \mu^2 D_k^2/8M^2\bar h^2$, which implies $D_k \leq 8M^2\bar h^2/\mu^2(k-K)$. \footnote{The recurrence $D_{k+1} \leq D_k - \alpha D_k^2$, is equivalent to $\frac{1}{D_k} \leq \frac{1}{D_{k+1}}-\alpha \frac{D_k}{D_{k+1}}$, then by $\frac{D_k}{D_{k+1}}\geq 1$ and $1/D_K>0$, we obtain $D_k$.}
By assumption, we have $T/2 > 2K$ and so $T/2-K+1 \geq T/4 +1 $. Hence recurrence ensures
$$ D_{T/2} \leq \frac{32 M^2\bar h^2}{\mu^2(T+1)} \ . $$
Then inductively applying~\eqref{eq:level-proj-induction} from $k=T/2$ to $T$ yields our claimed rate as
\begin{align*}
    & \min_{k=T/2\dots T}\left\{\frac{\left(\frac{1}{2}h^2(y_k) - \frac{1}{2}\bar h^2\right)^2}{M^2\bar h^2 + 2M^2\left(\frac{1}{2}h^2(y_k) - \frac{1}{2}\bar h^2\right)}\right\} \leq \frac{D_{T/2}}{T/2} \\
    \implies & \min_{k=0\dots T}\left\{\frac{1}{2}h^2(y_k) - \frac{1}{2}\bar h^2\right\} \leq \frac{M\bar h\sqrt{D_{T/2}}}{\sqrt{T+1}} + \frac{2M^2D_{T/2}}{T+1} \\
    \implies & \min_{k=0\dots T}\left\{\frac{1}{2}h^2(y_k) - \frac{1}{2}\bar h^2\right\} \leq \frac{\sqrt{32}M^2\bar h^2}{\mu (T+1)} + \frac{64M^4\bar h^2}{\mu^2(T+1)^2} \ . 
\end{align*}
\end{proof}
We expect that faster convergence rates could be proven for this method when each $\frac{1}{2}h_i^2$ is smooth (with or without strong convexity). Numerically, Section~\ref{sec:appli} saw the level projection method perform very well across a range of settings. The ideas of Lan~\cite{Lan2015} may provide a tractable path to establish such analysis and enable the analysis of accelerated level projection methods when objective functions square to become smooth.

\subsection{Numerics with \texttt{Convex.jl} instead of \texttt{JuMP.jl}} \label{append:jump}
Finally, we repeat our experiments from Section~\ref{subsec:numerics2} using \texttt{Convex.jl} as an alternative to \texttt{JuMP.jl} to verify that our numerical observations are not an artifact of JuMP's problem reformulations. We find similar performance with two exceptions: (i) Gurobi is able to run on the Convex reformulation of $p=4$-norm ellipsoids, so it is now included, (ii) Most methods are slightly slower in \texttt{Convex.jl} with COSMO being substantially slower. 
\begin{figure}[t]
\begin{minipage}[t]{0.3\linewidth}
        \captionsetup{justification=centering}
        \centering
	\includegraphics[width=\linewidth]{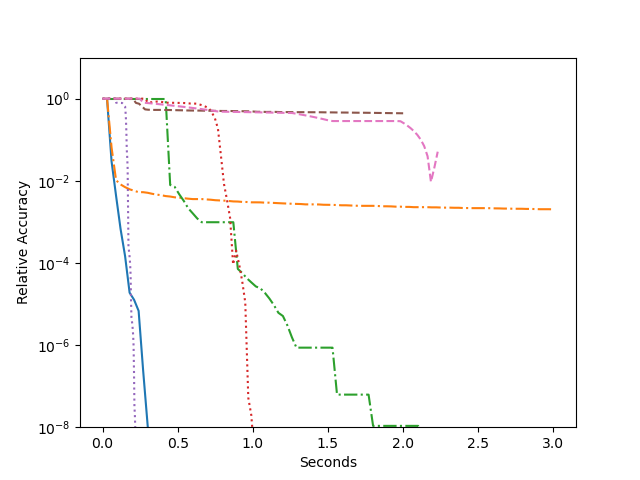}
	\caption*{(a) $(n,m)=(400,200)$}
\end{minipage}
\begin{minipage}[t]{0.3\linewidth}
        \captionsetup{justification=centering}
	\centering
	\includegraphics[width=\linewidth]{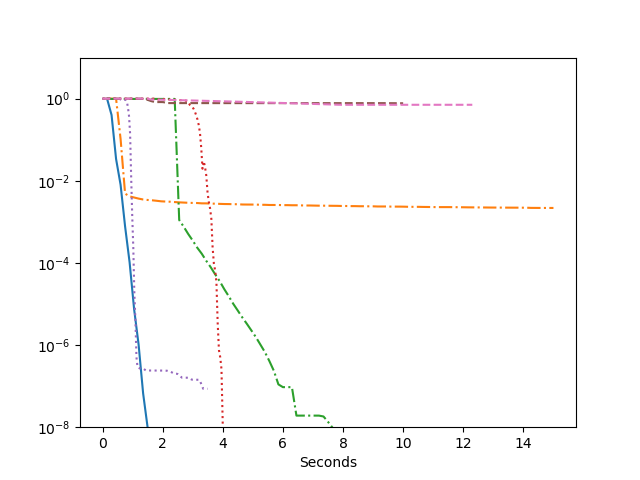}
        \caption*{(b) $(n,m)=(800,400)$}
\end{minipage}
\begin{minipage}[t]{0.36\linewidth}
        \captionsetup{justification=centering}
	\centering
	\includegraphics[width=\linewidth]{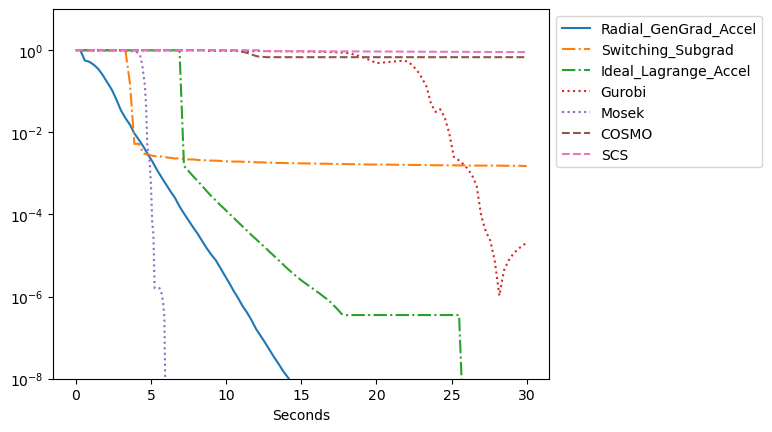}
        \caption*{(c) $(n,m)=(1600,800)$}
\end{minipage}
\caption{The minimum relative accuracy $|f(x_k)-f(x^*)|/|f(x_0)-f(x^*)|$ of~\eqref{eq:quad-SOCP}, mirroring Figure~\ref{fig:ellips_cons} but with \texttt{Convex.jl}.}
\label{fig:jellips_cons}
\end{figure}

\begin{figure}[t]
\begin{minipage}[t]{0.3\linewidth}
        \captionsetup{justification=centering}
        \centering
	\includegraphics[width=\linewidth]{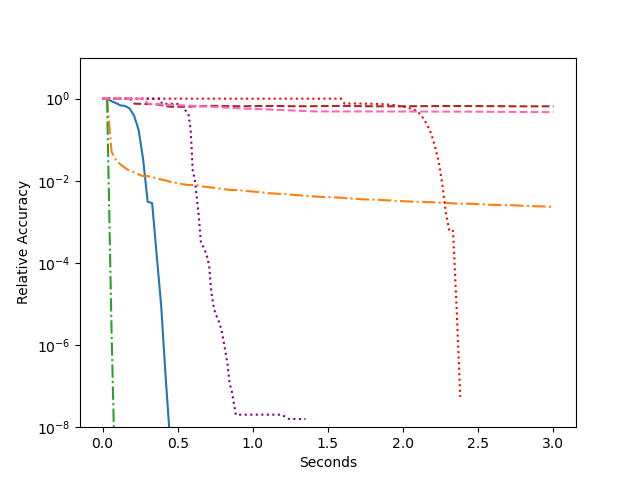}
	\caption*{(a) $(n,m)=(400,200)$}
\end{minipage}
\begin{minipage}[t]{0.3\linewidth}
        \captionsetup{justification=centering}
	\centering
	\includegraphics[width=\linewidth]{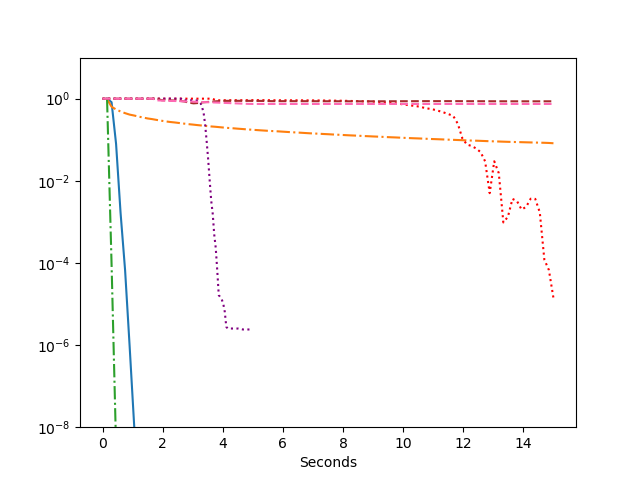}
        \caption*{(b) $(n,m)=(800,400)$}
\end{minipage}
\begin{minipage}[t]{0.36\linewidth}
        \captionsetup{justification=centering}
	\centering
	\includegraphics[width=\linewidth]{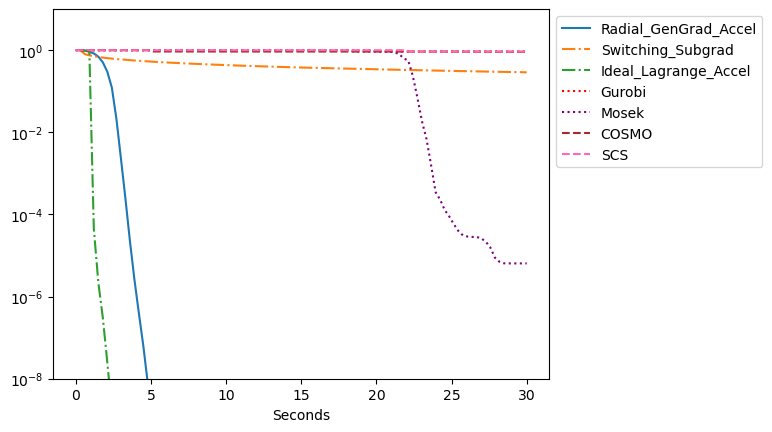}
        \caption*{(c) $(n,m)=(1600,800)$}
\end{minipage}
\caption{The minimum relative accuracy $|f(x_k)-f(x^*)|/|f(x_0)-f(x^*)|$ of~\eqref{eq:quart}, mirroring Figure~\ref{fig:quartic_cons} but with \texttt{Convex.jl}.}
\label{fig:jquartic_cons}
\end{figure}


}
\end{appendices}
\end{document}